\documentclass[twoside, 12pt]{article}

\usepackage{mathtools}
\usepackage{amsmath,amsfonts,amsthm,amssymb}
\usepackage{tikz-cd}
\usepackage{stmaryrd}
\usepackage{mathrsfs}
\usepackage{color}
\usepackage{stmaryrd} 
\usepackage{bbm}
\usepackage{verbatim}
\usepackage{enumerate}
\usepackage{quiver}
\usepackage{hyperref}

\usepackage[margin=1in, headheight=15pt]{geometry}

\newcommand{\bbone}{\text{\usefont{U}{bbold}{m}{n}1}}
\MakeRobust{\bbone}

\newtheorem{theorem}{Theorem}[subsection]
\newtheorem{lemma}[theorem]{Lemma}
\newtheorem{proposition}[theorem]{Proposition}
\newtheorem{corollary}[theorem]{Corollary}

\theoremstyle{remark}
\newtheorem{remark}[theorem]{Remark}

\newtheorem{definition}[theorem]{Definition}

\DeclareMathOperator{\Spec}{Spec}

\newcommand{\nc}{\newcommand} 
\nc{\on}{\operatorname}

\nc{\BC}{\mathbb{C}}
\nc{\C}{\mathbb{C}}

\nc{\BR}{\mathbb{R}}
\nc{\R}{\mathbb{R}}

\nc{\BZ}{\mathbb{Z}}
\nc{\Z}{\mathbb{Z}}

\nc{\BQ}{\mathbb{Q}}
\nc{\Q}{\mathbb{Q}}

\nc{\BF}{\mathbb{F}}

\renewcommand{\O}{\mathcal{O}}

\nc{\BP}{\mathbb{P}}

\nc{\BA}{\mathbb{A}}

\nc{\Hom}{\on{Hom}}
\nc{\Aut}{\on{Aut}}
\nc{\Gal}{\on{Gal}}
\nc{\End}{\on{End}}
\nc{\Ext}{\on{Ext}}
\nc{\Sht}{\on{Sht}}
\nc{\Bun}{\on{Bun}}
\nc{\CH}{\on{CH}}
\nc{\Vir}{\on{Vir}}
\nc{\Ord}{\on{Ord}}

\usepackage[OT2,T1]{fontenc}
\DeclareSymbolFont{cyrletters}{OT2}{wncyr}{m}{n}
\DeclareMathSymbol{\Sha}{\mathalpha}{cyrletters}{"58}

\nc{\GL}{\on{GL}}
\nc{\PSL}{\on{PSL}}
\nc{\SL}{\on{SL}}
\nc{\PGL}{\on{PGL}}
\nc{\SO}{\on{SO}}
\nc{\Sp}{\on{Sp}}
\nc{\Spin}{\on{Spin}}

\nc{\conj}{\overline}

\nc{\ad}{\on{ad}}
\nc{\V}{\check{\mathbf{V}}}
\nc{\Ind}{\on{Ind}}

\nc{\Ptriv}{\tilde{P}_{1^\vee}}
\nc{\Pw}{\tilde{P}_{\omega^\vee}}
\nc{\PSp}{\tilde{P}_{\on{Sp}^\vee}}
\nc{\PSpw}{\tilde{P}_{(\on{Sp} \otimes \omega)^\vee}}
\nc{\Mone}{\tilde{M}_{1^\vee}}
\nc{\Monenot}{\tilde{M}_{1^\vee, 0}}
\nc{\Ppi}{\tilde{P}_{\pi^\vee}}
\nc{\Ptau}{\tilde{P}_{\tau^\vee}}

\nc{\mideal}{\mathfrak{m}}
\nc{\aideal}{\mathfrak{m}}

\nc{\lb}{\llbracket}
\nc{\rb}{\rrbracket}

\usepackage{graphicx} % Required for inserting images
\date{}
\title{The $3$-adic Steinberg block and a categorical $3$-adic local Langlands correspondence for $\GL_2(\BQ_3)$}
\author{CJ Dowd}
\usepackage[
backend=biber,
style=alphabetic
]{biblatex}

\DeclareLabelalphaNameTemplate{
  \namepart[use=true, pre=true, strwidth=1, compound=true]{prefix}
  \namepart[compound=true]{family}
}

\usepackage{fancyhdr}
\begin{document}
\maketitle
\abstract{We extend the work of Johansson--Newton--Wang-Erickson on a categorical $p$-adic Langlands correspondence over $\GL_2(\BQ_p)$ to $p = 3$ by recontextualizing their construction in the Steinberg case. We produce a novel, more categorical method of analyzing the properties of this construction without requiring an explicit algebraic presentation. As input to this argument, we also perform some new analysis of the representation theory of the $3$-adic Steinberg block of $\GL_2(\BQ_3)$, obtaining results complementing those of Pa\v{s}k\={u}nas.}
\tableofcontents 

\section{Introduction}
\subsection{Context}

This paper aims to extend the main results of \cite{JNW}, which establishes a categorical $p$-adic local Langlands correspondence, to $p = 3$. We first discuss this result and the obstacles involved in the $p = 3$ case. 

Let $G = \GL_2(\BQ_p)$ and $\Gamma = \Gal(\conj{\BQ}/\BQ)$. Let $L/\BQ_p$ be a finite field extension, and let $\O$ be the ring of integers of $L$ with residue field $k$. Fix a smooth character $\zeta: \BQ_p^\times \to \O^\times$, and let $\varepsilon: \Gamma \to \BZ_p^\times$ be the cyclotomic character. We let $\mathfrak{X}_{\zeta \varepsilon}$ denote the algebraized moduli stack of $2$-dimensional continuous representations of $\Gamma$ over $\O$ with fixed determinant $\zeta \varepsilon$, and we let $\on{Mod}_{G, \zeta}^{\on{lfin}}(\O)$ denote the category of locally finite smooth representations of $G$ over $\O$ with central character $\zeta$; see \S\ref{sec:moduli} and \S\ref{sec:rep-theory} for precise definitions. The main result of \cite{JNW} is a formulation of a categorical $p$-adic local Langlands correspondence:

\begin{theorem}[{\cite[Theorem 1.1.1]{JNW}}] Assume $p \geq 5$. Then there exists a fully faithful embedding 
    \begin{align*}
        \mathcal{D}(\on{Mod}^{\on{lfin}}_{G, \zeta}(\O)) \to \on{IndCoh}(\mathfrak{X}_{\zeta \varepsilon}),
    \end{align*}
    where $\mathcal{D}(\on{Mod}^{\on{lfin}}_{G, \zeta}(\O))$ denotes the derived $\infty$-category of $\on{Mod}^{\on{lfin}}_{G, \zeta}(\O)$ and $\on{IndCoh}$ denotes the ind-completion of the bounded derived $\infty$-category of coherent sheaves on $\mathfrak{X}_{\zeta \varepsilon}$. 
\end{theorem}

We give a brief overview of the proof of this theorem. The category $\on{Mod}^{\on{lfin}}_{G, \zeta}(\O)$ has a direct sum decomposition $$\on{Mod}^{\on{lfin}}_{G, \zeta}(\O) = \bigoplus_{\mathfrak{B}} \on{Mod}^{\on{lfin}}_{G, \zeta}(\O)^{\mathfrak{B}},$$ where the $\mathfrak{B}$ are certain equivalence classes (``blocks'') of irreducible objects in $\on{Mod}^{\on{lfin}}_{G, \zeta}(\O)$ and $\on{Mod}^{\on{lfin}}_{G, \zeta}(\O)^{\mathfrak{B}}$ is the full subcategory of $\on{Mod}^{\on{lfin}}_{G, \zeta}(\O)^{\mathfrak{B}}$ consisting of representations whose irreducible subquotients all lie in $\mathfrak{B}$. Meanwhile, the stack $\mathfrak{X}_{\zeta \varepsilon}$ decomposes into a disjoint union 
$$\mathfrak{X}_{\zeta \varepsilon} = \coprod_{\conj{D}} \mathfrak{X}_{\conj{D}}$$
where the $\conj{D}$ range over residual pseudorepresentations of $\Gamma$ and $\mathfrak{X}_{\conj{D}}$ is the algebraized moduli stack of representations with residual pseudorepresentation $\conj{D}$. The mod $p$ local Langlands correspondence establishes a bijection between blocks $\mathfrak{B}$ and residual pseudorepresentations $\conj{D}$. Therefore, the theorem may be established by working block-by-block; that is, by proving that we have a fully faithful embedding $$\mathcal{D}(\on{Mod}^{\on{lfin}}_{G, \zeta}(\O)^{\mathfrak{B}}) \to \on{IndCoh}(\mathfrak{X}_{\mathfrak{B}})$$ for each block $\mathfrak{B}$, where $\mathfrak{X}_{\mathfrak{B}} = \mathfrak{X}_{\conj{D}}$ for the residual pseudorepresentation $\conj{D}$ corresponding to $\mathfrak{B}$.

Via Morita theory, each category $\on{Mod}^{\on{lfin}}_{G, \zeta}(\O)^{\mathfrak{B}}$ may be identified with the category $\on{LMod}_{\on{disc}}(\tilde{E}_{\mathfrak{B}})$ of discrete left $\tilde{E}_{\mathfrak{B}}$-modules for a certain compact algebra $\tilde{E}_{\mathfrak{B}}$. When $p \geq 5$, these algebras have been computed explicitly in \cite{paskunas13}. The functor in the theorem is given by a derived tensor product of the form 
\begin{align*}
    V \rightsquigarrow X^*_{\mathfrak{B}} \otimes^L_{\tilde{E}_{\mathfrak{B}}} V
\end{align*}
for $V \in \on{LMod}_{\on{disc}}(\tilde{E}_{\mathfrak{B}})$. Here, the object $X^*$ is a sheaf on $\mathfrak{X}_{\mathfrak{B}}$ satisfying $\on{REnd}_{\mathfrak{X}}(X^*_{\mathfrak{B}}) \simeq \tilde{E}_{\mathfrak{B}}^{\on{op}}$, which is a necessary and sufficient condition for the derived tensor functor to be fully faithful. 

One of the main achievements of \cite{JNW} is constructing this ``tensor kernel'' $X_{\mathfrak{B}}^*$ with this desired property for each block. For most blocks, there is a clear candidate for $X^*_{\mathfrak{B}}$: it is the universal vector bundle on $\mathfrak{X}_{\mathfrak{B}}$ in the residually generic or scalar cases, or a twisted structure sheaf in the residually supersingular cases. However, a novel construction is required in the so-called \emph{Steinberg} case; this is the block $\mathfrak{B}$ that includes the trivial representation and the special representation, and the corresponding residual pseudocharacter on the Galois side is $\omega \oplus 1$, where $\omega$ is the mod $p$ cyclotomic character. The algebra $\tilde{E}_{\mathfrak{B}}$ is larger and more complicated in the Steinberg case than in others. 

In the Steinberg case, the sheaf $X^*_{\mathfrak{B}}$ is constructed as the direct sum of the universal vector bundle with an explicit maximal Cohen-Macaulay sheaf. The methods used to define $X_{\mathfrak{B}}^*$ and prove its properties in \cite{JNW} depend on explicit presentations of various Galois (pseudo)deformations rings as well as the explicit presentation of $\tilde{E}_{\mathfrak{B}}$ from \cite{paskunas13}. In particular, the isomorphism $\on{\End}_{\mathfrak{X}}(X^*_\mathfrak{B}) \simeq \tilde{E}_{\mathfrak{B}}^{\on{op}}$ is established by giving explicit generators and relations for $\on{\End}_{\mathfrak{X}}(X^*_\mathfrak{B})$ as an algebra over the universal pseudodeformation ring $R$ and then matching generators and relations to the presentation of $\tilde{E}_{\mathfrak{B}}$ from \cite{paskunas13}. Such an approach is double-edged. The correspondence is made completely explicit---see, for example, \cite[Proposition 5.5.3]{JNW}---but in the words of \cite[\S 1.1]{JNW}, the work is computationally demanding. 

This paper aims to extend the main result of \cite{JNW} to $p = 3$. The reasons for the assumption $p \geq 5$ in \cite{JNW} include the following:
\begin{itemize}
    \item When $p = 2$, the mod $p$ cyclotomic character $\omega$ is trivial, and when $p = 3$ we have $\omega = \omega^{-1}$. These coincidences cause certain objects to coincide in a complicated way and also tend to introduce $H^2$ obstructions in the Galois deformation theory, making explicit presentations of the relevant deformation rings messy or impossible. 
    \item The results of \cite{paskunas13} assume $p \geq 5$. Much of this work remains valid for $p < 5$, but for the Steinberg block $p \geq 5$ is essential, as the structure of this block is completely different depending on whether $p = 2, p = 3$, or $p \geq 5$. The results of \cite{paskunas16} deal with the generic case for $p < 5$, and \cite{paskunas-tung} provides a uniform treatment covering all blocks and all primes, but the latter work does not explicitly compute the algebra $\tilde{E}_{\mathfrak{B}}$. 
\end{itemize}

As suspected in \cite[\S 1.3]{JNW}, the work \cite{paskunas16} allows results to be extended to arbitrary $p$ in the generic case. The argument in the scalar case (labeled ``non-generic case I'' in \cite{JNW}) also works for $p = 3$. Most of the arguments for the supersingular cases carry over for arbitrary primes, although for $p < 5$ there are cases where the universal framed deformation ring is singular, and infinite cohomological amplitude causes issues with the machinery of \cite[\S4.2]{JNW}. 

For the $p = 3$ Steinberg case, there are several obstacles to constructing the tensor kernel $X_{\mathfrak{B}}^*$ using the methods of \cite{JNW}. We lack the results of \cite[\S10]{paskunas13} on the structure of the algebra $\tilde{E}_{\mathfrak{B}}$ for $p = 3$, and moreover the complexity of the $p = 3$ Steinberg block seems to make writing down an explicit presentation difficult. Likewise, while the relevant deformation rings are known on the Galois side (see \S\ref{sec:explicit-S}), the presentations of these rings seem too cumbersome to reasonably repeat the kinds of explicit computations from \cite[\S3.5]{JNW}. 

To circumvent these issues, we instead sought a more categorical, inexplicit way to compare the two categories. The insight on how to construct the sheaf $X^*_{\mathfrak{B}}$ in a more natural way comes from \cite[\S3.6]{DEG26}, where their description $X^*_{\mathfrak{B}}$ (notated as $\mathfrak{W}_{\conj{\theta}}/\mathfrak{W}_{\conj{\theta}}[J]$ in loc.\ cit.) involves $\on{Mod}_{G, \zeta}^{\on{lfin}}(\O)$ directly. However, the relevant properties of $\mathfrak{W}_{\conj{\theta}}/\mathfrak{W}_{\conj{\theta}}[J]$ stated in \cite[Theorem 3.6.44]{DEG26} are not proven using this new description; rather, the authors exhibit an explicit isomorphism with the sheaf $X_{\mathfrak{B}}^*$ from \cite{JNW} and import the necessary results from the latter paper. The construction of $X^*_{\mathfrak{B}}$ in this paper, which we denote by $\mathfrak{W}_{\on{tf}}$ (see \S4), is a direct generalization of the corresponding sheaf in \cite{DEG26}. 

A key input for this construction is the comparison between categories coming from Colmez's so-called Montreal functor, as applied in \cite{paskunas-tung}. The Montreal functor, with normalizations as in loc.\ cit., is an exact contravariant functor from $\on{Mod}_{G, \zeta}^{\on{lfin}}(\O)$ to the category of linearly compact $\O\lb \Gamma \rb$-modules. It is ``almost'' fully faithful except for the fact that is annihilates characters, a defect that makes the Steinberg block more difficult to analyze via this functor. Nevertheless, \cite{paskunas-tung} succeed in showing that the image of the Montreal functor retains partial information about the algebra $\tilde{E}_{\mathfrak{B}}$, Namely, they identify the part of this algebra ``coming from non-characters'' with the universal Cayley Hamilton algebra on the Galois side, and moreover they show that the (Bernstein) center of $\on{Mod}_{G, \zeta}^{\on{lfin}}(\O)$ may be naturally identified with the universal pseudodeformation ring. See \S\ref{sec:paskunas-tung} for an overview of these results.

Knowing these facts beforehand grants us a considerable advantage: if Langlands is supposed to be a ``bridge,'' then the results of \cite{paskunas-tung} are the ``suspension towers'' of that ``bridge,'' serving as a foundation on which to build out the rest. Less metaphorically, we can construct $X^*_{\mathfrak{B}}$ by using the results of \cite{paskunas-tung} to realize the universal vector bundle $\mathfrak{V}$ on $\mathfrak{X}$ as a starting point, satisfying $\End_{\mathfrak{X}}(\mathfrak{V})^{\on{op}} \simeq \tilde{E}'_{\mathfrak{B}}$ for a subalgebra $\tilde{E}'_{\mathfrak{B}} \subset \tilde{E}_{\mathfrak{B}}$. We then append additional information to this starting point to recover the full endomorphism algebra $\tilde{E}_\mathfrak{B}$. This idea, spurred by the particular framing of $X^*_{\mathfrak{B}}$ presented in \cite{DEG26}, was the motivation to pursue the generalization of this construction in the present paper. 
\subsection{Main results}

In this paper, we generalize the construction of $X_{\mathfrak{B}}^*$ from \cite[\S3.6]{DEG26} and demonstrate via a categorical argument that its endomorphism algebra matches $\tilde{E}_{\mathfrak{B}}^{\on{op}}$ in the Steinberg and generic blocks for $p \geq 3$. This is achieved without explicitly computing either of these algebras. 

Set $p = 3$, and let $\mathfrak{B}$ denote the Steinberg block of $\on{Mod}_{G/Z}^{\on{lfin}}(\O)$, let $\tilde{P}$ be the direct sum of the projective envelopes of the four irreducible objects in $\mathfrak{C}(\O)^{\mathfrak{B}}$, and let $\tilde{E}_{\mathfrak{B}} = \End_{\mathfrak{C}(\O)}(\tilde{P})$. Let $\mathfrak{X}$ denote the algebraized moduli stack of $2$-dimensional continuous Galois representations with residual pseudocharacter $\omega \oplus 1$ and fixed determinant. Our main new result is:

\begin{theorem}\label{thm:intro-thm-1}
    There is a maximal Cohen-Macaulay sheaf $\mathfrak{W}_{\on{tf}}$ on $\mathfrak{X}$ with
    \begin{enumerate}[(a)]
        \item A natural isomorphism $\tilde{E}_{\mathfrak{B}}^{\on{op}} \simeq \End_{\mathfrak{X}}(\mathfrak{W}_{\on{tf}})$; and
        \item $\Ext^i_{\mathfrak{X}}(\mathfrak{W}_{\on{tf}}, \mathfrak{W}_{\on{tf}}) = 0$ for all $i > 0$.
    \end{enumerate} 
\end{theorem}
\begin{proof}
    The sheaf $\mathfrak{W}_{\on{tf}}$ is constructed in \S\ref{sec:candidate}; an alternative presentation of this sheaf is also given in Corollary \ref{cor:wtf-presentation}. The isomorphism of endomorphism algebras is Theorem \ref{thm:endomorphism-algebra-isomorphism}, and the vanishing of the higher Ext groups is Theorem \ref{thm:ext-vanishing}. The fact that $\mathfrak{W}_{\on{tf}}$ is maximal Cohen-Macaulay is Corollary \ref{cor:wtf-is-mcm}. 
\end{proof} 
The proof of the endomorphism algebra isomorphism depends on being able to identify endomorphisms of $\mathfrak{W}_{\on{tf}}$ with a certain collection of commutative diagrams resembling the naturality diagrams obtained from the Yoneda embedding; see Corollary \ref{cor:compatible-diagram-hom}. These are diagrams of $R$-modules for the universal pseudodeformation ring $R$. To simplify the picture and more easily identify morphisms, it is often helpful to pass to the generic fiber $K := \on{Frac}(R)$. To this end, one of the main ingredients is showing that $\tilde{E}_{\mathfrak{B}} \otimes_R K = M_{4 \times 4}(K)$, which we deduce from a new analysis of representation theory of the $p = 3$ Steinberg block. To show that these diagrams are induced by an element of $\tilde{E}_{\mathfrak{B}}$, we give an argument using the fact that the depths of the $R$-modules involved are sufficiently large. This approach is flexible enough to re-prove the endomorphism algebra isomorphism for all cases where the residual pseudorepresentation is reducible and multiplicity-free, i.e. the generic and $p > 2$ Steinberg cases. The key inputs are the aforementioned $K$-rational structure of $\tilde{E}_{\mathfrak{B}}$ and its $R$-depth properties, both of which can be deduced without explicitly computing $\tilde{E}_{\mathfrak{B}}$. 

\begin{remark}
    Our proof methods also apply to compute the endomorphism algebra of the Steinberg block in the $p \geq 5$ setting as well, reframing and reproving the results of \cite{JNW} in a more canonical way. Definition \ref{def:candidate-sheaf} of the candidate sheaf $\mathfrak{W}_{\on{tf}}$ naturally extends to this case, and the results of \S4.2 all have natural modifications for the $p \geq 5$ case. The abstract algebraic inputs to the proof of our key result Proposition \ref{prop:hom-iso} are already known in the $p \geq 5$ via the explicit analysis of \cite{JNW} and \cite{paskunas13}. These key properties include:
    \begin{itemize}
        \item The modules $\Hom(\Ppi, \Ptau)$ (see \S\ref{sec:rep-theory} for notation) are rank $1$ as $R$-modules and moreover we have an embedding of the algebra $\tilde{E}_{\mathfrak{B}}$ into a matrix algebra over $K := \on{Frac}(R)$. For $p = 3$, this is the main result of \S\ref{sec:matrix-embedding}, and for $p \geq 5$ these results are implied by, e.g., \cite[Lemma 3.6.18]{DEG26}.
        \item The modules $\Hom(\Ppi, \Ptau)$ have depth at least $3$ as an $R$-module when $\tau$ is a character. For $p = 3$, this is the main result of \S\ref{sec:depths}, and for $p \geq 5$ this can be verified using the explicit presentation of $R$ in \cite[Theorem 3.4.1]{JNW}, which implies that $R$ is lci of dimension $4$, and using the explicit presentation of the modules $\Hom(\Ppi, \Ptau)$ in \cite[Theorem 3.5.5]{JNW} (using also the results of \cite[\S5.5]{JNW} comparing to \cite{paskunas13} for this identification).
        \item For infinite dimensional $\pi$, there exist certain morphisms $\Ppi \to \Ptriv$ with jointly large image, in the sense of Proposition \ref{prop:alpha-beta-image} for $p = 3$ and  \cite[Lemma 10.75]{paskunas13} for $p \geq 5$. This is related to the second layer of the Loewy filtation of $\Ptriv$ and the existence of nontrivial extensions between the trivial character and such $\pi$. 
    \end{itemize}
    We expect, but have not checked carefully, that all of these algebraic properties in the $p \geq 5$ case may be readily deduced via the methods of \S3 (or by a subset of the results in \cite[\S10]{paskunas13}) without requiring an explicit presentation of $\tilde{E}_{\mathfrak{B}}$. 
\end{remark}

Theorem \ref{thm:intro-thm-1} allows us to define a categorical $p$-adic local Langlands correspondence that includes $p = 3$ in the manner of \cite[Theorem 1.1]{JNW}. Fix a block $\mathfrak{B}$. For an abelian category $\mathcal{C}$, let $\mathcal{D}(\mathcal{C})$ denote the derived $\infty$-category of $\mathcal{C}$, and if $\mathcal{X}$ is an algebraic stack, let $\on{IndCoh}(\mathcal{X})$ denote the ind-completion of the derived category of coherent sheaves on $\mathcal{X}$. See the discussion in \cite[\S2.4]{JNW} for remarks on working over $\infty$-categories; as in loc.\ cit., our usage of $\infty$-categories is mainly to ensure that ind-completions behave correctly, and in practice it suffices to argue over triangulated categories. The usage of $\infty$-categories is more serious in \cite{DEG26}, since their argument involves gluing derived sheaves over the Emerton-Gee stack. 
\begin{theorem}[Categorical $p$-LLC for $\GL_2(\BQ_p)$, $p \geq 3$]
    Assume $p \geq 3$. Then for each block $\mathfrak{B}$ of $\on{Mod}_{G, \zeta}^{\on{lfin}}(\O)$, there is a fully faithful embedding 
    \begin{align*}
        \mathcal{D}(\on{Mod}_{G, \zeta}^{\on{lfin}}(\O)^{\mathfrak{B}}) \to \on{IndCoh}(\mathfrak{X}_{\mathfrak{B}}). 
    \end{align*}
    Via the block decomposition, we therefore obtain a fully faithful embedding 
    \begin{align*}
        \mathcal{D}(\on{Mod}_{G, \zeta}^{\on{lfin}}(\O)) \to \on{IndCoh}(\mathfrak{X}_{\zeta \varepsilon}). 
    \end{align*}
\end{theorem}
\begin{proof}
   Working with trivial central character, the $p = 3$ Steinberg case is Corollary \ref{cor:categorical-p-LLC}, and twisting proves the result for an arbitrary fixed central character. The other $p = 3$ cases follow from the discussion in \S5.3, and the $p \geq 5$ cases are \cite[Theorem 1.1]{JNW}. 
\end{proof}

%A shortcoming of this formulation is that, by Proposition \ref{prop:infinite-cohom-dim}, the category $\on{Perf}(\tilde{E}_{\mathfrak{B}})$ does not include all objects corresponding to the irreducible representations. To get a more complete formulation in the vein of \cite[Theorem 1.1]{JNW}, we would require a more detailed analysis of the image of finite-length (or, by d\'evissage, just irreducible) representations under the tensor functor, which is a direction for future work. 

\subsection{Relation to the work of Dotto-Emerton-Gee}

We mainly follow the program set out in \cite{JNW}. The paper \cite{DEG26} meanwhile establishes a different categorical $p$-adic local Langlands correspondence: rather than having the target category lie over an (algebraized) moduli stack of Galois representations $\mathfrak{X}$, the target category is the Emerton-Gee stack $\on{EG}_{\BQ_p, 2}$; see \cite[Theorem 1.1.1]{DEG26}, where the Emerton-Gee stack is denoted as $\mathcal{X}$. One advantage of using the Emerton-Gee stack is that it interpolates $p$-adic Galois representation into a connected geometric object, whereas the moduli stack of Galois deformations decomposes into components based on the residual pseudorepresentation as discussed previously. The two formulations are related by \cite[Proposition 5.2.18]{DEG26}, which recovers the functor constructed in \cite{DEG26} as a ``pointwise'' version of the one in \cite{JNW}. 

Conversely, the properties of the embedding from \cite{DEG26} ``lying over the Steinberg locus'' are deduced directly from the results of \cite{JNW}. The main argument in \cite{DEG26} decomposes $\on{EG}_{\BQ_p, 2}$ into a ``good'' dense open locus and a ``bad'' finite locus, constructs the fully faithful embedding for these loci independently, and then glues; see \cite[Theorem 1.1.8]{DEG26}. The proof for the bad locus refers to \cite{JNW} to deal with the points of Steinberg type; see \cite[Theorem 5.2.24]{DEG26}, which depends on Theorem 3.6.44 of loc.\ cit., which imports the results of \cite[\S3.5]{JNW}. It therefore seems reasonable to expect that a generalization of \cite{DEG26} to $p = 3$ would require the corresponding results of \cite{JNW} for the $p = 3$ Steinberg block, and that is what the current paper achieves. 

There is more to be done to fully extend the results of \cite{DEG26} to $p = 3$. In particular, the geometry of $\on{EG}_{\BQ_3, 2}$ fundamentally changes due to the nonexistence of very generic Serre weights for $p = 3$; see \cite[Remark 7.2.5]{EGH}. In future work, we hope to analyze the geometry of $\on{EG}_{\BQ_3, 2}$ and perform the necessary modifications to fully complete the program of \cite{DEG26} and realize their version of a categorical $3$-adic local Langlands correspondence. 

\subsection{Outline}

In \S2, we begin by defining the relevant Galois deformation moduli problems and giving presentations of certain representing rings in the $p = 3$ Steinberg case. Although these presentations seem too complicated to write out explicitly in full, the general shape of these rings will be helpful, and we also list some algebraic properties of these rings. In \S3, we analyze the category $\on{Mod}_{G/Z}^{\on{lfin}}(\O)^{\mathfrak{B}}$ for the $p = 3$ Steinberg block $\mathfrak{B}$. Since this analysis was omitted in \cite{paskunas13}, we obtain some new results, but not many new ideas, since we largely adhere to the blueprint established in \S10 of loc.\ cit.\ We do not aim for an explicit description of the block, but we deduce enough abstract properties about the structure of $\tilde{E}_{\mathfrak{B}}$ for the proof of our main result. In \S4.1, we define sheaves $\mathfrak{W}$ and $\mathfrak{W}_{\on{tf}}$ generalizing the construction in \cite[\S3.6]{DEG26}. In \S4.2, which contains our main novelty, we prove the isomorphism of endomorphism algebras stated in our main theorem, and in \S4.3 we separately prove the vanishing of the higher self-Ext groups. We conclude in \S5 by discussing the implications of our main result for a fully faithful embedding of derived categories in the vein of the main results of \cite{JNW}, i.e.\ a categorical $3$-adic local Langlands correspondence, also addressing some minor difficulties introduced by infinite projective dimension within some of the $p = 3$ blocks.

\subsection{Notation and conventions}

Throughout, we fix a prime $p$ and a field extension $L/\BQ_p$; we specialize to $p = 3$ as needed on a subsection-by-subsection basis. We let $\O$ be the ring of integers of $L$, with uniformizer $\varpi$ and residue field $k$. The main groups we work with are $\Gamma := \Gal(\conj{\BQ}_p/\BQ_p)$ and $G := \GL_2(\BQ_p)$. See \S\ref{sec:moduli} for notation and definitions we use for Galois deformation theory and \S\ref{sec:rep-theory} for notation and definitions we use for $p$-adic representation theory. We write $\omega: \Gamma \to k^\times$ for the mod $p$ cyclotomic character, and we write $\varepsilon: \Gamma \to \O^\times$ for the $p$-adic cyclotomic character. Following the conventions of \cite{JNW} and \cite{paskunas13}, we normalize local class field theory so that uniformizers correspond to geometric Frobenii, and we conflate continuous (resp. smooth) characters of $\Gamma$ with continuous (resp. smooth) characters of $\BQ_p^\times$ using local class field theory.

Our notation and conventions surrounding graded modules and sheaves on stacks follow those established in the notation section of \cite{JNW}. Let $S$ be a $\BZ$-graded commutative ring, and let $M, N$ be graded $S$-modules. We write $M_d$ for the $d$-th graded component of $M$, and we write $M(i)$ for the graded $S$-module with underlying ungraded $S$-module $M$ and grading $M(i)_d := M_{i + d}$. The category of quasicoherent sheaves on the stack $\mathfrak{X} := [\on{Spec} S/\mathbb{G}_m]$ is equivalent to the category of $\mathbb{G}_m$-equivariant graded $S$-modules, so that the morphisms are graded $S$-module homomorphisms, and this equivalence restricts to the subcategories of coherent sheaves and coherent graded $S$-modules. A morphism of graded $S$-modules $M \to N$ is an $S$-module homomorphism sending $M_d \to N_d$ for all $i \in \BZ$. The equivalences just described justify writing $\Hom_{\mathfrak{X}}(M, N)$ for graded $S$-module homomorphisms from $M$ to $N$. 

\subsection{Acknowledgments}

I would like to thank Carl Wang-Erickson for introducing me to \cite{JNW} and, more generally, to the categorical $p$-adic local Langlands program; for encouragement to pursue the $3$-adic case; and for several helpful discussions about the problem. In turn, I thank Alex Youcis for facilitating my introduction to Carl Wang-Erickson and for encouraging me to take up a new problem to complete my PhD. I thank Matthew Emerton, Toby Gee, Christian Johansson, James Newton, and Vytautas Pa\v{s}k\=unas for their correspondence, helpful discussions, and early feedback about this paper. Finally, I thank my advisor, Sug Woo Shin, for mathematical, financial, and moral support during my PhD. 

Generative AI was used to sanity check individual claims, search the literature, and consult for (non)commutative algebra facts, using the free online version of GPT-5.6 Luna. This consultation was primarily concerned with whether there was a simple proof of the depth result in Proposition \ref{prop:alg-properties}(a), which was ultimately unsuccessful. The suggestion to use Nagata's lemma in the proof of Lemma \ref{lem:s'-factorial} is due to AI consultation. No part of this document was prepared using AI tools. 

I was supported during my PhD at Berkeley at several points by an NSF RTG grant. I was also supported by departmental grants allocated to my advisor for travel support.

\section{Galois deformation theory}
\subsection{Moduli problems}\label{sec:moduli}

We recount the moduli problems that form the setting on the Galois side. This subsection mostly recapitulates \cite[\S2]{JNW}; see this reference or \cite[\S 3]{CWE18} for an exposition in greater generality and detail. See also \cite{chenevier-pseudoreps} for the definitions and basic properties of pseudorepresentations, which we consider as multiplicative polynomial (determinant) laws. For ease of exposition, we assume $p > 2$ so that we may identify pseudorepresentations with their underlying pseudocharacter by \cite[Proposition 1.29]{chenevier-pseudoreps}. 

Let $\Gamma$ be a profinite group satisfying Mazur's finiteness condition $\Phi_p$, and let $k$. Fix a $2$-dimensional residual pseudorepresentation $\conj{D}$ of $\Gamma$, and fix a character $\psi : \Gamma \to \O^\times$. Since we exclusively work with representations of fixed determinant, we shall omit $\psi$ from most of our notation. 

\begin{definition}
    We write $\on{PsR}_{\conj{D}}$ for the moduli functor sending a topologically finite-type formal $\O$-algebra $B$ to the groupoid of $2$-dimensional pseudorepresentation $D: \Gamma \to B$ with determinant $\psi$. By \cite[Proposition E]{chenevier-pseudoreps}, the functor $\on{PsR}^{\psi}$ is representable by a formal affine scheme. We denote the representing formal scheme by $\on{Spf} R$ with universal pseudodeformation $D^{\on{univ}}$, and we call $R$ the \emph{universal pseudodeformation ring}.
\end{definition}

\begin{definition}
    Write $E$ for the \emph{universal Cayley-Hamilton algebra} over $\conj{D}$ with fixed determinant $\psi$, which is
    \begin{align*}
        E := \frac{R\lb\Gamma\rb}{\on{CH}({D}^{\on{univ}})}
    \end{align*}
    where the Cayley-Hamilton ideal $\on{CH}({D}^{\on{univ}})$ denotes the minimal two-sided ideal generated by $g^2 - D^{\on{univ}}(g) + \psi(g)$ ranging over $g\in \Gamma$. We let $\rho^{\on{univ}}: \Gamma \to E$ denote the natural map and we let $D: E \to R$ denote the pseudorepresentation attached to $E$ (uniquely determined by the image of $\Gamma$ under $D^{\on{univ}}$).  
\end{definition}

The morphism $D: E \to R$ is universal among Cayley-Hamilton representations deforming $\conj{D}$; see \cite[Definition 2.1.3]{JNW} for the precise formulation of this universal property. 

\begin{definition}
    Let $C$ be an $R$-algebra. A $C$-valued \emph{compatible representation} of $E$ is a homomorphism of $R$-algebras $E \to M_{2}(C)$ such that the following diagram commutes:
    \begin{center}
        \begin{tikzcd}
            E \ar[r]\ar[d, "D"] & M_2(C)\ar[d, "\det"] \\
            R \ar[r] & C
        \end{tikzcd}
    \end{center}
    We denote by $\on{Rep}^{\square}(E)$ the $\Spec R$-functor of compatible representations of $E$, we denote by $\on{Rep}(E)$ the $\Spec B$-groupoid sending an $R$-algebra $C$ to the groupoid of projective rank $2$ $C$-modules $V$ equipped with an isomorphism $\bigwedge^2 V \simeq C$ and a compatible representation of $E$ on $V$, i.e. a map $E \to \End_C(V)$ making the same diagram as above commute with $\End_C(V)$ in place of $M_2(C)$. 
\end{definition}

By \cite[Proposition 3.6]{CWE18}, $E$ is finitely generated as an $R$-module, from which we obtain:
\begin{proposition}[\cite{JNW}, Proposition 2.2.3] $\on{Rep}^{\square}(E)$ is an affine $R$-scheme of finite type, which we denote by $\Spec A$, and $\on{Rep}(E)$ is a $\Spec B$-algebraic stack of finite type given by $[\Spec B/\SL_2]$ under the adjoint action of $\SL_2$ on $\on{Rep}^{\square}(E)$. 
\end{proposition}

We now assume that $\conj{D}$ is multiplicity-free and reducible, arising from a direct sum of distinct characters $\chi_1 \oplus \chi_2: \Gamma \to \GL_2(k)$. Fix a choice of lifts of the two standard idempotents under the natural map $E \to M_{2 \times 2}(k)$, which induces a $2\times 2$  generalized matrix algebra (GMA) structure on $E$; see \cite[Theorem 2.22]{chenevier-pseudoreps}. 

\begin{definition}
    An \emph{adapted representation} of the GMA $E$ valued in an $R$-algebra $C$ is an $R$-algebra homomorphism $E \to M_{2 \times 2}(C)$ preserving the GMA structure, i.e. sending the chosen idempotents in $E$ to the standard idempotents in $M_{2
    \times 2}(C)$.

    We denote by $\on{Rep}^{\square, \on{Ad}}(E)$ the $\Spec R$-functor of adapted representations of $E$, and we denote by $\on{Rep}^{\on{Ad}}$ the $\Spec B$-groupoid sending $C$ to ordered pairs of rank $1$ projective $C$-modules $(V_1, V_2)$ equipped with an isomorphism $V_1 \otimes V_2 \simeq C$ and a homomorphism of $R$-GMAs $E \to \End(V_1 \oplus V_2)$. 
\end{definition}

The standard diagonal torus $T \subset \SL_2$ acts by its adjoint action on $\on{Rep}^{\square, \on{Ad}}(E)$, inducing an isomorphism $\on{Rep}^{\on{Ad}}(E) \simeq \on{Rep}^{\square, \on{Ad}}(E)$. The $R$-GMA structure on $E$ induces coordinates
\begin{align*}
 E \simeq \begin{bmatrix}
     E_{1, 1} & E_{1, 2} \\
     E_{2, 1} & E_{2, 2}
 \end{bmatrix}.
\end{align*}
We fix an isomorphism $T \simeq \mathbb{G}_m$ such that $\mathbb{G}_m$ acts on the upper-right corner $E_{1, 2}$ with weight $2$ and the lower right $E_{2, 1}$ with weight $-2$. 

\begin{theorem}[\cite{JNW} Theorem 2.3.2(1)-(2)]
    Adapted representations of $E$ are compatible representations, and the resulting map $\on{Rep}^{\square, \on{Ad}}(E) \to \on{Rep}^{\square}$ is a closed immersion of affine $R$-schemes.  
\end{theorem}

\begin{definition}
    We denote by $S$ the ring representing $\on{Rep}^{\square, \on{Ad}}(E)$ and $A$ the ring representing $\on{Rep}^{\square}(E)$.
\end{definition}

The $\mathbb{G}_m$-action on $\on{Rep}^{\square, \on{Ad}}(E)$ induces a $\mathbb{G}_m$-action on $S$ and hence the structure of a $\BZ$-graded ring. We write $S_{i}$ for the $i$-th graded piece of $S$ for $i \in \BZ$. 

\begin{theorem}[\cite{JNW}, Theorem 2.3.2(4)-(6)]
    We have natural isomorphisms $R \simeq A^{\on{SL}_2} \simeq S^{\mathbb{G}_m} = S_0$, and $A$ and $S$ are finite type over $R$. 
\end{theorem}

The $\BZ$-graded structure on $S$ provides a convenient presentation of $E$:

\begin{proposition}
    [\cite{BC09}, Proposition 1.3.13, Remark 1.3.15] Writing $E_{i, j}$ for the $R$-GMA coordinates of $E$ there are canonical isomorphisms of $R$-modules $E_{i,j} \simeq S_{2(j - i)}$ such that the coordinate-wise multiplication maps $E_{i, j} \otimes_R E_{j, k} \to E_{i, k}$ are compatible with the multiplication law of $S$, and $S$ is generated as an $R$-algebra by $E_{1,2}$ and $E_{2,1}$.
\end{proposition}

That is, we have
\begin{align*}
    E \simeq \begin{bmatrix}
        R & S_2 \\ 
        S_{-2} & R
    \end{bmatrix}
\end{align*}
as an $R$-algebra with the cross-multiplication given by multiplication in $S$.

\subsection{Galois deformation theory in the $3$-adic Steinberg case}\label{sec:explicit-S}
Let $\Gamma$ be the absolute Galois group of $\BQ_3$, and let $\mu_3$ denote the group of third roots of unity. Let $F := \BQ_3(\mu_3)$ and $H := \Gal(\BQ_3(\zeta_3), \BQ_3) \simeq \BZ/2\BZ$. Write $F'$ for the maximal pro-$3$ extension of $\BQ_3(\zeta_3)$, which is a normal field extension of $\BQ_3$, and set $\Gamma' := \Gal(L'/L)$. By the Schur-Zassenhaus lemma, there exist a splitting of $\Gal(F'/\BQ) \to H$, which induces an isomorphism $\Gal(F'/\BQ) \rtimes H$.  Let $\omega: \Gamma \to k$ denote the mod $3$ cyclotomic character, which factors through $H$, and let $\conj{\rho}: \Gamma \to \GL_2(k)$ be the direct sum of characters $\omega \oplus 1$, with associated pseudorepresentation $\conj{D}$. Let $R, E$, and $S$ be the universal algebras associated to $\conj{D}$ as in \S \ref{sec:moduli}, where the fixed determinant $\psi$ is the multiplicative lift $\tilde{\omega}$ of $\omega$ to $\O$.  

We review some facts from \cite{bockle-3adic} as recounted in the proof of \cite[Proposition A.12]{paskunas-tung}), modifying the arguments to work with the algebraic versions of moduli functors instead of the formal versions. By \cite[Lemma 3.2]{bockle-3adic}, $\Gamma'$ may be presented as the pro-$3$ group with generators $x_1, x_2, x_3, x_4$ and relation $$r = x_1[x_1, x_2][x_3, x_4][x_4, x_3^{-1}][x_2x_1^{-1}x^3_1],$$
where $[x_i, x_j] := x_ix_jx_i^{-1}x_j^{-1}$, such that action of $H$ on $\Gamma'$ in the semidirect product $\Gal(F'/\BQ_p) \simeq \Gamma' \rtimes H$ is given by $\sigma(x_i) = x_i^{(-1)^i}$, where $\sigma$ is the nontrivial element of $H$. (Unlike the situation of \cite[Lemma 3.4.2]{JNW}, this is a genuine presentation of this group, not just an approximate one.)

We write matrices in indeterminate variables
\begin{align*}
    A_1 &= \begin{bmatrix}
        (1 + b_0c_0)^{1/2} & b_0 \\ c_0 & (1 + b_0 c_0)^{1/2} 
    \end{bmatrix} \\
    A_2 &= \begin{bmatrix}
        (1 + a_0)^{1/2} & 0 \\ 0 & (1 + a_0)^{-1/2}
    \end{bmatrix}\\
    A_3 &= \begin{bmatrix}
        (1 + b_1c_1)^{1/2} & b_1 \\ c_1 & (1 + b_1 c_1)^{1/2} 
    \end{bmatrix} \\
    A_4 &= \begin{bmatrix}
        (1 + a_1)^{1/2} & 0 \\ 0 & (1 + a_1)^{-1/2}
        \end{bmatrix}
\end{align*}
and define the $\BZ$-graded ring
\begin{align*}
    S' := \frac{\O\lb a_0, a_1, b_0c_0, b_0c_1, b_1c_0, b_1c_1\rb [b_0, b_1, c_0, c_1]}{(A_1^3[A_1, A_2][A_3, A_4][A_4, A_3^{-1}][A_2, A_1^{-1}]A^3_1 - \on{id}_2)}
\end{align*}
with $\deg(a_i) = 0, \deg(b_i) = 2, \deg(c_i) = -2$, where the ideal of relations is generated by the entries of the matrix
\begin{align*}
    A_1^3[A_1, A_2][A_3, A_4]&[A_4, A_3^{-1}][A_2, A_1^{-1}]A^3_1 - \on{id}_2 =: M - \on{id}_2.
\end{align*}
We may check that $M$ is invariant under antidiagonal transposition and has determinant $1$, so $M - \on{id}_2$ is of the form 
$$\begin{bmatrix}
    \sqrt{1 + XY} - 1 & X \\ Y & \sqrt{1 + XY} - 1\\ 
\end{bmatrix}$$
and so in fact we only need the relations coming from the off-diagonal entries. The upper right (resp. lower left) entry lies in degree $2$ (resp. degree $-2$), so we let $x_0, x_1$ be the coefficients of $b_1, b_0$ in $X$ respectively (resp. let $y_0, y_1$ be the coefficients of $c_1, c_0$ in $Y$ respectively).\footnote{Note the swapped order: we do this to more closely match the analogous presentation in \cite[\S 3.4]{JNW}, which will be relevant in our \S\ref{sec:ext-computations}.} We may therefore present 
\begin{align*}
    S' = \frac{\O\lb a_0, a_1, b_0c_0, b_0c_1, b_1c_0, b_1c_1\rb [b_0, b_1, c_0, c_1]}{(x_0b_1 + x_1b_0, y_0c_1 + y_1c_0)}.
\end{align*}
Modulo third powers in the variables, we have
\begin{align*}
    M - \on{id}_2 \equiv \begin{bmatrix}
        18b_0c_0 & 6b_0 - 2a_0b_0 - 2a_1b_1 \\ 
        6c_0 + 2a_0c_0 + 2a_1c_1 & 18b_0c_0
    \end{bmatrix},
\end{align*}
so we have approximations \begin{align*}
    x_1 &\equiv 6 - 2a_0 \\
    x_0 &\equiv -2a_1 \\
    y_1 &\equiv 6 + a_0 \\
    y_0 &\equiv 2a_1.
\end{align*} 

We note that $S$ has a natural involution induced by swapping the first and second basis vectors in the matrix presentation, i.e. conjugating all matrices $A_i$ by $\begin{bmatrix}
    0 & 1 \\ 1 & 0
\end{bmatrix}$. This involution sends $b_i \leftrightarrow c_i$ and sends $(1 + a_i)^{1/2} \leftrightarrow (1 + a_i)^{-1/2}$. This symmetry corresponds to twisting a deformation by $\omega$, which leaves the residual pseudodeformation unchanged. We will later exploit this symmetry to show that certain statements that hold for $b_i$ also hold automatically for $c_i$. 

We define $\rho': \Gamma \to M_2(S')$ by passing to the quotient $\Gamma \twoheadrightarrow \Gamma' \rtimes H$, sending $x_i \in \Gamma'$ to $X_i$, and sending the nontrivial element of $H$ to $\begin{bmatrix}
    -1 & 0 \\ 0 & 1
\end{bmatrix}.$ The relations we have imposed on $S'$ are precisely the relations necessary to make $\rho'$ well-defined, noting that
\begin{align*}
    \begin{bmatrix} -1 & 0 \\ 0 & 1\end{bmatrix} X_i \begin{bmatrix} -1 & 0 \\ 0 & 1\end{bmatrix}^{-1} = X_i^{(-1)^i}.
\end{align*}
The representation $\rho'$ has residual pseudocharacter $\conj{D}$ (reducing modulo the maximal ideal of $S'$ generated by $\varpi$ and the variables), so universal properties induce a ring homomorphism $R \to S'_0$ and an $R$-algebra homomorphism $\eta: E \to M_2(S')$ inducing a morphism of Cayley-Hamilton algebras such that $\rho' = \eta \circ \rho^{\on{univ}}$. Pulling back the standard idempotents in $M_2(S')$ along $\eta$ induces an $R$-GMA structure on $E$, and we let $S$ denote the ring representing the associated representation adapted moduli functor. The universal property of $S$ then induces a graded $R$-algebra homomorphism $\phi: S \to S'$, with degree $0$ part $R = S_0 \to S_0'$. 

\begin{proposition}
    $\varphi: S \to S'$ is an isomorphism.
\end{proposition}
\begin{proof}
    The universal adapted representation of $\conj{D}$ factors through $\Gamma' \rtimes H$. We have a couple constraints on the image of $\Gamma'$ under ${\rho}^{\on{univ}}$: 
    \begin{itemize}
        \item Since we work with fixed determinant, the image of $\Gamma'$ must have determinant $1$. 
        \item Letting $x_i \in \Gamma'$ be the generators from the Demu\v{s}kin presentation of $\Gamma'$ above, for the nontrivial element $\sigma \in H$ we must have $$\rho^{\on{univ}}(\sigma)\rho^{\on{univ}}(x_i)\rho^{\on{univ}}(\sigma)^{-1} = \rho^{\on{univ}}(x_i)^{(-1)^{i}}.$$ 
    \end{itemize}
    Using the chosen GMA structure to put coordinates on $E$, $\rho^{\on{univ}}$ sends $H$ to $\begin{bmatrix}
        \pm 1 & 0 \\ 0 & 1
    \end{bmatrix}$. The two constraints above then force $\rho^{\on{univ}}(x_i)$ to be of the form $\begin{bmatrix}
        (1 + bc)^{1/2} & b \\ c & (1 + bc)^{1/2}
    \end{bmatrix}$ for odd $i$ and $\begin{bmatrix}
        1 + a & 0 \\ 0 & (1 + a)^{-1}
    \end{bmatrix}$ 
    for even $i$. We may make a change of variables to express the latter matrix as $\begin{bmatrix}
        (1 + a)^{1/2} & 0 \\ 0 & (1 + a)^{-1/2}
    \end{bmatrix}.$ 
    
    Since the $x_i \in \Gamma'$ satisfy the relation $r$, we conclude that the universal adapted representation factors through $\rho'$, inducing an inverse to $\varphi$. 
\end{proof}

\subsection{Algebraic properties of deformation rings}

Proposition \ref{prop:alg-properties} lists the main properties of $R$ and $S$ that we will require. First, we prepare for the depth computation for $R$:

\begin{lemma}\label{lem:regular-sequence-s-to-r}
    Let $r_1, \dots, r_n$ be a regular sequence in a noetherian $\BZ$-graded ring $S$ with all $r_i \in R := S_0$, and suppose furthermore that $R$ is local noetherian. Then $r_1, \dots, r_n$ is also a regular sequence in $R$. 
\end{lemma}
\begin{proof}
    Let $K_\bullet(r_1, \dots, r_n; S)$ be the Koszul complex associated to the regular sequence in $S$, which by regularity is exact away from cohomological degree $0$. The $\BZ$-grading on $S$ induces a $\BZ$-grading on $K_\bullet(r_1, \dots, r_n; S)$, and since all of the $r_i$ lie in degree $0$, the maps in this complex preserve the $\BZ$-grading. Hence $K_\bullet(r_1, \dots, r_n; S)_0 = K_\bullet(r_1, \dots, r_n; R)$, where the latter denotes the Koszul complex taken with respect to the ambient ring $R$. Taking degree $0$ components is an exact functor, so we conclude that $K_\bullet(r_1, \dots, r_n; R)$ is exact away from cohomological degree $0$. Since $R$ is noetherian local, this implies that $r_1, \dots, r_n$ is also a regular sequence in $R$. 
\end{proof}

\begin{proposition}\label{prop:alg-properties}~
    \begin{enumerate}[(a)]
        \item $R$ is an $\O$-torsion-free domain of dimension $4$ and depth at least $3$. 
        \item $S$ is an $\O$-flat local complete intersection domain of dimension $5$. 
    \end{enumerate}
\end{proposition}
\begin{proof}
The fact that $R$ is $\O$-flat and a domain is \cite[Corollary 5.11]{bip}. To see that $R$ has dimension $4$, we claim that the images of $b_0c_1, b_1c_0, b_1c_1$ in $R/(b_0c_0)$ are zero divisors. We have relations
\begin{align*}
    x_0b_1c_0 + x_1b_0c_0 = y_0c_1b_0 + y_1c_0b_0 = 0
\end{align*} in $R$, and since none of the $x_i, y_i$ vanish in $R/(b_0c_0)$ these relations imply that $b_0c_1$ and $b_1c_0$ are zero divisors in $R/(b_0c_0)$. We also have the relation
\begin{align*}
    y_0c_1(x_0b_1 + x_1b_0) = x_0y_0b_1c_1 + x_1y_0b_0c_1 = 0
\end{align*}
in $R$; the previous relations show that $y_0b_0c_1 = 0$ in $R/(b_0c_0)$, and $x_0y_0$ does not vanish in $R/(b_0c_0)$, so $b_1c_1$ is also a zero divisor in $R/(b_0c_0)$, finishing the proof of the claim. Hence:
\begin{align*}\dim R &= \dim(R/(b_0c_0)) + 1 \\
    &=\dim(R/(b_0c_0, b_0c_1, b_1c_0, b_1c_1)) + 1 \\
    &= \dim(\O\lb a_0, a_1 \rb) + 1\\
    &= 4.
\end{align*}
    
Before proving the statement about the depth of $R$, we prove part (b) first.  To show that $S$ is a domain, we first show that the elements $b_i, c_i$ are regular. Suppose $b_0$ is a zero divisor in $S$; then by symmetry of the matrix relations, $c_0$ is also a zero divisor, and so $\dim S/(b_0, c_0) = \dim S$. However, $b_1$ and $c_1$ are also nonregular in $S/(b_0, c_0)$ since we have relations $x_1b_0 + x_0b_1 = y_1c_0 + y_0c_1 = 0$ in $S$, so we conclude that $\dim S = \dim S/(b_0, b_1, c_0, c_1)$. But this ring is just $\O\lb a_0, a_1\rb$, which has dimension $3$, whereas $S$ has dimension at least $5$ since it is a quotient of a $7$-dimensional ring by two relations; contradiction. Hence $b_0$ and $c_0$ are regular. The same argument shows that $b_1$ and $c_1$ are regular. 
    
    This lets us deduce that homogeneous elements of $S$ (with respect to the $\BZ$-grading) are regular. If $f$ is homogeneous and we have $fg = 0$, then without loss of generality $g$ is also homogeneous. We can multiply $f$ and $g$ by powers of $b_0$ or $c_0$ to get multiples $f', g'$ lying in $R$, again with $f'g' = 0$. Since $R$ is a domain, we conclude that $f'$ or $g'$ is $0$, and since the $b_i, c_i$ are regular we conclude that the corresponding $f$ or $g$ is also $0$. 

    Since all homogeneous elements of $S$ are regular, we conclude that all elements of $S$ are regular: if $fg = 0$ with $f, g \in S$ arbitrary, then by considering the maximal degree component of $f$ and $g$, we conclude that either $f = 0$ or $g = 0$. Hence $S$ is a domain. We also conclude that $S$ is $\O$-flat, since $S$ contains $\O \subset R$ as a subring. 

    To finish part (b), we show that $S$ is locally complete intersection. Since $S$ is a quotient of a $7$-dimensional ring by 2 relations, it suffices to show that $S$ has dimension at most $5$. Since we have established that $b_0$ is regular in $S$, $\dim S = \dim S/(b_0) + 1$, and $b_1$ is a zero divisor in $S/(b_0)$, so furthermore $\dim S = \dim S/(b_0, b_1) + 1$. Now,
    \begin{align*}
        S/(b_0, b_1) = \frac{\O\lb a_0, a_1\rb [c_0, c_1]}{(\conj{y}_1c_0 + \conj{y}_0c_1)}
    \end{align*}
    where $\conj{y}_0, \conj{y}_1$ are the images of $y_0, y_1$ in $S/(b_0, b_1)$. We have approximated $y_0 \equiv 2a_1, y_1 \equiv 6 + 2a_0$ modulo third powers, so $\conj{y}_1c_0 + \conj{y}_0c_1 \neq 0$ in $\O\lb a_0, a_1 \rb [c_0, c_1]$. Hence 
    $$\dim \frac{\O\lb a_0, a_1\rb [c_0, c_1]}{(\conj{y}_1c_0 + \conj{y}_0c_1)} = 4,$$
    so we conclude that $\dim S = 5$.     
 
    We return to finding a lower bound for the depth of $R$. We will show that $\varpi, a_0, b_1$ is a regular sequence in $S$. Once this is known, we may also conclude by symmetry that $\varpi, a_0, c_1$ is a regular sequence in $S$, so that in particular $b_0c_0$ is regular in $S/(\varpi, a_0)$ and hence $\varpi, a_1, b_0c_0$ is also regular in $S$. Then by Lemma \ref{lem:regular-sequence-s-to-r} the sequence $\varpi, a_1, b_0c_0$ is also a regular sequence in $R$. 
    
    Since $S$ is a noetherian local complete intersection, the claim will follow if we can show that $S/(\varpi, a_0, b_1)$ has dimension at most $5 - 3 = 2$. Reducing modulo $\varpi$ and setting $a_0 = b_1 = 0$, we may compute the relation modulo fourth powers:\begin{comment}directly 
    \begin{align*}
        [A_3, A_4][A_4, A_3^{-1}] \equiv \begin{bmatrix}
            1 & 0 \\
            \frac{-a_1c_1}{1 + a_1} & 1
        \end{bmatrix} \bmod (\varpi, a_0, b_1)
    \end{align*}
    and
    \begin{align*}
        A_1^3 \equiv \begin{bmatrix}
            (1 + b_0c_0)^{3/2} & b_0^2c_0 \\
            b_0 c_0^2 & (1 + b_0c_0)^{3/2}
        \end{bmatrix} \bmod (3), 
    \end{align*}
    and we may compute out \end{comment}
    \begin{align*}
    A_1^3[A_1, A_2][A_3, A_4]&[A_4, A_3^{-1}][A_2, A_1^{-1}]A^3_1\\
    &\equiv \begin{bmatrix}
        * & - b_0^2c_0 \\
        \frac{a_1c_1 + b_0c_0^2}{1 + a_1} & *
    \end{bmatrix} \bmod (\varpi, a_0, b_1) + \mideal^4
    \end{align*}
    where $\mideal = (a_0, a_1, b_0, b_1, c_0, c_1)$. 
    
    Write $f = -b_0^2c_0 + \dots$ for the upper right entry of the relation and $g = a_1c_1 + b_0c_0^2 + \dots$ for $1+a_1$ times the lower left entry of the relation, both considered as elements of $S'' := k\lb a_1, b_0c_0, b_0 c_1\rb [b_0, c_0, c_1]$, so that $S/(\varpi, a_0, b_1) \simeq S''/(f, g)$. The ring $S''$ inherits the $\BZ$-grading from $S$, with $f$ and $g$ of degree $2$ and $-2$ respectively. Any element $s$ of $S''$ may be decomposed into the sum of its $\BZ$-graded components $s = \sum_{i \in \BZ} s_i$, where the sum is finitely supported. 

    The ring $S''$ is a UFD by Lemma \ref{lem:s'-factorial} below. Suppose there exists $h \in S'$ dividing both $f$ and $g$. The initial forms\footnote{With respect to the standard monomial grading, not the $\BZ$-grading} $-b_0^2c_0$, $a_1c_1$ of $f$, $g$, respectively, are coprime in $S''$. The initial form of $h$ must divide the initial forms of both $f$ and $g$, so the initial form of $h$ must be a constant unit. Writing $f = hh'$, both $h$ and $h'$ must be concentrated in a single $\BZ$-degree: else either the maximal or minimal $\BZ$-graded component of $hh'$ would not be degree $2$. In particular, since we know that $h$ has a nonzero constant term, $h$ must be concentrated in $\BZ$-degree $0$, and therefore $h$ is itself a unit since $S'_0$ is local. 
    
    We conclude that the ideal $(f, g)$ must have height $2$: in a UFD, all height $1$ prime ideals are principal, so $(f, g)$ cannot be contained in any height $1$ prime ideal since $f$ and $g$ have no proper common divisor. Therefore $S''/(f, g) \simeq S/(\varpi, a_0, b_1)$ has dimension $2$, and thus $R$ has depth at least $3$ as explained previously.
    %\footnote{This proof seems pretty tortured. It's hard to write down a regular sequence of two elements that nicely simplify the matrix relations mod $\varpi$; for example, setting $a_0 = a_1 = 0$ kills both relations, meaning that $a_1$ is a zero divisor in $S/(a_0)$ and vice versa. I would hope that there is simpler proof for depth $\geq 3$, but it's not clear how well-behaved the ring $R$ is. Standard results on the Cohen-Macaulayness of the invariant subring of a CM ring under a reductive group action (i.e. Hochster-Roberts-type results) don't seem to hold in this generality, and I have no idea whether $R$ is CM or not. }  
\end{proof}

\begin{lemma}\label{lem:s'-factorial}
    The ring $S''$ in the previous proof is a UFD.
\end{lemma}
\begin{proof}
    Let $B := k\lb a_1, u_0, u_1\rb$, where $a_1, u_0, u_1$ are indeterminates. We may present $S''$ as $B[b_0, c_0, c_1]/(b_0c_1 - u_0, b_0c_1 - u_1)$. We have a localization map $S'' \to B[b_0, b_0^{-1}]$ given by sending $b_0 \mapsto b_0$, $c_0 \mapsto u_0/b_0$, $c_1 \mapsto u_1/b_0$. The kernel of this map is $(b_0c_0 - u_0, b_0c_1 - u_1) = 0$ in $S''$, so we conclude that we have an injection $S'' \hookrightarrow B[b_0, b_0^{-1}]$. In particular, $S''$ is a domain. 

    The ring $B[b_0, b_0^{-1}]$ is a UFD, being a Laurent extension of a formal power series ring. The given map $S'' \hookrightarrow B[b_0, b_0^{-1}]$ is the localization map for $S''$ obtained by inverting $b_0$. The element $b_0$ is prime in $S''$, since $S''/(b_0) = k\lb a_1 \rb[c_0, c_1]$ is a domain. We conclude by applying Nagata's criterion for factoriality (\cite[Tag 0AFU]{stacks-project}), which states that if a ring $S''$ has multiplicative subset $T$ generated by prime elements, then $S''$ is factorial if and only if every element in $S''$ has a factorization into irreducibles and $T^{-1}S''$ is factorial. $S''$ is noetherian, so the former condition holds, and we have shown that $T^{-1}S''$ is factorial with $T = \{b_0\}$. Hence $S''$ is factorial. 
\end{proof}

\begin{corollary}\label{cor:si-depth}
    For any $i \in \BZ$, $S_i$ has $R$-depth at least $3$.
\end{corollary}
\begin{proof}
     If $x_1, \dots, x_n \in R = S_0$, then $(S/(x_1, \dots, x_n)S)_i \simeq S_i/(x_1, \dots, x_n)S_i$, so if $x_n$ is regular in $S/(x_1, \dots, x_{n - 1})S$, then $x_n$ is also regular in $S_i/(x_1, \dots, x_{n - 1})S_i$. 
     
     In the proof of Proposition \ref{prop:alg-properties}, we found a regular sequence $\varpi, a_0, b_1c_1$ of length $3$ in $R$ that is also a regular sequence in $S$. By the reasoning in the previous paragraph, we conclude that $\varpi, a_0, b_1 c_1$ is also a regular sequence for $S_i$. 
\end{proof}
 Of course, $S_i = 0$ for odd $i$, but the depth of the zero module is infinity. 

\section{Representation theory of the 3-adic Steinberg block of $\GL_2(\BQ_3)$}

\subsection{Notation and preliminaries}\label{sec:rep-theory}

We establish the notation and setting of $p$-adic representation theory for $\GL_2(\BQ_p)$, mostly following the conventions of \cite[\S 2]{paskunas13}. For the moment, we let $G$ be an arbitrary locally pro-$p$ topological group. 

An $\O$-representation $V$ of $G$ is said to be \emph{smooth} if
\begin{align*}
    V = \bigcup_{i, H} V^H[\varpi^i]
\end{align*}
where the union is over compact open subgroups $H$ of $G$ and integers $i > 0$, and $(-)[\varpi^i]$ denotes the submodule of $\varpi^i$-torsion elements. $V$ is said to be \emph{locally finite} if each $v \in V$ generates a finite (equivalently, finite length) $G$-module. 

Let $Z$ denote the center of $G$, and fix a central character $\zeta: Z \to \O$. We write $\on{Mod}_{G}^{\on{lfin}}(\O)$ to denote the category of locally finite smooth $G$-modules with coefficients in $\O$, and we write $\on{Mod}_{G, \zeta}^{\on{lfin}}(\O)$ for the full subcategory of representations with central character $\zeta$. We write $\on{Mod}_{G}^{\on{lfin}}(k)$ and $\on{Mod}_{G, \zeta}^{\on{lfin}}(k)$ for the respective full subcategories consisting of $\varpi$-torsion objects. 

Let $H$ be an arbitrary open compact subgroup of $G$ with complete group algebra $\O\lb H \rb$. We let $\on{Mod}_{G}^{\on{pro-aug}}(\O)$ denote the category of profinite linearly topological $\O\lb H \rb$-modules with an action of $\O [G]$ that agrees with the restricted action of $\O[H]$, with morphisms being $G$-equivariant continuous homomorphisms of topological $A\lb H \rb$-modules. Since any two open compacts in $G$ are commensurable, this definition is independent of the choice of $H$ up to equivalence of categories. 

For a smooth $G$-module $V$ with $\O$-coefficients, we define the \emph{Pontryagin dual} $$V^\vee := \Hom_{\O}^{\on{cont}}(V, L/\O),$$ where $L/\O$ has the discrete topology and $V$ has the compact-open topology. The module $V^\vee$ is naturally an object of $\on{Mod}_{G}^{\on{oro-aug}}(\O)$, and moreover $V \mapsto V^\vee$ induces an anti-equivalence between the category of smooth $G$-modules over $\O$ and $\on{Mod}_{G}^{\on{pro-aug}}(\O)$. We write $\mathfrak{C}(\O)$ (resp. $\mathfrak{C}(k)$) for the essential image of $\on{Mod}_{G, \zeta}^{\on{lfin}}(\O)$ (resp. $\on{Mod}_{G, \zeta}^{\on{lfin}}(k)$) under this anti-equivalence. When we wish to clarify the ambient group, we may write $\mathfrak{C}_G(\O)$, etc. We will primarily consider $\mathfrak{C}(\O)$ simply to be the dual category to $\on{Mod}_{G, \zeta}^{\on{lfin}}(\O)$, without referring to the description of $\on{Mod}_{G}^{\on{pro-aug}}(\O)$ given above. 

By \cite[Corollary 2.3]{paskunas13}, injective envelopes exist both in $\on{Mod}_{G, \zeta}^{\on{lfin}}(\O)$ and in $\on{Mod}_{G, \zeta}^{\on{lfin}}(k)$, so projective envelopes exist in $\mathfrak{C}(\O)$ and $\mathfrak{C}(k)$. For $\pi\in \on{Mod}_{G, \zeta}^{\on{lfin}}(\O)$, we write $\tilde{J}_\pi$ for the injective envelope of $\pi$ taken inside $\on{Mod}_{G, \zeta}^{\on{lfin}}(\O)$. If furthermore $\pi \in \on{Mod}_{G, \zeta}^{\on{lfin}}(k)$, we write $J_{\pi}$ for the injective envelope of $\pi$ taken inside $\on{Mod}_{G, \zeta}^{\on{lfin}}(k)$. Likewise we write $\Ppi$ for the projective envelope of $\pi^\vee$ taken inside $\mathfrak{C}(\O)$ and we write $P_{\pi^\vee}$ for the projective envelope of $\pi^\vee$ taken inside $\mathfrak{C}(k)$. By the properties characterizing projective and injective envelopes, we have $J_{\pi}^\vee \simeq P_{\pi^\vee}$ and $\tilde{J}_{\pi}^\vee \simeq \Ppi$, and moreover by \cite[Lemma 2.11]{paskunas13}, we have $\Ppi \otimes_\O k \simeq P_{\pi^\vee}$ and $\tilde{J}_\pi[\varpi] \simeq J_{\pi}$. 

From now on, we set $G = \GL_2(\BQ_p)$ and write $Z$ for the center of $G$. Let $T \subset G$ be the standard diagonal torus, and let $P \subset G$ be the standard parabolic (Borel) subgroup consisting of upper-triangular matrices. If $\tau \in \on{Mod}_{T/Z}^{\on{lfin}}(\O)$, then we write $\Ind_P^G \tau \in \on{Mod}_{G/Z}^{\on{lfin}}(\O)$ for smooth parabolic induction, which is an exact functor $\on{Mod}_{T/Z}^{\on{lfin}}(\O) \to \on{Mod}_{G/Z}^{\on{lfin}}(\O)$ that restricts to an exact functor $\on{Mod}_{T/Z}^{\on{lfin}}(k) \to \on{Mod}_{G/Z}^{\on{lfin}}(k)$. If $\chi$ is a character of $\BQ_p^\times$, then by abuse of notation we also let $\chi$ (resp. $\chi_T$) refer to the character $\chi \circ \det$ of $G$ (resp. $T$), depending on context; note that we reserve the unadorned notation for $G$.  We let $\omega: \BQ_p^\times \to k$ denote the character $x \mapsto x|x| \bmod \varpi$, which corresponds to the mod $p$ cyclotomic character of $\Gal(\conj{\BQ}_p/\BQ_p)$ under the normalization of local class field theory sending geometric Frobenius to a uniformizer.

The irreducible objects in $\on{Mod}_{G}^{\on{lfin}}(\O)$, which necessarily lie in $\on{Mod}_{G}^{\on{lfin}}(k)$, have been classified in \cite[Theorem 33]{barthel-livne} and \cite[Theorem 1.1]{breuil}. They are:
\begin{itemize}
    \item The supersingular representations, which are of the form $$\pi(r, \lambda, \chi):= (\on{c-Ind}_{KZ}^G \on{Sym} \BF^r)/(T) \otimes \chi$$ for $r \in \{0, \dots, p - 1\}$, where $T$ is a certain Hecke operator acting on the induction and $\chi$ is a character;
    \item Principal series $\Ind_{P}^G \chi_1 \otimes \chi_2$, where $\chi_1, \chi_2$ are distinct characters;
    \item Characters $\chi$; and
    \item The special representations $\on{Sp} \otimes \chi$, which can be characterized as the unique irreducible quotient of $\Ind_{P}^G \chi_T$, fitting into a non-split exact sequence
    \begin{align*}
        0 \to \chi \to \Ind_P^G \chi_T \to \on{Sp} \otimes \chi \to 0.
    \end{align*}
\end{itemize}

This classification allows us to write down a block decomposition of $\on{Mod}_{G}^{\on{lfin}}(\O)$. In an abelian category, a \emph{block} $\mathfrak{B}$ is an equivalence class of irreducible objects, where the equivalence is generated by declaring $\pi \sim \tau$ if $\Ext^1(\pi, \tau) \neq 0$ or $\Ext^1(\tau, \pi) \neq 0$. By \cite[Proposition 5.42, Remark 5.43]{paskunas13}, after passing to a sufficiently large extension $L/\BQ_p$, the blocks of $\on{Mod}_{G}^{\on{lfin}}(\O)$ are:
\begin{enumerate}[(i)]
    \item Supersingular: $\mathfrak{B} = \{\pi\}$ for supersingular $\pi$;
    \item Generic principal series: $\mathfrak{B} = \{\Ind_{P}^G \chi_1 \otimes \chi_2 \omega^{-1}, \Ind_{P}^G \chi_2 \otimes \chi_1 \omega^{-1}\}$ with $\chi_1\chi_2^{-1} \neq 1, \omega^{\pm 1}$. 
    \item Non-generic principal series: If $p > 2$, $\mathfrak{B} = \{\Ind_{P}^G \chi \otimes \chi\omega^{-1}\}$
    \item Steinberg: If $p > 3$, $\mathfrak{B} = \{1, \Sp, \Ind_P^G \omega \otimes \omega^{-1}\} \otimes \chi$ for a character $\chi$;
    \item $p = 3$ Steinberg: If $p = 3$, $\mathfrak{B} = \{1, \Sp, \omega, \Sp \otimes \omega\} \otimes \chi$ for a character $\chi$; and
    \item $p = 2$ Steinberg/non-generic: If $p = 2$, $\mathfrak{B} = \{1, \Sp\} \otimes \chi$ for a character $\chi$. 
\end{enumerate}

We will always assume that $L$ is large enough so that all irreducible representations are absolutely irreducible. The mod $p$ local Langlands correspondence establishes a bijection of these blocks with semisimple residual representations (equivalently, pseudocharacters) of $\Gamma$. See \cite[\S4.1]{JNW} for a dictionary of this correspondence as well as further exposition of the blocks of $\on{Mod}_{G}^{\on{lfin}}(\O)$. Most importantly for our purposes, the Steinberg blocks correspond to the residual pseudorepresentations of the form $\omega \chi \oplus \chi$. 

For the remainder of this paper, we restrict to representations with trivial central character, since we may always twist by an appropriate character to translate to this case. We write $\on{Mod}_{G/Z}^{\on{lfin}}(\O)$, etc.\ instead of $ \on{Mod}_{G, 1}^{\on{lfin}}(\O)$, etc.\ in this case. For a block $\mathfrak{B}$, we write $\on{Mod}_{G/Z}^{\on{lfin}}(\O)^{\mathfrak{B}}$ for the full subcategory of $\on{Mod}_{G/Z}^{\on{lfin}}(\O)$ consisting of objects whose irreducible subquotients lie in $\mathfrak{B}$, and we write $\mathfrak{C}(\O)^{\mathfrak{B}}$ for the essential image of this subcategory in $\mathfrak{C}(\O)$. By \cite[Proposition 5.34, Corollary 5.35]{paskunas13}, we have a decomposition
\begin{align*}
    \on{Mod}_{G/Z}^{\on{lfin}}(\O) = \bigoplus_{\mathfrak{B}}  \on{Mod}_{G/Z}^{\on{lfin}}(\O)^{\mathfrak{B}}
\end{align*}
as a direct sum of categories, where the sum ranges over blocks $\mathfrak{B}$ with trivial central character. We may therefore analyze the category $\on{Mod}_{G/Z}^{\on{lfin}}(\O)$ by analyzing each block category independently. We write $\mathfrak{C}(\O)^{\mathfrak{B}}$ for the essential image of $\on{Mod}_{G/Z}^{\on{lfin}}(\O)^{\mathfrak{B}}$ under Pontryagin duality. 

\begin{remark} We pause to comment on our notation for Hom sets. If $M, N$ both lie in one of the categories we have discussed, then to avoid cluttering notation we simply write $\Hom(M, N)$ for the morphisms from $M$ to $N$ in this category, with the ambient category (usually one of $\on{Mod}_{G/Z}^{\on{lfin}}(\O), \on{Mod}_{G/Z}^{\on{lfin}}(k), \mathfrak{C}(\O), \mathfrak{C}(k)$ or the corresponding categories over $T/Z$) implied by context. To help avoid confusion, we point out that we only ever consider injective envelopes $J_{\pi}$ or $\tilde{J}_{\pi}$ over a category of locally finite smooth representations, and we only consider projective envelopes $P_{\pi^\vee}$ or $\tilde{P}_{\pi^\vee}$ over the dual category.
\end{remark}

We now fix a block $\mathfrak{B}$. We write 
\begin{align*}
    P := \bigoplus_{\pi \in \mathfrak{B}} P_{\pi^\vee}\\
    \tilde{P} := \bigoplus_{\pi \in \mathfrak{B}} \tilde{P}_{\pi^\vee}
\end{align*}
for the direct sum of all projective envelopes of irreducibles in $\mathfrak{C}(k)^{\mathfrak{B}}$ and $\mathfrak{C}(\O)^{\mathfrak{B}}$, respectively. One of our central objects of study is 
\begin{align*}
    \tilde{E}_{\mathfrak{B}} := \End(\tilde{P}). 
\end{align*}
The importance of this algebra comes from Morita theory. There is a natural topology on $\tilde{E}_{\mathfrak{B}}$ making it into a compact algebra, and by \cite[Proposition 5.45]{paskunas13} there is an equivalence of categories between $\mathfrak{C}(\O)^{\mathfrak{B}}$ and the category of compact right $\tilde{E}_{\mathfrak{B}}$-modules. Dually, $\on{Mod}_{G/Z}^{\on{lfin}}(\O)^{\mathfrak{B}}$ is equivalent to the category of discrete left $\tilde{E}_{\mathfrak{B}}$-modules. See \cite[\S4.1]{JNW} and the discussion preceding \cite[Corollary 2.5]{paskunas13} for more details. We have
\begin{align*}
    \tilde{E}_{\mathfrak{B}} = \bigoplus_{\pi, \tau \in \mathfrak{B}} \Hom(\Ppi, \Ptau),
\end{align*}
so much of our analysis will focus on the properties of the individual $\Hom(\Ppi, \Ptau)$ for each pair of irreducibles $\pi, \tau$. 

We will utilize Emerton's functor of ordinary parts $\on{Ord}_P(-)$, which is right adjoint to $\Ind_{\conj{P}}^G(-)$, where $\conj{P}$ denotes the opposite standard parabolic subgroup. See \cite[\S 7.1]{paskunas13} for a brief discussion of this functor and its relevant properties. We note that since $\on{Ord}_P$ is adjoint to $\Ind_{\conj{P}}^G$, rather than $\Ind_{P}^G$, we must perform a certain twist to translate between the two types of induction. Letting $s := \begin{bmatrix}
    0 & 1 \\ 1 & 0
\end{bmatrix}$, for $M \in \on{Mod}^{\on{lfin}}_{T/Z}(\O)$ we write $M^s$ for the representation obtained by conjugating by $s$. Then we have a natural isomorphism
\begin{align*}
    \Ind_{\conj{P}}^G M \simeq \Ind_P^G M^s.
\end{align*}
A small mercy of the $p = 3$ Steinberg case is that the only characters $1, \omega$ of $T/Z$ we consider are scalar (i.e. factor through the determinant), and for such characters we have $\chi^s = \chi$, which effectively means that we need not worry about twisting by $s$ at all. 

\newpage 

\subsection{Useful resolutions in the Steinberg block}
In this section, we set $p = 3$ and let $\mathfrak{B} = \{1, \Sp, \omega, \Sp \otimes \omega\}$ be the Steinberg block of $\on{Mod}_{G/Z}^{\on{lfin}}(\O)$. 

The $\Ext^1$-groups between irreducibles in $\mathfrak{B}$ have already been computed in \cite[Theorem 11.4]{paskunas_supersingular}. The $k$-dimensions of these groups are:
\begin{center}\begin{table}[h]\begin{center}\begin{tabular}{|c||c|c|c|c|} 
    \hline $\dim_k \Ext^1(\tau, \pi)$ & $1$ & $\on{Sp}$ & $\omega$ & $\on{Sp}\otimes \omega$ \\ 
    \hline \hline $1$ & 0 & 2 & 0 & 0\\
    \hline $\on{Sp}$ & 1& 0 & 1 & 0 \\
    \hline $\omega$ & 0 & 0 & 0& 2 \\
    \hline $\on{Sp} \otimes \omega $ &1& 0& 1& 0\\
    \hline
\end{tabular}
\end{center}
\caption{$\Ext^1$ dimensions}
\label{table:ext1}
\end{table}
\end{center}
where the entry in row $\tau$ and column $\pi$ is $\dim_k \Ext^1(\tau, \pi)$.  Half of these entries are not stated directly in loc.\ cit., but they follow from the other half by twisting by $\omega$. 

\begin{remark}
    Here and in the rest of this section, we write an unadorned $\Ext^i(-, -)$ to refer to the Ext group computed in the category over $k$, with the exact ambient category implied by context (one of $\on{Mod}_{G/Z}^{\on{lfin}}(k)$, $\mathfrak{C}(k)$, or the versions of these categories over $T/Z$). We will have no need to directly consider Ext groups computed in the larger categories over $\O$. 
\end{remark}

Write $1_T \in \on{Mod}_{T/Z}^{\on{lfin}}(k)$ for the trivial representation of $T$, and let $J_{1_T}$ denote its injective envelope in this category. We reserve unadorned notation for characters in $\on{Mod}_{G/Z}^{\on{lfin}}(k)$.
\begin{lemma}
    The only absolutely irreducible objects in $\on{Mod}_{T/Z}^{\on{lfin}}(\O)$ are characters, and $\Ext^1(\chi_1, \chi_2) = 0$ if $\chi_1, \chi_2$ are distinct characters of $T$. Hence the blocks of $\on{Mod}_{T/Z}^{\on{lfin}}(\O)$ consist of singleton characters. 
\end{lemma}
\begin{proof}
    See \cite[Corollary 3.35, Proposition 5.46]{paskunas13}.  
\end{proof}
\begin{lemma}\label{lem:exts-to-indJ1}~
    \begin{enumerate}[(a)]
        \item $\Ext^1(\on{Sp} \otimes \omega, \Ind_P^G J_{1_T})$ is $1$-dimensional.
        \item $\Ext^1(\pi, \Ind_P^G J_{1_T}) = 0$ for $\pi \in \{1, \omega, \on{Sp}\}$.
        \item $\Ext^i(\pi, \Ind_P^G J_{1_T}) = 0$ for all $i \geq 2$ and $\pi \in \mathfrak{B}$. 
    \end{enumerate}
\end{lemma}

\begin{proof}
    \cite[Lemma 7.4]{paskunas13} gives $$\Ext^1(\Ind_P^G 1_T, \Ind_P^G J_{1_T}) = \Hom_T(1_T^s, J_{1_T} \otimes \omega_T) = \Hom_T(1_T, J_{1_T} \otimes \omega_T) = 0,$$ since the only irreducible subrepresentations of $J_{1_T} \otimes \omega_T$ are isomorphic to $\omega_T$. The vanishing of $\Ext^1_{G/Z}(1, \Ind_P^G J_{1_T})$ and $\Ext^1_{G/Z}(\on{Sp}, \Ind_P^G J_{1_T})$ then follows from \cite[Proposition 10.6]{paskunas13}.

    Likewise, we have $$\Ext^1(\Ind_P^G \omega_T, \Ind_P^G J_{1_T}) = \Hom_T(\omega_T, J_{1_T} \otimes \omega_T),$$ which is $1$-dimensional, so \cite[Proposition 10.6]{paskunas13} implies that one of $\Ext^1(\omega, \Ind_P^G J_{1_T})$, $\Ext^1(\on{Sp} \otimes \omega, \Ind_P^G J_{1_T})$ is $1$-dimensional and the other is $0$. But \cite[Lemma 10.50]{paskunas13} gives $\Ext^1(\omega, \Ind_P^G J_{1_T}) = 0$, since $\on{Ord}_P \omega = 0$, so we conclude (a) and (b). 

    Part (c) follows from \cite[Lemma 7.4]{paskunas13}, which shows that $$\Ext^i(\Ind_P^G 1_T, \Ind_P^G J_{1_T}) = \Ext^i(\Ind_P^G \omega_T, \Ind_P^G J_{1_T}) = 0$$ for all $i \geq 2$, and \cite[Proposition 10.6]{paskunas13} again. 
\end{proof}
\begin{remark}
    We note that all of the results cited in the previous proof are indeed valid for $p = 3$. In particular, the proofs of \cite[Proposition 10.6, Lemma 10.50]{paskunas13} depend only on \cite[\S5, \S7]{paskunas13} and \cite[Proposition 3.6.2, Remark 4.2.1]{emerton-ord-ii}, all of which allow $p = 3$. 
\end{remark}

We write $\tilde{J}_{1_T}$ for the injective envelope of $1_T$ in $\on{Mod}_{T/Z}^{\on{lfin}}(\O)$, which, by \cite[Corollary 5.21]{paskunas13}, satisfies $\tilde{J}_{1_T}[\varpi] = J_{1_T}$. 

\begin{lemma}\label{lem:subquots-J1}
    All irreducible subquotients of $\Ind_P^G \tilde{J}_{1_T}$ are isomorphic to either $1$ or $\Sp$. 
\end{lemma}
\begin{proof}
    We first note that if $\pi$ is an irreducible object of $\on{Mod}_{G/Z}^{\on{lfin}}(\O)$, then $\pi$ appears as an irreducible subquotient of another object $M \in \on{Mod}_{G/Z}^{\on{lfin}}(\O)$ if and only if $\Hom(M, \tilde{J}_\pi) \neq 0$. Indeed, if we have a diagram of the form
    \begin{center}
        \begin{tikzcd}
            N \ar[r, hook]\ar[d, two heads] & M \\
            \pi
        \end{tikzcd}
    \end{center}
    expressing $\pi$ as a subquotient of $M$, then by injectivity of $\tilde{J}_{\pi}$ we may extend this diagram to 
    \begin{center}
        \begin{tikzcd}
            N \ar[r, hook]\ar[d, two heads] & M \ar[ldd, dashed] \\
            \pi \ar[d] \\
            \tilde{J}_{\pi}
        \end{tikzcd}
    \end{center}
    with the dashed arrow necessarily nonzero. Conversely, given a nonzero morphism $\varphi: M \to \tilde{J}_{\pi}$, by the definition of an essential monomorphism the image of $M$ under $\varphi$ must contain $\pi$, so that $\pi$ appears as a subrepresentation of $M/\ker(\varphi)$. 
    
    Now fix $\pi = \omega$ or $\on{Sp} \otimes \omega$. Then $\pi$ is not a subquotient of $\Ind_P^G \tilde{J}_{1_T}$ if and only if $\Hom(\Ind_P^G \tilde{J}_{1_T}, \tilde{J}_{\pi}) = 0$. By the adjunction from Emerton's functor of ordinary parts, this is equivalent to vanishing of $\Hom(\tilde{J}_{1_T}^s, \on{Ord}_P(\tilde{J}_\pi)) \simeq \Hom(\tilde{J}_{1_T}, \on{Ord}_P(\tilde{J}_\pi))$, using the properties of injective envelopes to deduce $\tilde{J}_{1_T}^s \simeq \tilde{J}_{1_T}$. Then the desired result follows from \cite[Lemma 4.4]{paskunas-tung}: $\on{Ord}_{P}(\tilde{J}_\pi)$ is an injective envelope of $\omega_T$ (regardless of whether $\pi = \omega$ or $\on{Sp} \otimes \omega$), so it lies in a different block of $\on{Mod}_{T/Z}^{\on{lfin}}(\O)$ than $1_T$. Hence $\Hom(\tilde{J}_{1_T}, \on{Ord}_P(\tilde{J}_\pi)) = 0$.
\end{proof}

\begin{proposition}\label{prop:injective-resolutions}
    We have exact sequences
    \begin{center}
    \begin{tikzcd}
        0 \ar[r] & \Ind_{P}^G J_{1_T} \ar[r] & J_1 \ar[r] & J_{\on{Sp} \otimes \omega} \ar[r] & 0\\
        0 \ar[r] & (\Ind_P^G J_{1_T})/1 \ar[r] & J_{\on{Sp}}
    \end{tikzcd}
    \end{center}
    in the category $\on{Mod}_{G/Z}^{\on{lfin}}(k)$. 
\end{proposition}

\begin{proof}
    \cite[Proposition 7.1(ii)]{paskunas13} gives an injection $$\Ind_{\conj{P}}^G \on{Ord}_P J_{\Ind_{\conj{P}}^G 1_T} \hookrightarrow J_{\Ind_{\conj{P}}^G 1_T}.$$ Meanwhile \cite[Proposition 7.1(i)]{paskunas13} and the immediately preceding formulas (125), (126) let us identify the source of this injection with $\Ind_P^G J_1$, and $J_{\Ind_{\conj{P}}^G 1_T} \simeq J_{1}$ follows from the fact that an injective envelope is uniquely determined by the socle. 

    We may extend this injection to an injective resolution of $\Ind_{P}^G J_{1_T}$ via the procedure outlined in \cite[Remark 10.11]{paskunas13}; computing an injective resolution amounts to computing $\Ext^i(\pi, \Ind_{P}^G J_{1_T})$ for all $i \geq 1$ and all irreducibles $\pi \in \mathfrak{B}$. These Ext groups are supplied by Lemma \ref{lem:exts-to-indJ1}, and we obtain the injective resolution
    \begin{align*}
        0 \to \Ind_{P}^G J_{1_T} \to J_1 \to J_{\on{Sp} \otimes \omega} \to 0.
    \end{align*}

    Write $\kappa := (\Ind_P^G J_{1_T})/1$. It suffices to show that the socle of $\kappa$ is $\on{Sp}$, or equivalently that $\Hom(\pi, \kappa) = 0$ for $\pi \neq \on{Sp}$ and $\Hom(\Sp, \kappa) = k$. By Lemma \ref{lem:subquots-J1}, we need only check for the irreducibles $1$ and $\on{Sp}$. 
    
    By applying $\Hom(1, -)$ to the short exact sequence defining $\kappa$, we get exact sequence
    \begin{align*} \Hom(1, 1) \to \Hom(1, \Ind_P^G J_{1_T})\to \Hom(1, \kappa) \to \Ext^1(1, 1) 
\end{align*}
    From Table \ref{table:ext1}, $\Ext^1(1, 1) = 0$, and $\Hom(1, \Ind_P^G J_{1_T}) \simeq k$ by the first exact sequence in the proposition statement, so the arrow $\Hom(1, 1)  \to  \Hom(1, \Ind_P^G J_{1_T})$ is an isomorphism and we deduce that $\Hom(1, \kappa) = 0$. 

    The long exact sequence for $\Hom(\on{Sp}, -)$ gives an exact sequence
    \begin{align*}
         \Hom(\on{Sp}, \Ind_P^G J_{1_T}) = 0 \to  \Hom(\on{Sp}, \kappa) \to
         \Ext^1(\on{Sp}, 1) \to \Ext^1(\on{Sp}, \Ind_P^G J_{1_T})  
        \end{align*}
    From Table \ref{table:ext1} we have $\dim_k \Ext^1(\on{Sp}, 1) = 1$, and Lemma \ref{lem:exts-to-indJ1} shows that $\Ext^1(\on{Sp}, \Ind_P^G J_{1_T}) = 0$, so we deduce $\Hom(\on{Sp}, \kappa) = k$. 
\end{proof}

We now pass to the dual category $\mathfrak{C}(\O) = \on{Mod}_{G/Z}^{\on{lfin}}(\O)^{\on{op}}$. Write $\tilde{M}_{1^\vee} := (\Ind_{P}^G \tilde{J}_{1_T})^\vee$. We let $\Monenot$ denote the kernel of the natural surjection $\Mone \to \O$ (lifting the dual of the natural inclusion $1 \hookrightarrow \Ind_{P}^G J_{1_T}$). See also the discussion at the beginning of \cite[\S10.5]{paskunas13}.

We obtain a commutative diagram
\begin{equation}
        \label{seq:main} \begin{tikzcd}
        & & & \O &\\
        0 \ar[r] &  \PSpw\ar[r, "\alpha"] & \Ptriv \ar[r] &  \Mone \ar[r] \ar[u, two heads]& 0\\
&& \PSp \ar[u, "\beta", dashed]\ar[r]  & \Monenot \ar[u, hook] \ar[r] & 0
\end{tikzcd}
\end{equation}
with exact rows. This diagram follows from the mod $p$ injective resolutions in Proposition \ref{prop:injective-resolutions}, dualizing, and then using \cite[Proposition 5.21]{paskunas13} to lift the resolutions to characteristic $0$. The arrow $\beta$ is induced by inclusion $\Monenot \hookrightarrow \Mone$ and the standard fact that morphisms between objects can be extended to a morphism between their projective resolutions. 

\begin{remark}
    These resolutions are not canonical, so we fix a particular choice of the Diagram (\ref{seq:main}) from now on.
\end{remark}

\definition{We let $\alpha: \PSpw \hookrightarrow \Ptriv$ and $\beta: \PSp \to \Ptriv$ be the labeled morphisms in Diagram (\ref{seq:main}).}

\begin{proposition}\label{prop:hom-identifications}
    The morphism $\alpha$ induces by postcomposition:
    \begin{enumerate}[(a)]
        \item An isomorphism $\End(\PSpw) \simeq \Hom(\PSpw, \Ptriv)$;
        \item An isomorphism $\Hom(\Pw, \PSpw) \simeq \Hom(\Pw, \Ptriv)$; and
        \item An injection $\Hom(\Ptriv, \PSpw) \hookrightarrow \End(\Ptriv)$. 
    \end{enumerate}
\end{proposition}
\begin{proof}
    Part (c) follows immediately from applying $\Hom(\Ptriv, -)$ to the first exact sequence in Diagram \ref{seq:main}. Parts (a) and (b) follow by similarly applying $\Hom(\PSpw, -)$ and $\Hom(\Pw, -)$ and using Lemma \ref{lem:subquots-J1} to deduce that $\Hom(\PSpw, \Mone) = \Hom(\Pw, \Mone) = 0$. 
\end{proof}

\begin{proposition}\label{prop:alpha-beta-image}
    The image of $\alpha \oplus \beta: \PSpw \oplus \PSp \to \Ptriv$ is equal to the kernel of the composition $\Ptriv \to \Mone \to \O$ in Diagram (\ref{seq:main}). 
\end{proposition}
\begin{proof}
    This is a straightforward diagram chase. See also the analogous result \cite[Lemma 10.75]{paskunas13}.
\end{proof}

\subsection{The center} \label{sec:paskunas-tung} Fix $p \geq 2$, and let $\mathfrak{B}$ be an arbitrary block. Let $$\tilde{P}' := \bigoplus_{\substack{\pi \in  \mathfrak{B}\\ \dim \pi > 1}} \Ppi$$ be the direct sum of the projective envelopes of irreducible representations in $\mathfrak{B}$ that are not characters, and let $\tilde{E}'_{\mathfrak{B}} := \End(\tilde{P}')$. Let $R$ be the universal pseudodeformation ring of fixed determinant with residual pseudorepresentation corresponding to $\mathfrak{B}$, and let $E$ be the universal Cayley-Hamilton algebra. The foundation for our work is the following result:
\begin{theorem}[\cite{paskunas-tung}, Theorem 1.4] \label{thm:paskunas-tung} For $p \geq 3$, Colmez's Montreal functor induces isomorphisms \begin{align*} E_{\on{tf}}^{\on{op}} &\simeq \tilde{E}'_{\mathfrak{B}} \\
R_{\on{tf}} &\simeq Z(\mathfrak{C}(\O)^{\mathfrak{B}}),
\end{align*}
where the subscript $\on{tf}$ denotes the maximal $\O$-torsion-free quotient and $Z(\mathfrak{C}(\O)^{\mathfrak{B}})$ is the Bernstein center of the category $\mathfrak{C}(\O)^{\mathfrak{B}}$. 
\end{theorem}

Hence by Proposition \ref{prop:alg-properties} (for $p = 3$, with torsion-freeness already known for $p \geq 5$ from \cite{paskunas13}):
\begin{corollary}\label{cor:E-vs-E}
    In fact, the maps in Theorem \ref{thm:paskunas-tung} are already isomorphisms $E^{\on{op}} \simeq \tilde{E}_\mathfrak{B}'$ and $R \simeq Z(\mathfrak{C}(\O))$. 
\end{corollary}
This means that $R$ acts canonically on every object in $\mathfrak{C}(\O)$, inducing an $R$-module structure on every Hom group $\Hom(X, Y)$, where the $R$-module structures coming from acting on the source and the target agree. We will use this canonical $R$-module structure frequently without further comment. 

We may identify certain Hom modules using this result:
\begin{corollary}\label{cor:spw-sp-homs}
    In the generic or Steinberg cases, the isomorphism $E^{\on{op}} \to \tilde{E}_{\mathfrak{B}}'$ can be upgraded into an isomorphism of GMAs. In particular, if $p=3$ and $\mathfrak{B}$ is the Steinberg block, we obtain isomorphisms \begin{align*}
        \End(\PSpw) \simeq \End(\PSp) &\simeq R \\
        \Hom(\PSpw, \PSp) &\simeq S_2 \\
        \Hom(\PSp, \PSpw) &\simeq S_{-2}
    \end{align*}
    coming from an isomorphism of $R$-algebras
    \begin{align*}
    \tilde{E}_{\mathfrak{B}}' \simeq \begin{bmatrix}
        \End(\PSpw) & \Hom(\PSpw, \PSp) \\
        \Hom(\PSp, \PSpw) & \End(\PSp)
    \end{bmatrix} \simeq \begin{bmatrix}
        R & S_2 \\ S_{-2} & R
    \end{bmatrix} \simeq E
    \end{align*}
    preserving the standard idempotents on each side. 
\end{corollary}
\begin{proof}
    We only prove the first part for the $p = 3$ Steinberg case; the other cases (for generic blocks or $p \geq 5$ Steinberg) are similar. We already know $\tilde{E}'_{\mathfrak{B}} \simeq E$ as continuous $R$-algebras. From the proof \cite[Theorem 1.4]{paskunas-tung}, this isomorphism passes through an isomorphism $E \simeq \End(\check{\mathbf{V}}(\tilde{P}'))$ in the Galois category, where $\check{\mathbf{V}}$ is the contravariant version of the Montreal functor used in loc.\ cit. Under the normalizations used in loc.\ cit, $\check{\mathbf{V}}$ sends $(\Sp \otimes \chi)^\vee$ to the character $\chi$ for $\chi \in \{1, \omega\}$, so the standard idempotents on $\End(\check{\mathbf{V}}(\tilde{P}')) \simeq \tilde{E}_{\mathfrak{B}}^{', \on{op}}$ lift the standard idempotents on the residual pseudorepresentation $\omega \oplus 1$. 
    
    The claim then follows since 
    the GMA structure on $E$ is unique up to conjugacy (assuming a fixed order for the residual idempotents), as any pair of idempotent lifts differ by conjugation by an element of $1 + \mideal_R E$, where $\mideal_R$ is the maximal ideal of $R$. We note that we have fixed residual idempotents so that the idempotent associated to the character $\omega$ is first. 
    \begin{comment}Letting $K := \on{Frac}(R)$, we note that $E_K := E \otimes_R K$, hence also $\tilde{E}'_{\mathfrak{B}, K} := \tilde{E}'_{\mathfrak{B}} \otimes_R K$ is isomorphic to $M_2(K)$, so by change of basis over $K$ we may adjust any isomorphism $\tilde{E}'_{\mathfrak{B}, K} \to E_K$ to send the idempotents $\on{id}_{\PSpw}, \on{id}_{\PSp}$ to the standard idempotents $e_{11}, e_{22} \in E_K$ coming from the matrix coordinates we have fixed. Since these idempotents already lie in $\tilde{E}'_{\mathfrak{B}}$ and $E$, respectively, these isomorphisms descend to a coordinate-preserving isomorphism as desired. 
    \end{comment}
\end{proof}

\begin{corollary}\label{End-is-R}
    For $p = 3$ Steinberg block, the maps $R \simeq Z(\mathfrak{C}(\O)) \to \End(\Ppi)$ are isomorphisms for all $\pi \in \mathfrak{B}$.
\end{corollary}
\begin{proof}
    We already know this for $\pi = \Sp, \Sp \otimes \omega$ by the previous corollary. To deduce the claim for the other irreducibles, we argue as in \cite[Proposition 10.76]{paskunas13}.
    
    The claim follows for $\pi = 1$ if we can show that $\alpha: \PSpw \hookrightarrow \Ptriv$ in Diagram \ref{seq:main} induces an isomorphism $\End(\PSpw) \simeq \End(\Ptriv)$. Every endomorphism of $\Ptriv$ preserves $\PSpw$ since $\Hom(\PSpw, \Mone) = 0$, so the restriction map $\End(\Ptriv) \to \End(\PSpw)$ is well-defined. The argument in \cite[Proposition 10.76]{paskunas13} goes through to reduce to showing $\End(J_1) \to \End(J_{\Sp \otimes \omega})$ (in the mod $\varpi$ category) is an isomorphism. This follows from \cite[Lemma 10.60]{paskunas13} with $J = J_1$, where now we have $J/J^1 \simeq J_{\Sp \otimes \omega}$ in our case by Proposition \ref{prop:injective-resolutions}. We note that the assertion $\chi^s \neq \chi \alpha$ with $\chi = 1_T$ used to prove loc.\ cit.\ remain true for $p = 3$, since $1_T^s = 1_T$ and $\alpha = \omega_T \neq 1_T$ (but this fails for $p = 2$). 

    The claim for $\pi = \omega$ then follows by twisting. 
\end{proof}

\subsection{Matrix embedding of the endomorphism algebra}\label{sec:matrix-embedding}
For the remainder of this section, we fix $p = 3$ and let $\mathfrak{B}$ be the Steinberg block. 
 The $R$-modules $\Hom(\Ppi, \Ptau)$ for $\pi, \tau \in \mathfrak{B}$ seem tricky to present explicitly. Fortunately, our argument will not require the explicit structure of these modules. Inspired by \cite[Lemma 3.6.18]{DEG26}, we will show that $\tilde{E}_{\mathfrak{B}} := \End(\bigoplus_{\pi \in \mathfrak{B}} \Ppi)$ embeds into a matrix algebra over $K := \on{Frac}(R)$, which helps to rigidify the structure of this endomorphism algebra and more easily identify morphisms.

\begin{proposition}[\cite{paskunas-tung} Proposition 4.3]\label{prop:no-socle}
    Let $\tilde{P}, \tilde{M} \in \mathfrak{C}(\O)$ with $\tilde{P}$ projective and $\tilde{M}/\varpi \tilde{M}$ of finite length. Then $\Hom(\tilde{M}, \tilde{P}) = 0.$
\end{proposition}

\begin{lemma}
    $\Hom(\Ptriv, \O) = \O$.
\end{lemma}
\begin{proof}
    By Proposition \ref{prop:alpha-beta-image}, the kernel $\tilde{P}_{1^\vee, 0}$ of the natural map $\Ptriv \to \O$ is a quotient of $\PSpw \oplus \PSp$. Since every irreducible subquotient of $\O$ is isomorphic to $1^\vee$, $\Hom(\PSpw \oplus \PSp, \O) = 0$, hence also $\Hom(\tilde{P}_{1^\vee, 0}, \O) = 0$. Consequently, every homomorphism $\Ptriv \to \O$ factors through $\Ptriv/\tilde{P}_{1^\vee, 0} \simeq \O$, so that $\Hom(\Ptriv, \O) = \End(\O) = \O$. 
\end{proof}
\begin{lemma}\label{lem:1-mone}~
    \begin{enumerate}[(a)]
        \item There is a continuous surjection of rings (and $R$-modules) $R = \End(\Ptriv) \to \End(\Mone) \simeq \O\lb x, y\rb$. 
        \item We have $R$-module isomorphisms $$\End(\Mone) \simeq \Hom(\Ptriv, \Mone) \simeq \Hom(\PSp, \Mone).$$ 
        \item Any $R$-module of the form $\Hom(\Mone, N)$ or $\Hom(N, \Mone)$ with $N \in \mathfrak{C}(\O)$ is $R$-torsion.
    \end{enumerate}
\end{lemma}
\begin{proof}
\begin{enumerate}[(a)]
    \item This part follows by \cite[Corollary 7.2]{paskunas13}. 
    \item This follows by the same methods used to prove \cite[Lemma 7.1(iii)]{paskunas13}. By \cite[Lemma 4.4.]{paskunas-tung} and \cite[Lemma 7.1]{paskunas13}, $\on{Ord}_{P} \tilde{J}_1$, $\on{Ord}_P \tilde{J}_{\Sp}$, and $\on{Ord}_P (\Ind_{P}^G \tilde{J}_{1_T})$ are all isomorphic to $\tilde{J}_{1_T}$ (noting that $\tilde{J}_{1_T}^s \simeq \tilde{J}_1$). Therefore, we have a morphism $\Ind_P^G \tilde{J}_{1_T} \to \tilde{J}_{1}$  over $G$ adjoint to the identity on $\tilde{J}_{1_T}$ over $T$. Applying $\Hom(\Ind_P^G \tilde{J}_{1_T}, -)$ to this morphism gives a commutative diagram
    \begin{center}
        \begin{tikzcd}
            \Hom_{T/Z}(\tilde{J}_{1_T}, \tilde{J}_{1_T}) \ar[d, equals] \ar[r, equals] & \Hom_{T/Z}(\tilde{J}_{1_T}, \tilde{J}_{1_T}) \ar[d, equals] \\
            \Hom_{G/Z}(\Ind_P^G \tilde{J}_{1_T},\Ind_P^G \tilde{J}_{1_T}) \ar[r] & \Hom_{G/Z}(\Ind_P^G \tilde{J}_{1_T}, \tilde{J}_1),
        \end{tikzcd}
            \end{center}
        where the vertical arrows are the adjunction isomorphisms, so the bottom arrow is also an isomorphism; see the similar proof of \cite[Lemma 7.1(iii)]{paskunas13}. Dualizing yields an isomorphism $$\Hom(\Ptriv, \Mone) \simeq \End(\Mone).$$ The same argument applies with $\tilde{J}_1$ replaced by $\tilde{J}_{\Sp}$, yielding $$\Hom(\PSp, \Mone) \simeq \End(\Mone).$$
    \item By part (a), $R$ acts on $\Mone$ through a proper quotient (noting $\O\lb x, y\rb \not \simeq R$). Since $R$ acts centrally on $\mathfrak{C}(\O)$, it must also act through that quotient on $\Hom(\Mone, N)$ and $\Hom(N, \Mone)$. 
\end{enumerate}
    
\end{proof}
\begin{lemma}\label{lem:Sp-lift}~
    \begin{enumerate}[(a)]
        \item There exists a lift of $\Sp^\vee$ to an $\O$-torsion-free representation $\tilde{\Sp}^\vee$ such that $\tilde{\Sp}^\vee \otimes_\O k \simeq \Sp^\vee$.  
        \item For any lift $\tilde{\Sp}^\vee$ as in part (a), we have a partial projective resolution
        \begin{center}
            \begin{tikzcd}
                \Ptriv^{\oplus 2} \ar[r] & \PSp \ar[r] & \tilde{\Sp}^\vee \ar[r] & 0.
            \end{tikzcd}
        \end{center}
    \end{enumerate}
\end{lemma}
\begin{proof}
    \begin{enumerate}[(a)]
        \item By Pontryagin duality it is equivalent to show that there is a $p$-divisible object $\tilde{\Sp}$ in $\on{Mod}_{G/Z}^{\on{lfin}}(\O)$ such that $\tilde{\Sp}[\varpi] \simeq \Sp$. This essentially amounts to showing that the standard construction of $\Sp$ behaves well with respect to base change and taking $\varpi$-torsion. 
    
    Let $\Lambda := L/\O$ and set $$\tilde{\Sp} := C^\infty(\BP^1(\BQ_p), \Lambda)/\Lambda,$$ where the submodule $\Lambda$ is identified with the constant functions. We claim that $\tilde{\Sp}[\varpi] \simeq C^\infty(\BP^1(\BQ_p), k)/k = \Sp$. Fix an inclusion $k \simeq \left(\frac{1}{\varpi} \O\right)/\O \hookrightarrow  \Lambda$. 
    
    The inclusion $k\hookrightarrow \Lambda$ induces a map $\Sp \to \tilde{\Sp}[\varpi]$, which we claim is an isomorphism. This map is injective since a nonconstant function remains so after the inclusion. Surjectivity follows from the fact that if $\varpi \cdot f = c$ for $f \in C^\infty(\BP^1(\BQ_p), \Lambda)$ and a constant $c \in \Lambda$, then $f$ must be of the form $\frac{1}{\varpi} c + g$ for a function $g$ valued in $(\frac{1}{\varpi} \O)/\O$, so that $f$ lies in the image of $\Sp \to \tilde{\Sp}[\varpi]$. Hence we have an isomorphism as desired. 
    \item As with our previous resolutions, this resolution follows from Table \ref{table:ext1} and \cite[Corollary 5.21, Remark 10.11]{paskunas13}. 
    \end{enumerate}
\end{proof}

\begin{lemma}\label{lem:Hom-1-Sp}
    Precomposition with the inclusion $\alpha: \tilde{P}_{(\on{Sp} \otimes \omega)^\vee} \hookrightarrow \tilde{P}_{1^\vee}$ induces an isomorphism $\Hom(\tilde{P}_{1^\vee}, \tilde{P}_{\on{Sp}}) \simeq \Hom(\tilde{P}_{(\on{Sp} \otimes \omega)^\vee}, \tilde{P}_{\on{Sp}})$. 
\end{lemma}
\begin{proof}
     We use the resolution for $\tilde{\on{Sp}}^\vee$ from Lemma \ref{lem:Sp-lift}(b);
    let $\tilde{P}_{\on{Sp}^\vee, 0}$ denote the kernel of $\tilde{P}_{\on{Sp}^\vee} \to \tilde{\on{Sp}}^\vee$, and let $\varphi: \tilde{P}_{(\on{Sp} \otimes \omega)^\vee} \to \tilde{P}_{\on{Sp}^\vee}$ be an arbitrary morphism. Since every irreducible subquotient of $\tilde{\on{Sp}}^\vee$ is isomorphic to $\on{Sp}^\vee$, the image of $\varphi$ is contained in $\tilde{P}_{\on{Sp}^\vee, 0}$. By exactness, the image of $(\gamma_1, \gamma_2): \Ptriv^{\oplus 2} \to \PSp$ from Lemma \ref{lem:Sp-lift}(b) is $\tilde{P}_{\on{Sp}^\vee, 0}$. Consider the diagram
    \begin{center}
        \begin{tikzcd}
            & \PSpw \ar[d, "\varphi"] \ar[ld, dashed, swap,  "\psi"] & \\
            \Ptriv^{\oplus 2} \ar[r, two heads, "{(\gamma_1, \gamma_2)}"] & {\tilde{P}_{\on{Sp}^\vee, 0}}
        \end{tikzcd}
    \end{center}
   By projectivity, we obtain a morphism $\psi: \PSpw \to \Ptriv^{\oplus 2}$ completing the diagram. We know already that $\End(\PSpw) = \Hom(\PSpw, \Ptriv) \simeq R$ by Proposition \ref{prop:hom-identifications}, so the projection of $\psi$ to either factor of $\tilde{P}_{1^\vee}^{\oplus 2}$ must be of the form $\alpha \circ r = r \circ \alpha$ for some $r \in R = Z(\mathfrak{C}(\O)) = \End(\PSpw) = \End(\Ptriv)$. Hence, $\varphi \in \Hom(\tilde{P}_{(\on{Sp} \otimes \omega)^\vee}, \tilde{P}_{\on{Sp}})$ is of the form $$\varphi = r_1(\gamma_1 \circ \alpha) + r_2 (\gamma_2 \circ \alpha) = (r_1 \gamma_1 + r_2 \gamma_2) \circ \alpha$$ for some $r_1, r_2 \in R$, so the map $\Hom(\Ptriv, \PSp) \to \Hom(\PSpw, \PSp)$ induced by precomposition with $\alpha$ is surjective.

   Applying $\Hom(-, \PSp)$ to the first row of Diagram \ref{seq:main} yields an exact sequence
    \begin{center}
    \begin{tikzcd}
        0 \ar[r] & \Hom(\Mone, \PSp) \ar[r] & \Hom(\Ptriv, \PSp) \ar[r, "- \circ \alpha"] & \Hom(\PSpw, \PSp)
    \end{tikzcd}
    \end{center}
    so injectivity of the map $- \circ \alpha: \Hom(\Ptriv, \PSp) \to \Hom(\PSpw, \PSp)$ is equivalent to vanishing of $\Hom(\Mone, \PSp)$. This module is $R$-torsion by Lemma \ref{lem:1-mone}(c), so vanishing follows if we can show that $\Hom(\Ptriv, \PSp)$ is $R$-torsion-free. 
    
    Again by Table \ref{table:ext1} and \cite[Corollary 5.21, Remark 10.11]{paskunas13}, we have the start of a projective resolution for $\O$:
    \begin{center}
        \begin{tikzcd}
          \PSpw \oplus \PSp \ar[r] & \Ptriv \ar[r] & \O.
        \end{tikzcd}
    \end{center}
    Applying $\Hom(-, \PSp)$ to this yields an exact sequence
\begin{align*}
          0 \to \Hom(\O, \PSp) \to \Hom(\Ptriv, \PSp) \to \Hom(\PSpw, \PSp) \oplus \End(\PSp).
\end{align*}
    The first term vanishes by Proposition \ref{prop:no-socle}, so we conclude that $\Hom(\Ptriv, \PSp)$ embeds into $\Hom(\PSpw, \PSp) \oplus \End(\PSp)$, which is the torsion-free $R$-module $S_{2} \oplus R$ by Corollary \ref{cor:spw-sp-homs}. Hence $\Hom(\Ptriv, \PSp)$ is also $R$-torsion-free, as desired. 
\end{proof}

Let $K := \on{Frac}(R)$. We say that an $R$-module $M$ has \emph{rank} $n$ if $\dim_K (M \otimes_R K) = n$. Note that this notion is distinct from the minimum number of generators for $M$.

\begin{proposition}\label{prop:homs-are-ideals}
    For each $\pi, \tau \in \mathfrak{B}$, the $R$-module $\Hom(\Ppi, \Ptau)$ is a rank $1$ torsion-free $R$-module.
\end{proposition}
\begin{remark}
    Since $R$ is an integral domain, being torsion-free of rank $1$ is equivalent to being isomorphic to a nonzero ideal of $R$. 
\end{remark}
\begin{proof}
    This is true when $\pi = \tau$, since $\Hom(\Ppi, \Ppi) = R$ by Corollary \ref{End-is-R}. It is also true for $\Ppi, \Ptau \in \{\PSpw, \PSp\}$, since these Hom modules are isomorphic to $R, S_{2},$ or $S_{-2}$ by Corollary \ref{cor:spw-sp-homs}, which can be shifted to a nonzero ideal of $R$ by multiplication by an element in $S$ of appropriate degree (since $S$ is an integral domain). We can then deduce the claim for $\Hom(\Ptriv, \PSp)$ from Lemma \ref{lem:Hom-1-Sp}, and Proposition \ref{prop:hom-identifications} lets us deduce the claim for $\Hom(\Ptriv, \PSpw)$ and $\Hom(\PSpw, \Ptriv)$. The claim for $\Hom(\Pw, \Ptriv)$ also follows from Proposition \ref{prop:hom-identifications} and the fact that $$\Hom(\Pw, \PSpw) \simeq \Hom(\Ptriv, \PSp)$$ by twisting.  

    This covers all cases up to twist except for the module $$\Hom(\PSp, \Ptriv)\simeq \Hom(\PSpw, \Pw).$$ Apply $\Hom(\PSp, -)$ to the first row of Diagram \ref{seq:main} to get an exact sequence
    \begin{center}
    \begin{tikzcd}
        0 \ar[r] & \Hom(\PSp, \PSpw) \ar[r] & \Hom(\PSp, \Ptriv) \ar[r] & \Hom(\PSp, \Mone) \ar[r] & 0. 
    \end{tikzcd}
    \end{center}
    We may tensor this sequence of $R$-modules with $K = \on{Frac}(R)$ to get an exact sequence of $K$-vector spaces
    \begin{center}
    \begin{tikzcd}
        0 \ar[r] & \Hom(\PSp, \PSpw)_K \ar[r] & \Hom(\PSp, \Ptriv)_K \ar[r] & \Hom(\PSp, \Mone)_K \ar[r] & 0. 
    \end{tikzcd}
    \end{center}
    We already concluded that the first term is isomorphic to $K$, and the last term is $0$ by Lemma \ref{lem:1-mone}(c). Hence $\Hom(\PSp, \Ptriv)_K \simeq K$. Finally, we check that $\Hom(\PSp, \Ptriv)$ has no $R$-torsion: apply $\Hom(-, \Ptriv)$ to the exact sequence in Lemma \ref{lem:Sp-lift}(b) to get an exact sequence
    \begin{center}
    \begin{tikzcd}
        0 \ar[r] & \Hom(\tilde{\Sp}^\vee, \Ptriv) \ar[r] & \Hom(\PSp, \Ptriv) \ar[r] & \Hom(\Ptriv, \Ptriv^{\oplus 2})
    \end{tikzcd}
    \end{center}
    Since $\Hom(\tilde{\Sp}^\vee, \Ptriv) = 0$ by Proposition \ref{prop:no-socle}, we conclude that $\Hom(\PSp, \Ptriv)$ embeds into $\End(\Ptriv)^{\oplus 2} = R^{\oplus 2}$ and is therefore $R$-torsion-free. 
\end{proof}
\begin{lemma}\label{lem:same-center}
    Let $\varphi \in \Hom(\Ppi, \Ptau)$, $\psi \in \Hom(\Ptau, \Ppi)$. Then $\varphi \circ \psi$ and $\psi \circ \varphi$ map to the same element of $R$ under the isomorphisms $R \simeq Z(\mathfrak{C}(\O)) \simeq \End(\Ppi) \simeq \End(\Ptau)$. 
\end{lemma}
\begin{proof}
    Let $r, r'$ be the elements of $R \simeq Z(\mathfrak{C}(\O))$ restricting to $\varphi \circ \tau, \tau \circ \varphi$, respectively. Then: 
    \begin{align*}
        0 &= (\varphi \circ \psi - r) \circ \on{id}_{\Ptau}  \\
        &= (\varphi \circ \psi - r) \circ \varphi\\
        &= \varphi \circ \psi \circ \varphi - r \varphi \\
        &= \varphi \circ r' - r \circ \varphi \\
        &= \varphi \circ (r' - r)
    \end{align*}
    using centrality for the last step. If $\varphi = 0$, then clearly $r = r' = 0$. Otherwise, since $\Hom(\Ppi, \Ptau)$ is $R$-torsion-free by Proposition \ref{prop:homs-are-ideals}, we conclude that $r = r'$.  
\end{proof}

\begin{lemma}\label{lem:non-degen-pairing}
    For any $\Ppi, \Ptau \in \{\PSpw, \PSp, \Ptriv, \Pw\}$, the composition map $\Hom(\Ppi, \Ptau) \otimes_R \Hom(\Ptau, \Ppi) \to \End(\Ppi)$ is nonzero. 
\end{lemma}

\begin{proof}
    We first note that since all of modules $\Hom(\Ppi, \Ptau)$ are torsion-free of rank 1 by Proposition \ref{prop:homs-are-ideals}, any nonzero $R$-module homomorphism between them must be an injection. Moreover, to prove the claim, it suffices to show that there is a either nonzero morphism $\Hom(\Ppi, \Ptau) \to \End(\Ppi)$ induced by postcomposing with a particular map $\Ptau \to \Ppi$ or that there is a nonzero morphism $\Hom(\Ptau, \Ppi) \to \End(\Ppi)$ induced by precomposing with a particular map $\Ptau \to \Ppi$. Finally, if $\psi: \Ppi \to \Ptau, \varphi: \Ptau \to \Ppi$, we note that Lemma \ref{lem:same-center} implies that if $\varphi \circ \psi \neq 0$, then $\psi \circ \varphi \neq 0$ too, so the order of $\Ppi$ and $\Ptau$ does not matter for verifying the claim.

    We verify the claim for each pair $\pi, \tau \in \mathfrak{B}$, using the above observations to reduce the casework.  
    \begin{itemize}
    \item $\Ppi = \Ptau$: The claim is clearly true due to the existence of the identity map. 
    \item $\Ppi = \Ptriv, \Ptau = \PSpw$: Proposition \ref{prop:hom-identifications}(a) shows that the map $\alpha$ induces an injection $\Hom(\Ptriv, \PSpw) \to \End(\Ptriv)$. 
        \item $\Ppi = \Ptriv, \Ptau = \PSp$: The exact sequence at the end of the proof of Proposition \ref{prop:homs-are-ideals} shows that there exist a pair of maps $\Ptriv \to \PSp$ inducing an injection $\Hom(\PSp, \Ptriv) \to \End(\Ptriv^{\oplus 2})$ by precomposition, so at least one of the corresponding maps $\Hom(\PSp, \Ptriv) \to \End(\Ptriv)$ must be nonzero. 
        \item $\Ppi = \Pw, \Ptau = \Ptriv$: Applying $\Hom(-, \Pw)$ to the first row of Diagram \ref{seq:main} gives an exact sequence
        \begin{center}
            \begin{tikzcd}
               0 \ar[r] & \Hom(\Mone, \Pw) \ar[r] & \Hom(\Ptriv, \Pw) \ar[r] & \Hom(\PSpw, \Pw).
            \end{tikzcd}
        \end{center}
        Since $\Hom(\Mone, \Pw)$ is $R$-torsion and $\Hom(\Ptriv, \Pw)$ has no $R$-torsion, the former must be $0$. Hence precomposition with $\alpha$ induces an injection $$\Hom(\Ptriv, \Pw) \hookrightarrow \Hom(\PSpw, \Pw).$$ From the previous case, we know that we can choose a morphism $\varphi: \Ptriv \to \PSp$ inducing an injection $\Hom(\PSp, \Ptriv) \to \End(\Ptriv)$, so that $\varphi \otimes \omega$ induces an injection $\Hom(\PSpw, \Pw) \to \End(\Pw)$. Hence the composition $(\varphi \otimes \omega)\circ \alpha$ induces an injection $\Hom(\Ptriv, \Pw) \to \End(\Pw)$ by precomposition. 
        \item $\Ppi = \PSpw, \Ptau = \PSp$: This case follows from the identification $$\End(\PSpw \oplus \PSp) \simeq \begin{bmatrix}
            R & S_2 \\ S_{-2} & R
        \end{bmatrix}$$ and the fact that the image of the cross-multiplication is the (nonzero) reducibility ideal $(b_0c_0, b_0c_1, b_1c_0, b_1c_1)$ in $R$. 
    \end{itemize}
    The claim for all other pairs $\Ppi, \Ptau$ follows from the above cases by twisting and/or swapping the arguments. 
\end{proof}

We arrive at the main result of this section, which strongly constrains the behavior of $\tilde{E}_{\mathfrak{B}} = \End\left(\bigoplus_{\pi \in \mathfrak{B}} \Ppi \right)$.
\begin{proposition}\label{cor:embed-matrix}
    Let $K = \on{Frac}(R)$. Then $\tilde{E}_{\mathfrak{B}}$ embeds into $M_{4 \times 4}(K)$ as an $R$-algebra.
\end{proposition}

\begin{proof}
    Proposition \ref{prop:homs-are-ideals} implies that $\tilde{E}_K := \tilde{E}_{\mathfrak{B}} \otimes_R K$ is a $16$-dimensional $K$-algebra.  Since $\tilde{E}_{\mathfrak{B}}$ is $R$-torsion-free, we have a natural injection $\tilde{E}_{\mathfrak{B}} \hookrightarrow \tilde{E}_K$, so it suffices to prove that the latter is isomorphic to $M_{4 \times 4}(K)$ as a $K$-algebra. 
    
    Since $R = Z(\tilde{E}_{\mathfrak{B}})$, we conclude that $\tilde{E}_K$ is a central $K$-algebra. Lemma \ref{lem:non-degen-pairing}, along with the existence of the four idempotents in $\tilde{E}_K$ coming from the identity maps on each $\Ppi$, implies that this algebra must be simple. By the Artin-Wedderburn theorem, the only central simple $K$-algebra of dimension $16$ with a system of four orthogonal idempotents is $M_{4 \times 4}(K)$. 
\end{proof}
\begin{remark}\label{rem:rigidity}
    We may adjust the embedding given in Corollary \ref{cor:embed-matrix} to send the four idempotents $\on{id}_{\Ppi}$ in $\tilde{E}_\mathfrak{B}$ to the four standard idempotents in $M_{4\times 4}(K)$, so that the embedding preserves the decomposition $$\tilde{E}_{\mathfrak{B}} \simeq \bigoplus_{\pi, \tau \in \mathfrak{B}} \Hom(\Ppi, \Ptau)$$ by mapping each component $\Hom(\Ppi, \Ptau)$ into a particular matrix entry, yielding a specific choice of embedding $\Hom(\Ppi, \Ptau) \hookrightarrow K$. Fixing one such embedding $\tilde{E}_{\mathfrak{B}}\hookrightarrow M_{4 \times 4}(K)$ once and for all, we may identify a given $\varphi \in \Hom(\Ppi, \Ptau)$ with a fixed element $x_{\varphi}\in K$, and moreover \emph{any} (post- or pre-!) composition map between Hom modules induced by $\varphi$ may be identified with multiplication by this same element $x_{\varphi}$. 
\end{remark}

\begin{remark}
    While sufficient for our purposes, a choice of embedding from Corollary \ref{cor:embed-matrix} does not seem very natural. One might hope to write down a more canonical-looking embedding involving the ring $S$, e.g. \`a la \cite[Remark 3.6.26]{DEG26}. 
\end{remark}

\subsection{Depths of Hom modules}\label{sec:depths}

A key input for our computation of $\tilde{E}_{\mathfrak{B}}$ will be that the $R$-depths of the modules $\Hom(\Ppi, \Ptau)$ are sufficiently large. We recall from Proposition \ref{prop:alg-properties} that $R$ has depth at least $3$. 

\begin{lemma}[Depth lemma]\label{lemma:depth-lemma}
     Let $R$ be a noetherian local ring and let $M$ be a finitely generated $R$-module. Let
    $$0 \to N \to M \to P \to 0$$
    be an exact sequence of $R$-modules. Then:
    \begin{enumerate}[(a)]
        \item $\on{depth}_R(N) \geq \min\{\on{depth}_R(M), \on{depth}_R(P) + 1 \}$.
        \item $\on{depth}_R(M) \geq \min\{\on{depth}_R(N), \on{depth}_R(P)\}$.
        \item $\on{depth}_R(P) \geq \min\{\on{depth}_R(M), \on{depth}_R(N) - 1\}$.
    \end{enumerate}
\end{lemma}
\begin{proof}
    This is  \cite[{Tag 00LX}]{stacks-project}. All of these inequalities follow from the long exact sequence obtained by applying $\Hom(R/\mideal_R, -)$ to the short exact sequence and the characterization of $\on{depth}_R(M)$ as the smallest integer $i$ such that $\Ext_R^i(R/\mideal_R, M) \neq 0$.
\end{proof}
The hypotheses of Lemma \ref{lemma:depth-lemma}
will hold in our applications because Proposition \ref{prop:homs-are-ideals} implies that all the modules $\Hom(\Ppi, \Ptau)$ are finite over $R$ since they are isomorphic to ideals in the noetherian local ring $R$. 
\begin{comment}
\begin{lemma}\label{lem:spw-sp-depth}
    $\Hom(\PSpw, \PSp)$ and $\Hom(\PSp, \PSpw)$ have $R$-depth at least $3$. 
\end{lemma}

\begin{proof}
    By Corollary \ref{cor:spw-sp-homs}, these $R$-modules are isomorphic to $S_2$ and $S_{-2}$, and they are abstractly isomorphic to each other, so it suffices to prove that $S_2$ has depth $4$. 

    We have a presentation
    \begin{center}
        \begin{tikzcd}
            0 \ar[r] &R \ar[r, "\begin{bmatrix} x_1 \\ x_0
            \end{bmatrix}"] & R^{\oplus 2} \ar[r, "(b_0 \; b_1)"] \ar[r] & S_2 \ar[r] & 0
        \end{tikzcd}
    \end{center}
    The depth lemma then immediately gives $\on{depth}_{R} S_2 \geq \on{depth}_R R - 1 = 3$. 
\end{proof}
\end{comment}
\begin{proposition}\label{prop:Hom-depths}
    For any $\pi \in \mathfrak{B}$, the $R$-modules $\Hom(\Ppi, \Ptriv)$ have $R$-depth at least $3$. 
\end{proposition}
\begin{proof}
    The claim is true for $\Hom(\Ptriv, \Ptriv) = R$ by Proposition \ref{prop:alg-properties}, and it is likewise true for $\Hom(\PSpw, \Ptriv) = \Hom(\PSpw, \PSpw) = R$ by the same proposition and Proposition \ref{prop:hom-identifications}. Proposition \ref{prop:hom-identifications} also tells us that $\Hom(\Pw, \Ptriv) = \Hom(\Pw, \PSpw) = \Hom(\Ptriv, \PSp)$, which by Lemma \ref{lem:Hom-1-Sp} is isomorphic to $$\Hom(\PSpw, \PSp) \simeq S_2,$$ which has $R$-depth at least $3$ by Corollary \ref{cor:si-depth}. 

    Hence the only remaining case is $\Hom(\PSp, \Ptriv)$. From Diagram \ref{seq:main}, we obtain an exact sequence
    \begin{center}
        \begin{tikzcd}
            0 \ar[r] & \Hom(\PSp, \PSpw) \ar[r] & \Hom(\PSp, \Ptriv) \ar[r] & \Hom(\PSp, \Mone) \ar[r] & 0.
        \end{tikzcd}
    \end{center}
    By Corollaries \ref{cor:si-depth} and \ref{cor:spw-sp-homs}, $\Hom(\PSp, \PSpw)$ has depth at least $3$, and by Lemma \ref{lem:1-mone}, $\Hom(\PSp, \Mone) \simeq \O\lb x, y\rb$ has depth $3$. Therefore, the depth lemma yields
    \begin{align*}
        \on{depth}_R(\Hom(\PSp, \Ptriv)) &\geq \min \{\on{depth}_R(\Hom(\PSp, \PSpw)), \on{depth}_R(\Hom(\PSp, \Mone))\} \\
        &= 3,
    \end{align*}
    as desired.

    \begin{comment}true for $ = S_{-2}$ by Corollaries \ref{cor:si-depth} and \ref{cor:spw-sp-homs}, and true for $$\Hom(\Pw, \PSpw) \simeq \Hom(\Ptriv, \PSp) \simeq \Hom(\PSpw, \PSp) \simeq S_{2}$$ by Lemma \ref{lem:Hom-1-Sp} and Corollary \ref{cor:spw-sp-homs}. It remains to prove the claim for $\Hom(\Ptriv, \PSpw)$. From Diagram \ref{seq:main}, we obtain an exact sequence
    \begin{align*}
        0 \to \Hom(\Ptriv, \PSpw) \to \End(\Ptriv) \to \Hom(\Ptriv, \Mone) \to 0.
    \end{align*}
    By Lemma \ref{lem:1-mone}, we know that $\Hom(\Ptriv, \Mone)$ is isomorphic to $\O\lb x, y \rb$ with $R$ acting through a quotient isomorphic to this ring, so $\on{depth}_R \Hom(\Ptriv, \Mone) = \on{depth}_{\O\lb x, y \rb} \O\lb x, y \rb = 3$. Hence the depth lemma yields
    \begin{align*}
        \on{depth}_R \Hom(\Ptriv, \PSpw) &\geq \min\{\on{depth}_R \End(\Ptriv), \on{depth}_R \Hom(\Ptriv, \Mone) + 1\}\\
        &= \min\{\on{depth}(R), 4\}\\
        &\geq 3.
    \end{align*}
    \end{comment}
\end{proof}

\section{The tensor kernel and its endomorphism algebra}

\subsection{The candidate sheaf}\label{sec:candidate}

Let $\mathfrak{V}$ be the universal vector bundle on $\mathfrak{X} = [\Spec(S)/\mathbb{G}_m]$, so that we may identify $\mathfrak{V} \simeq S(-1) \oplus S(1)$, where if $M$ is a $\BZ$-graded $S$-module then $M(i)$ denotes the $S$-module $M$ with new grading $M(i)_d := M_{i + d}$. 

For this subsection, we work in some generality. Let $p \geq 3$ be prime, and let $\mathfrak{B}$ be either the generic or Steinberg block of $\on{Mod}_{G/Z}^{\on{lfin}}(\O)$. Let 
\begin{align*}
\tilde{P} &:= \bigoplus_{\pi \in \mathfrak{B}} \Ppi \\
    \tilde{P}' &:= \bigoplus_{\substack{\pi \in \mathfrak{B} \\  
\dim_k(\pi) > 1}}\Ppi \\
\tilde{P}'' &:= \bigoplus_{\substack{\pi \in \mathfrak{B} \\  
\dim_k(\pi) = 1}} \Ppi
\end{align*}
where the direct sum for $\tilde{P}''$ (resp. $\tilde{P}'$) ranges over irreducibles $\pi \in \mathfrak{B}$ that are characters (resp. not characters). In all cases, $\tilde{P} = \tilde{P}' \oplus \tilde{P}''$ and $\tilde{P}'$ is the direct sum of two projective envelopes of irreducible non-characters. 

\begin{itemize}
    \item If $\mathfrak{B}$ is the generic block, then $\tilde{P}'' = 0$ and $\tilde{P} = \tilde{P}'$. 
    \item If $\mathfrak{B}$ is the Steinberg block for $p \geq 5$, then $\tilde{P}' = \tilde{P}_{(\Ind_{P^G}(\omega \otimes \omega^{-1}))^\vee} \oplus \PSp$ and $\tilde{P}'' = \Ptriv$.
    \item Finally, if $\mathfrak{B}$ is the Steinberg block and $p = 3$, then $\tilde{P}' = \PSpw \oplus \PSp$ and $\tilde{P}'' = \Ptriv \oplus \Pw$. 
\end{itemize} 

Let $\tilde{E}_{\mathfrak{B}} := \End(\tilde{P})$ and $\tilde{E}_\mathfrak{B}' := \End(\tilde{P}')$. By Corollary \ref{cor:E-vs-E}
for the $p = 3$ Steinberg block, \cite[\S5.3] {JNW} for the generic block\footnote{As we discuss in \S\ref{sec:non-steinberg}, the results in loc.\ cit.\ remain valid for $p = 3$.}, and \cite[\S5.5]{JNW} for the $p \geq 5$ Steinberg block, we have $\End_{\mathfrak{X}}(\mathfrak{V})^{\on{op}} \simeq {\tilde{E}_{\mathfrak{B}}}'$. (This isomorphism is presented in a uniform way in \cite{paskunas-tung}.)

The following definition generalizes \cite[Definition 3.6.12]{DEG26}. 
\begin{definition}\label{def:candidate-sheaf}
    Write
    \begin{align*}
        \mathfrak{W} := \mathfrak{V} \otimes_{\tilde{E}_{\mathfrak{B}}'} \Hom(\tilde{P}, \tilde{P}')
    \end{align*}
    and for $\pi \in \mathfrak{B}$, write
    \begin{align*}
        \mathfrak{W}^{\pi} := \mathfrak{V} \otimes_{\tilde{E}_{\mathfrak{B}}'} \Hom(\Ppi, \tilde{P}'),
    \end{align*}
    where $\tilde{E}_{\mathfrak{B}}'$ acts on $\mathfrak{V}$ from the right via the left action of $\End_{\mathfrak{X}}(\mathfrak{V}) \simeq E \simeq (\tilde{E}'_{\mathfrak{B}})^{\on{op}}$, and $\tilde{E}_{\mathfrak{B}}'$ acts on $\Hom(-, \tilde{P}')$ from the left by its natural action on $\tilde{P}'$.

    We write $\mathfrak{W}_{\on{tf}}$ or $\mathfrak{W}_{\on{tf}}^{\pi}$ for the maximal $R$-torsion-free quotient\footnote{Note that this differs from our usage of the subscript $\on{tf}$ in \S\ref{sec:paskunas-tung}, where it referred to the maximal $\O$-torsion-free quotient.} of these respective modules.
\end{definition}

Our candidate for our tensor kernel, playing the role of $X^*$ in \cite[Theorem 5.1]{JNW}, is $\mathfrak{W}_{\on{tf}}$. The main goal of this paper is to show that $$\End_{\mathfrak{X}}(\mathfrak{W}_{\on{tf}}) \simeq \tilde{E}_{\mathfrak{B}}^{\on{op}}$$ and to show that $$\Ext^i_{\mathfrak{X}}(\mathfrak{W}_{\on{tf}}, \mathfrak{W}_{\on{tf}}) = 0$$ for $i > 0$, so that the sheaf $\mathfrak{W}_{\on{tf}}$ satisfies the hypotheses of \cite[Assumption 4.2.10]{JNW}. 

\subsection{Endomorphism algebra isomorphism}

For the remainder of this section, we take $p = 3$ and we take $\mathfrak{B}$ to be the Steinberg block. We do this for simplicity of exposition and because the results are already known in the generic and $p \geq 5$ Steinberg cases. However, we note that the methods of this section may be adapted to give a proof of \cite[Theorem 3.6.44]{DEG26} for $p \geq 5$ without requiring the explicit identifications of the respective endomorphism algebras computed in \cite[Theorem 3.5.5]{JNW} and \cite[\S 10.5]{paskunas13}, using the setup established in \S\ref{sec:candidate}. The key facts we use, such as the matrix embedding of the endomorphism algebra, depth results, and identifications of certain Hom modules, all have appropriate known analogues deducible from the explicit presentations in \cite{JNW} and \cite{paskunas13}. Our methods and setup may similarly be adapted to obtain the known results for the generic blocks, although many parts of the argument become trivial in this case.  
\begin{proposition}\label{prop:categorical-interpretation}
    As a graded $S$-module, $\mathfrak{W} = \mathfrak{V} \otimes_{\tilde{E}'} \Hom(\tilde{P}, \tilde{P}')$ has graded components
    \begin{align*}
        \mathfrak{W}_i = \begin{cases}
            S(-1)_i \otimes_R \Hom(\tilde{P}, \PSp) : & i \geq 0 \\ 
            S(1)_i \otimes_R \Hom(\tilde{P}, \PSpw) : & i < 0
        \end{cases}
    \end{align*}
    where the $S$-action is defined by the action on the left factor induced by the following transition maps between degrees $-1$ and $1$:
    \begin{align*}
        c : S(-1)_{1} \otimes_R \Hom(\tilde{P}, \PSp) &\to S(1)_{-1} \otimes_R \Hom(\tilde{P}, \PSpw) \\ 
        r \otimes \varphi &\mapsto r \otimes (c \circ \varphi) \\
        b : S(1)_{-1} \otimes_R \Hom(\tilde{P}, \PSpw) &\to S(-1)_{1} \otimes_R \Hom(\tilde{P}, \PSp) \\ 
        r \otimes \psi &\mapsto r \otimes (b \circ \psi)
    \end{align*}
for $b \in S_{2}, c \in S_{-2}, r \in R = S(1)_{-1} = S(-1)_1$, using the identifications $$S_2 \simeq \Hom(\PSpw, \PSp), S_{-2} \simeq \Hom(\PSp, \PSpw).$$ 
\end{proposition}

\begin{proof}
    We write elements of $\mathfrak{V} \simeq S(1) \oplus S(-1)$ as row vectors with $$\tilde{E}' \simeq E^{\on{op}} \simeq \begin{bmatrix}
        R & S_{-2} \\
        S_2 & R
    \end{bmatrix}$$ acting by matrix multiplication from the right, and we write elements of $$\Hom(\tilde{P}, \PSpw \oplus \PSp) \simeq \Hom(\tilde{P}, \PSpw) \oplus \Hom(\tilde{P}, \PSp)$$ as column vectors with $\tilde{E}'$ acting by matrix multiplication from the left using the identifications in Corollary \ref{cor:spw-sp-homs}.
    We have inclusions
    \begin{align}
        \label{map:inclusion-tensor} S(-1)_{i} \otimes_R \Hom(\tilde{P}, \PSp) &\hookrightarrow \mathfrak{W}_i \\
        s \otimes_R \varphi &\mapsto \begin{bmatrix}
            0 & s
        \end{bmatrix}\otimes_{\tilde{E}'} \begin{bmatrix}
            0 \\ 
            \varphi
        \end{bmatrix}
    \end{align}
    in degrees $i \geq 0$ and
    \begin{align}\label{map:inclusion-tensor-b}
    S(1)_{i} \otimes_R \Hom(\tilde{P}, \PSpw) &\hookrightarrow \mathfrak{W}_i \\
        s \otimes_R \psi &\mapsto \begin{bmatrix}
            s & 0
        \end{bmatrix}\otimes_{\tilde{E}'} \begin{bmatrix}
            \psi \\ 
            0
        \end{bmatrix}
    \end{align}
    in degrees $i < 0$. These inclusions are compatible with the action of $S$: for $c \in S_{-2}, r \in S(-1)_{1} = R, \varphi \in \Hom(\tilde{P}, \PSp)$, we have
    \begin{align*}
        c \cdot \left(\begin{bmatrix}
            0 & r
        \end{bmatrix}\otimes_{\tilde{E}'} \begin{bmatrix}
            0 \\ 
            \varphi
        \end{bmatrix}\right) &= \begin{bmatrix}
            0 & cr
        \end{bmatrix}\otimes_{\tilde{E}'} \begin{bmatrix}
            0 \\ 
            \varphi
        \end{bmatrix}\\ 
        &= \begin{bmatrix}
              r & 0
        \end{bmatrix}\begin{bmatrix}
            0 & c \\
            0 & 0
        \end{bmatrix}\otimes_{\tilde{E}'} \begin{bmatrix}
            0 \\ 
            \varphi
        \end{bmatrix}\\
        &= \begin{bmatrix}
              r & 0
        \end{bmatrix}\otimes_{\tilde{E}'} \begin{bmatrix}
            0 & c \\
            0 & 0
        \end{bmatrix}\begin{bmatrix}
            0 \\ 
            \varphi
        \end{bmatrix} \\
        &= \begin{bmatrix}
              r & 0
        \end{bmatrix}\otimes_{\tilde{E}'} \begin{bmatrix}
            c \circ \varphi \\ 
            0
        \end{bmatrix},
    \end{align*}
    which is the image of $c \cdot (r \otimes_R \varphi) = r \otimes_R (c \circ \varphi)$ under the inclusion (\ref{map:inclusion-tensor-b}), and a similar calculation shows compatibility for multiplication by $b \in S_2$ under the transition map from degree $-1$ to degree $1$. 

    We claim that these inclusions are in fact isomorphisms. Without loss of generality assume $i \geq 0$; the same argument will work for $i \leq 0$ by symmetry. An arbitrary element of $\mathfrak{W}_i$ is a sum of terms of the form
    \begin{align*}
        \begin{bmatrix}
            s_{i + 1} & s_{i - 1}
        \end{bmatrix}\otimes_{\tilde{E}'} \begin{bmatrix}
            \psi \\ 
            \varphi
        \end{bmatrix}
    \end{align*}
    with $s_j \in S_j, \psi \in \Hom(\tilde{P}, \PSpw), \varphi \in \Hom(\tilde{P}, \PSp)$, so we must show that such an element lies in the image of the map (\ref{map:inclusion-tensor}). Since $i + 1 \geq 2$, the element $s_{i + 1}$ may be written as 
    \begin{align*}
        s_{i + 1} = \sum_j b'_j q_j'
    \end{align*} for some $b'_j \in S_2$ and $q_{j}' \in S_{i - 1}$, yielding
    \begin{align*}\begin{bmatrix}
            s_{i + 1} & s_{i - 1}
        \end{bmatrix}\otimes_{\tilde{E}'} \begin{bmatrix}
            \psi \\ 
            \varphi
        \end{bmatrix} &= \sum_j\left(\begin{bmatrix}
            b_j'q_j' & 0
        \end{bmatrix}\otimes_{\tilde{E}'} \begin{bmatrix}
            \psi \\ 
            \varphi
        \end{bmatrix}\right) + \begin{bmatrix}
            0 & s_{i - 1}
        \end{bmatrix}\otimes_{\tilde{E}'} \begin{bmatrix}
            \psi \\ 
            \varphi
        \end{bmatrix}\\
        &= \sum_j \left(\begin{bmatrix}
            0 & q_j'
        \end{bmatrix} \begin{bmatrix}
            0 & 0 \\
            b_j' & 0 
        \end{bmatrix}\otimes_{\tilde{E}'} \begin{bmatrix}
            \psi \\ 
            \varphi
        \end{bmatrix}\right) + \begin{bmatrix}
            0 & s_{i - 1}
        \end{bmatrix}\begin{bmatrix}
            0 & 0 \\ 0 & 1
        \end{bmatrix}\otimes_{\tilde{E}'} \begin{bmatrix}
            \psi \\ 
            \varphi
        \end{bmatrix}\\
        &= \sum_j\left(\begin{bmatrix}
            0 & q_j'
        \end{bmatrix} \otimes_{\tilde{E}'} \begin{bmatrix}
            0 & 0 \\
            b_j' & 0 
        \end{bmatrix}\begin{bmatrix}
            \psi \\ 
            \varphi
        \end{bmatrix}\right) + \begin{bmatrix}
            0 & s_{i - 1}
        \end{bmatrix}\otimes_{\tilde{E}'} \begin{bmatrix}
            0 & 0 \\ 0 & 1
        \end{bmatrix}\begin{bmatrix}
            \psi \\ 
            \varphi
        \end{bmatrix}\\
        &= \sum_j\left(\begin{bmatrix}
            0 & q_j'
        \end{bmatrix} \otimes_{\tilde{E}'} \begin{bmatrix}
            0 \\ 
            b_j' \circ \psi
        \end{bmatrix}\right) + \begin{bmatrix}
            0 & s_{i - 1}
        \end{bmatrix}\otimes_{\tilde{E}'} \begin{bmatrix}
            0 \\ 
            \varphi
        \end{bmatrix},
        \end{align*}
        and the terms in the final expression are visibly in the image of the map (\ref{map:inclusion-tensor}). 
\end{proof}

\begin{corollary}\label{cor:compatible-diagram-hom}
    If $M$ is a graded $S$-module, then the $R$-module $\Hom_{\mathfrak{X}}(\mathfrak{W}, M)$ is naturally isomorphic to the $R$-module consisting of pairs of $R$-module homomorphisms $x: \Hom(\tilde{P}, \PSp) \to M_1, y: \Hom(\tilde{P}, \PSpw) \to M_{-1}$
    making the following squares commute for every choice of $b \in S_{2}, c \in S_{-2}$:
    \begin{center}
        \begin{tikzcd}
            \Hom(\tilde{P}, \PSp) \ar[d, "c"]\ar[r, "x"] & M_1 \ar[d, "c"] \\
            \Hom(\tilde{P}, \PSpw) \ar[r, "y"]  & M_{-1} 
        \end{tikzcd}
    \end{center}

    \begin{center}
        \begin{tikzcd}
            \Hom(\tilde{P}, \PSp) \ar[r, "x"] & M_1  \\
            \Hom(\tilde{P}, \PSpw) \ar[r, "y"] \ar[u, "b"] & M_{-1} \ar[u, "b"]
        \end{tikzcd}
    \end{center}
    The same holds for $\Hom_{\mathfrak{X}}(\mathfrak{W}^\pi, M)$ if we replace $\tilde{P}$ with $\Ppi$.
\end{corollary}
\begin{remark}
    From now on, we will write such collections of commutative squares more compactly as
    \begin{center}
        \begin{tikzcd}
            \Hom(\tilde{P}, \PSp) \ar[d, bend left, "c"]\ar[r, "x"] & M_1 \ar[d, bend left, "c"] \\
            \Hom(\tilde{P}, \PSpw) \ar[r, "y"] \ar[u, bend left, "b"] & M_{-1} \ar[u, bend left, "b"]
        \end{tikzcd}
    \end{center}
    with $x$ and $y$ fixed but $b$ and $c$ ranging over all $b \in S_2, c \in S_{-2}$. 
\end{remark}
\begin{proof}
    By Proposition \ref{prop:categorical-interpretation}, any $S$-graded homomorphism $\mathfrak{W} \to M$ yields such a pair $x, y$ by restricting to degrees $1$ and $-1$; commutativity of the diagrams follows from the transition law between these degrees as described in the proposition. 
    
    Conversely, given a pair $x, y$ as stated, we may define a homomorphism $\mathfrak{W} \to M$ given by $x$ in degree $1$ and $y$ in degree $-1$ and then extend $S$-linearly to the rest of $\mathfrak{W}$, since Proposition \ref{prop:categorical-interpretation} implies $\mathfrak{W}$ is generated over $S$ in degrees $1$ and $-1$. Commutativity of the given squares ensures that this extension is well-defined. 

    The claim for $\mathfrak{W}^\pi$ follows by restricting to the appropriate component of $\mathfrak{W}$. 
\end{proof}

\begin{corollary}\label{cor:endw=endwtf}
    $\End_{\mathfrak{X}}(\mathfrak{W}) \simeq \End_{\mathfrak{X}}(\mathfrak{W}_{\on{tf}})$.
\end{corollary}
\begin{proof}
    Every endomorphism of $\mathfrak{W}$ descends to $\mathfrak{W}_{\on{tf}}$ because the submodule of $R$-torsion elements is invariant under any endomorphism of $\mathfrak{W}$. By Corollary \ref{cor:compatible-diagram-hom}, every endomorphism is uniquely determined by its restriction to $\mathfrak{W}_{\pm 1}$. Since $$\Hom(\tilde{P}, \tilde{P}') = \Hom(\tilde{P}, \PSp) \oplus \Hom(\tilde{P}, \PSpw)$$ is $R$-torsion-free by Proposition \ref{prop:homs-are-ideals}, both $\mathfrak{W}_{1}$ and $\mathfrak{W}_{-1}$ are also $R$-torsion-free since by Proposition \ref{prop:categorical-interpretation} these modules are
    \begin{align*}
        \mathfrak{W}_1 \simeq S(-1)_{1} \otimes_R \Hom(\tilde{P}, \PSpw) \simeq R \otimes_R \Hom(\tilde{P}, \PSpw) \simeq \Hom(\tilde{P}, \PSpw)
    \end{align*}
    and likewise $\mathfrak{W}_{-1} \simeq \Hom(\tilde{P}, \PSp)$.
    Hence the restriction of the quotient $\mathfrak{W} \to \mathfrak{W}_{\on{tf}}$ to degrees $\pm 1$ is an isomorphism. By Corollary \ref{cor:compatible-diagram-hom}, every endomorphism of $\mathfrak{W}$ is uniquely determined by its restriction to $\mathfrak{W}_{\pm 1}$, so we conclude that the descent map $\End_{\mathfrak{X}}(\mathfrak{W}) \to \End_{\mathfrak{X}}(\mathfrak{W}_{\on{tf}})$ is injective. Conversely, $\mathfrak{W}_{\on{tf}}$ is also generated over $S$ in degrees $\pm 1$, so any endomorphism of $\mathfrak{W}_{\on{tf}}$ is determined by its restriction to these degrees. Hence $\End_{\mathfrak{X}}(\mathfrak{W}) \to \End_{\mathfrak{X}}(\mathfrak{W}_{\on{tf}})$ is surjective and hence an isomorphism. 
\end{proof}

Let $\mathfrak{a}$ denote the kernel of $R \simeq \End(\Ptriv) \to \End(\Ptriv, \O) \simeq \O$.

\begin{lemma}\label{lem:depth-extension}
    Let $M$ be an $R$-torsion-free finitely generated $R$-module, and let $\mathfrak{a} \to M$ be an $R$-module homomorphism. If $M$ has $R$-depth at least $3$, then this homomorphism extends uniquely to a homomorphism $R \to M$. 
\end{lemma}
\begin{proof}
    The short exact sequence
    \begin{align*}
        0 \to \mathfrak{a} \to R \to \O \to 0
    \end{align*} induces the long exact sequence
    \begin{align*}
            0 \to \Hom_R(\O, M) \to  \Hom_R(R, M) \to \Hom_R(\mathfrak{a}, M) \to \Ext^1_R(\O, M).
    \end{align*}
    We wish to show that the arrow $\Hom_R(R, M) \to \Hom_R(\mathfrak{a}, M)$ is an isomorphism. Since $\O$ is $R$-torsion, we have $\Hom_R(\O, M) = 0$, so to prove the claim it suffices to show that $\Ext^1_R(\O, M) = 0$. The short exact sequence
    \begin{center}
        \begin{tikzcd}
            0 \ar[r] & \O \ar[r, "\cdot \varpi"] \ar[r] & \O \ar[r] & k \ar[r] & 0
        \end{tikzcd}
    \end{center}
    induces the exact sequence
    \begin{center}
        \begin{tikzcd}
            \Ext^1_R(k, M) \ar[r] & \Ext^1_R(\O, M) \ar[r, "\cdot \varpi"] & \Ext^1_R(\O, M) \ar[r] & \Ext^2_R(k, M).
        \end{tikzcd}
    \end{center}
    Since $M$ has depth at least $3$, we have $\Ext^1_R(k, M) = \Ext^2_R(k, M) = 0$, so multiplication by $\varpi$ is an automorphism on $\Ext^1_R(\O, M)$. But this is a finitely generated $R$-module, so by Nakayama's lemma we conclude that $\Ext^1_R(\O, M) = 0$, as desired. 
\end{proof}

For $\pi, \tau \in \mathfrak{B}$, there is a natural map $\Hom(\Ppi, \Ptau) \to \Hom_{\mathfrak{X}}(\mathfrak{W}^\tau, \mathfrak{W}^\pi)$ induced by precomposition on the second factor of the tensor product: an element $\varphi \in \Hom(\Ppi, \Ptau)$ acts by
\begin{align*}
    \varphi: \mathfrak{V} \otimes_{E'} \Hom(\Ptau, \tilde{P}') &\to \mathfrak{V} \otimes_{E'} \Hom(\Ppi, \tilde{P}')\\
        v \otimes \psi \mapsto v \otimes (\psi \circ \varphi).
\end{align*}

\begin{proposition}\label{prop:hom-iso}
    The natural map $\Hom(\Ppi, \Ptau) \to  \Hom_{\mathfrak{X}}(\mathfrak{W}^\tau, \mathfrak{W}^\pi)$ is an isomorphism.
\end{proposition}

\begin{proof}
    By Corollary \ref{cor:compatible-diagram-hom}, the module $\Hom_{\mathfrak{X}}(\mathfrak{W}^\tau, \mathfrak{W}^\pi)$ is isomorphic to the set of pairs of morphisms $x, y$ making the following diagrams commute for all choices of $b \in S_2, c \in S_{-2}$: 
    \begin{center}
        \begin{tikzcd}
            \Hom(\Ptau, \PSpw) \ar[r, "x"]\ar[d, "c", bend left] & \Hom(\Ppi, \PSpw)\ar[d, "c", bend left] \\ 
            \Hom(\Ptau, \PSp) \ar[u, "b", bend left]\ar[r, "y"] & \Hom(\Ppi, \PSp) \ar[u, "b", bend left]
        \end{tikzcd}
    \end{center}
    Using Proposition \ref{cor:embed-matrix} and the subsequent Remark \ref{rem:rigidity}, we may fix an embedding $\tilde{E}_\mathfrak{B} \hookrightarrow M_{4 \times 4}(K)$ to identify each Hom module in the diagram with an $R$-submodule of $K$ and the morphisms $b, c, x, y$ with multiplication by elements of $K$. (Recall $K := \on{Frac} R$.) In particular, commutativity of the diagram implies that $x$ and $y$ are correspond to multiplication by the \emph{same} element of $K$. To reiterate, the key points are: 
    \begin{itemize}
        \item Each $\Hom(\Ppi, \Ptau)$ is isomorphic to an  $R$-submodule of $K$;
        \item Morphisms between such modules may be identified with multiplication by elements of $K$; and
        \item The commutativity condition for the above diagram is equivalent to equality of $x$ and $y$ as elements of $K$. 
    \end{itemize}

   To prove the proposition, we need to show that the morphisms $x$ and $y$ are both induced by precomposition by some morphism $\varphi: \Ppi \to \Ptau$ (necessarily the same morphism for both $x$ and $y$). Such a morphism is necessarily unique, since distinct elements $\varphi, \psi$ of $\Hom(\Ppi, \Ptau)$ act by multiplication by different elements $x_{\varphi}, x_{\psi}$ of $K$ under the matrix embedding. The claim follows if we can show that the diagram extends to a diagram of the following form:
    \begin{center}
        \begin{tikzcd}
            \Hom(\Ptau, \PSpw) \ar[r, "x"]\ar[dd, "v", curve={height=80pt}]\ar[d, "c", bend left] & \Hom(\Ppi, \PSpw)\ar[d, "c", bend left] \ar[dd, "v",curve={height=-80pt} ] \\ 
            \Hom(\Ptau, \PSp) \ar[d, "w"] \ar[u, "b", bend left]\ar[r, "y"] & \Hom(\Ppi, \PSp) \ar[u, "b", bend left]\ar[d, "w"] \\
            \End(\Ptau) \ar[r, dashed, "f"] & \Hom(\Ppi, \Ptau)
        \end{tikzcd}
    \end{center}
    where $v, w$ are fixed and induced by some nonzero maps $\PSpw \to \Ptau, \PSp \to \Ptau$, respectively, and where all squares formed in the diagram commute (but the vertical subdiagrams need not commute, i.e.\ we do not require $v = wc$ or $w = vb$). Indeed, if such $f$ exists, then under the matrix embedding the maps $x, y,$ and $f$ must all be given by multiplication by the same element in $K$. The map $f$ sends the identity morphism on $\Ptau$ to an element $\varphi \in \Hom(\Ppi, \Ptau)$, and so by Remark \ref{rem:rigidity} precomposition by $\varphi$ corresponds to multiplication by the same element as $x, y,$ and $f$. That is, both $x$ and $y$ are induced by precomposition by $\varphi$, as desired.

    For $\tau = \Sp$ or $\Sp \otimes \omega$, the existence of such an extension is trivial, since this is already achieved by the original diagram, taking $v$ or $w$ to be an identity map. If we can prove the claim for $\Ptau = \Ptriv$, the claim for $\Pw$ also follows by twisting. Therefore set $\tau = 1$. We choose $v$ to be the morphism induced by postcomposition with $\alpha$ and $w$ to be induced by postcomposition with $\beta$, where $\alpha$ and $\beta$ are defined in Diagram \ref{seq:main}. The desired diagram is:
        \begin{center}
        \begin{tikzcd}
            \Hom(\Ptriv, \PSpw) \ar[r, "x"]\ar[dd, "\alpha", curve={height=80pt}]\ar[d, "c", bend left] & \Hom(\Ppi, \PSpw)\ar[d, "c", bend left] \ar[dd, "\alpha",curve={height=-80pt} ] \\ 
            \Hom(\Ptriv, \PSp) \ar[d, "\beta"] \ar[u, "b", bend left]\ar[r, "y"] & \Hom(\Ppi, \PSp) \ar[u, "b", bend left]\ar[d, "\beta"] \\
            \End(\Ptriv) \ar[r, dashed, "f"] & \Hom(\Ppi, \Ptriv)
        \end{tikzcd}
    \end{center}

    The image of $(\alpha, \beta): \Hom(\Ptriv, \PSpw \oplus \PSp) \to \End(\Ptriv) \simeq R$ is the ideal $\mathfrak{a}$, which follows immediately from Proposition \ref{prop:alpha-beta-image}. We obtain a homomorphism $f': \mathfrak{a} \to \Hom(\Ppi, \Ptau)$ making the squares of the diagram commute: given $a \in \mathfrak{a}$, let $(s, t) \in \Hom(\Ptriv, \PSpw) \oplus \Hom(\Ptriv, \PSp)$ map to $a$ under $(\alpha, \beta)$. Then we define $f'(a) = (\alpha \circ x)(s) + (\beta \circ y)(t)$. This is well-defined: if $(s', t')$ is another pair mapping to $a$, then, again identifying every element and morphism in sight with an element of $K$, since $x$ and $y$ correspond to the same element of $K$ we have
    \begin{align*}
        \alpha x s + \beta y t &= \alpha x s + \beta x t \\
        &= x(\alpha s + \beta t) \\
        &= x(\alpha s' + \beta t')\\
        &= \alpha x s' + \beta y t'.
    \end{align*}
    Hence, the proof reduces to whether $f': \mathfrak{a} \to \Hom(\Ppi, \Ptriv)$ extends to a morphism $f: \End(\Ptriv) \to \Hom(\Ppi, \Ptriv)$. But the unique existence of such an extension is guaranteed by Proposition \ref{prop:Hom-depths} and Lemma \ref{lem:depth-extension}, so we are done.  
\end{proof}

\begin{remark}
    The main idea behind Proposition \ref{prop:hom-iso} is that the diagrams in Corollary \ref{cor:compatible-diagram-hom} resemble some of the naturality diagrams that one would obtain from the contravariant Yoneda embedding. The point is to show that, in this case, such diagrams do indeed come from the Yoneda embedding, i.e.\ that these commutative squares are induced by precomposition with a specific morphism. The proof of the Yoneda lemma was inspiration to examine the extension of the diagram in the proof of Proposition \ref{prop:hom-iso}, since the diagram is controlled by the image of the identity under $\End(\Ptau) \to \Hom(\Ppi, \Ptau)$.
\end{remark}

\begin{theorem}\label{thm:endomorphism-algebra-isomorphism}
    The natural map $\tilde{E}_{\mathfrak{B}}^{\on{op}} \to \End_{\mathfrak{X}}(\mathfrak{W}) \simeq \End_{\mathfrak{X}}(\mathfrak{W}_{\on{tf}})$ is a continuous isomorphism of $R$-algebras.
\end{theorem}
\begin{proof}
    This follows from assembling the isomorphisms from Proposition \ref{prop:hom-iso} using the respective $R$-module direct sum decompositions of $\tilde{E}_{\mathfrak{B}}^{\on{op}}$ and $\End_{\mathfrak{X}}(\mathfrak{W})$. Recall that we have the isomorphism $\End_{\mathfrak{X}}(\mathfrak{W}) \simeq \End_{\mathfrak{X}}(\mathfrak{W}_{\on{tf}})$ from Corollary \ref{cor:endw=endwtf}.
\end{proof}

\subsection{Vanishing of higher Exts}\label{sec:ext-computations}

Although taking the $R$-torsion-free quotient of $\mathfrak{W}$ does not affect its endomorphism algebra, this step is important to ensure that the higher self-Exts of our candidate vanish, as we will demonstrate in this section. 

Our approach will be very similar to that taken in \cite{JNW} once we have made the appropriate identifications for the $\mathfrak{W}^{\pi}_{\on{tf}}$. We require one more preparatory identification:

\begin{lemma}\label{lem:hom-ptriv-pspw-reducibility}
    Identifying $\Hom(\Ptriv, \PSpw)$ as an ideal of $R \simeq \End(\Ptriv)$ via \ref{prop:hom-identifications}(a), the ideal $\Hom(\Ptriv, \PSpw)$ is equal to the reducibility ideal $I$ in $R$, given by the image of $S_2 \otimes_R S_{-2} \to R$. Moreover, the natural map 
    $$\Hom(\Ptriv, \PSp) \otimes_R \Hom(\PSp, \PSpw) \to \Hom(\Ptriv, \PSpw)$$
    is a surjection. 
\end{lemma}
\begin{proof}
    From the description of $S$ given in \S\ref{sec:explicit-S}, the reducibility ideal $$I = (b_0c_0, b_0c_1, b_1c_0, b_1c_1)$$ satisfies $$R/I \simeq \O\lb a_0, a_1\rb.$$ Meanwhile, we know from Diagram \ref{seq:main} and Lemma \ref{lem:1-mone} that $$\End(\Ptriv)/\Hom(\Ptriv, \PSpw) \simeq \Hom(\Ptriv, \Mone) \simeq \O \lb x, y \rb$$ (abstractly as a continuous $R$-module). 

    Using the identification from \ref{cor:spw-sp-homs}, we know that the image of $$\Hom(\PSpw, \PSp) \otimes_R \Hom(\PSp, \PSpw) \to \End(\PSpw) = R$$ is the reducibility ideal $I$. By Lemma \ref{lem:Hom-1-Sp}, we deduce that the image of \begin{align}\label{map:ptriv-psp-tensor-psp-pspw}\Hom(\Ptriv, \PSp) \otimes_R \Hom(\PSp, \PSpw) \to \Hom(\Ptriv, \PSpw)\end{align} is also equal to $I$ under the identifications $$\Hom(\Ptriv, \PSpw) \hookrightarrow \End(\Ptriv) \simeq \End(\PSpw).$$ Hence $I \subseteq \Hom(\Ptriv, \PSpw)$, inducing a quotient map $$R/I \twoheadrightarrow \End(\Ptriv)/\Hom(\Ptriv, \PSpw).$$ But both of these modules are abstractly isomorphic to $\O\lb a_0, a_1\rb$ as continuous $R$-modules. No proper quotient of $\O\lb a_0, a_1\rb$ is isomorphic to $\O\lb a_0, a_1\rb$, since the dimension of this module would be lowered by at least one under any proper quotient. Hence we conclude that we must actually have equality $I = \Hom(\Ptriv, \PSpw)$ and the map (\ref{map:ptriv-psp-tensor-psp-pspw}) is surjective. 
\end{proof}
\begin{proposition}\label{prop:wtf-identifications}
As graded $S$-modules:
    \begin{enumerate}[(a)]
        \item $\mathfrak{W}_{\on{tf}}^{\Sp \otimes \omega} \simeq S(1)$;
        \item $\mathfrak{W}^{\Sp}_{\on{tf}} \simeq S(-1)$;
        \item $\mathfrak{W}^{1}_{\on{tf}} \simeq (b_0, b_1)(1)$; and 
        \item $\mathfrak{W}^\omega_{\on{tf}} \simeq (c_0, c_1)(-1)$,
    \end{enumerate}
    using the presentation of $S$ from \S\ref{sec:explicit-S}. 
\end{proposition}
\begin{proof}
    We prove this for $\mathfrak{W}_{\on{tf}}^{\Sp \otimes \omega}$ and $\mathfrak{W}_{\on{tf}}^{1}$; the other two cases follow by twisting. 

    The description from Proposition \ref{prop:categorical-interpretation} yields
    \begin{align*}
        \mathfrak{W}_{i}^{\Sp \otimes \omega} &= \begin{cases}
            S_{i - 1} \otimes_R \Hom(\PSpw, \PSp) : & i \geq 0 \\
            S_{i + 1} \otimes_R \Hom(\PSpw, \PSpw) : & i < 0 
        \end{cases}\\
        &= \begin{cases}
            S_{i - 1} \otimes_R S_2 : & i \geq 0 \\
            S_{i + 1} \otimes_R R : & i < 0 
        \end{cases}
    \end{align*}
    We deduce that we have a surjective map of graded $S$-modules $\mathfrak{W}^{\Sp \otimes \omega} \to S(1)$, so we may conclude part (a) if we can show that the kernel of this map is $R$-torsion. But every $S_{i}$ and $\Hom(\Ppi,\Ptau)$ is a rank $1$ $R$-module, so $\mathfrak{W}^{\Sp \otimes \omega} \to S(1)$ becomes an isomorphism after passing to the generic fiber of $R$. Hence the kernel consists of precisely the $R$-torsion. 

    Likewise, we have a description
    \begin{align*}
        \mathfrak{W}_{i}^{1} &= \begin{cases}
            S_{i - 1} \otimes_R \Hom(\Ptriv, \PSp) : & i \geq 0 \\
            S_{i + 1} \otimes_R \Hom(\Ptriv, \PSpw) : & i < 0 
        \end{cases}\\
        &= \begin{cases}
            S_{i - 1} \otimes_R S_2 : & i \geq 0 \\
            S_{i + 1} \otimes_R I : & i < 0 
        \end{cases}
    \end{align*}
    where $I$ is the reducibility ideal of $R$, using Lemmas \ref{lem:Hom-1-Sp} and \ref{lem:hom-ptriv-pspw-reducibility}. We therefore get a surjective map $\mathfrak{W}^1_i \to (b_0, b_1)(1)$, and the same argument as in the previous case shows that the kernel of this map is precisely the $R$-torsion. 
\end{proof}
\begin{corollary}\label{cor:wtf-presentation}
    As graded $S$-modules,
    \begin{align*}
        \mathfrak{W}_{\on{tf}} \simeq S(1) \oplus S(-1) \oplus (b_0, b_1)(1) \oplus (c_0, c_1)(-1).
    \end{align*}
\end{corollary}

To prove that $\Ext^i_{\mathfrak{X}}(\mathfrak{W}_{\on{tf}}, \mathfrak{W}_{\on{tf}}) = 0$ for all $i > 0$, we show that each $\Ext_{\mathfrak{X}}^i(\mathfrak{W}_{\on{tf}}^{\pi}, \mathfrak{W}_{\on{tf}}^{\tau})$ vanishes. Since $\mathfrak{W}_{\on{tf}}^{\pi}$ is projective\footnote{See \cite[Theorem 3.2.3, Theorem 3.5.1]{JNW}.} when $\pi = \on{Sp} \otimes \omega$ or $\Sp$, vanishing is automatic in these cases. We therefore focus on $\pi = 1$; the vanishing for $\pi = \omega$ follows from this case by twisting.  
\begin{definition}
    Let $Q^*$ denote the image of
    \begin{center}
        \begin{tikzcd}[ampersand replacement=\&, row sep=large, column sep=large]
            S(-1)^{\oplus 2} \ar[r, "{\begin{bmatrix} x_1 & x_0 \\ -b_1 & b_0\end{bmatrix}}"] \& S(-1) \oplus S(1).
        \end{tikzcd}
    \end{center}
\end{definition}
\begin{proposition}\label{prop:q*-is-b0b1}
    The projection $Q^* \to S(1)$ onto the second factor identifies $Q^*$ with $(b_0, b_1)(1) \simeq \mathfrak{W}_{\on{tf}}^1$. 
\end{proposition}
\begin{proof}
    We argue exactly as in \cite[Lemma 3.6.41]{DEG26}. The projection $Q^* \to S(1)$ has image $(b_0, b_1)(1)$, so we must show that this projection is injective. If $s_0, s_1 \in S$, then we have
    \begin{align*}
        (s_0x_0 + s_1x_1)b_0 = x_0( s_0b_0- s_1b_1).
    \end{align*}
    Therefore if the projection of
    \begin{align}\label{exp:linear-combo}
        s_1\begin{bmatrix}
            x_1 \\ -b_1
        \end{bmatrix} +
        s_0 \begin{bmatrix}x_0 \\ b_0\end{bmatrix}
    \end{align}
    to the second factor is $0 = s_0b_0- s_1b_1$, we must have $(s_0x_0 + s_1x_1)b_0 = 0$ too. But $S$ is an integral domain by Proposition \ref{prop:alg-properties}, so this implies $s_0x_0 + s_1x_1 = 0$ too, so the linear combination (\ref{exp:linear-combo}) is the zero vector. Hence the projection is injective. 
\end{proof}

We note that the quotient $$\frac{\O\lb a_0, a_1, b_0c_0, b_0c_1, b_1c_0, b_1c_1\rb [b_0, b_1, c_0, c_1]}{(y_0c_1 + y_1c_0)}$$ is an integral domain, since (using the approximations of $y_0$ and $y_1$ from \S\ref{sec:explicit-S}) we may change variables to see that this ring is isomorphic to 
$$\frac{\O\lb a_0', a_1', b_0c_0, b_0c_1, b_1c_0, b_1c_1\rb [b_0, b_1, c_0, c_1]}{(a_0'c_1 + a_1'c_0)}$$
which may be checked directly to be a domain. Therefore, we may argue as in \cite[\S3.5]{JNW} to conclude that we have a periodic resolution
\begin{center}
    \begin{tikzcd}
     \dotsb \ar[r, "M^t"] & S(-3)^{\oplus 2} \ar[r, "N^t"] & S(-3) \oplus S(-1) \ar[r, "M^t"] & S(-1)^{\oplus 2} \ar[r, "N^t"] & Q^* \ar[r] & 0
    \end{tikzcd}
\end{center}
where 
\begin{align*}
    M^t := \begin{bmatrix}
        b_0 & -x_0 \\
        b_1 & x_1
    \end{bmatrix}, N^t := \begin{bmatrix}
        x_1 & x_0 \\ -b_1 & b_0
    \end{bmatrix}
\end{align*}
and the surjection $S(-1)^{\oplus 2} \to Q^*$ comes from identifying $Q^*$ as the image of $N^t$ inside $S(-1) \oplus S(1)$ as in Proposition \ref{prop:q*-is-b0b1}. (We write the matrices with the transposition to emphasize the analogy to the dual of the situation in \cite{JNW}.)

\begin{proposition}\label{prop:ext-vanishing-1-spw-sp}
    $\Ext^i_{\mathfrak{X}}(\mathfrak{W}_{\on{tf}}^{1}, \mathfrak{W}_{\on{tf}}^{\Sp \otimes \omega}) = \Ext^i_{\mathfrak{X}}(\mathfrak{W}_{\on{tf}}^{1}, \mathfrak{W}_{\on{tf}}^{\Sp}) = 0$. 
\end{proposition}
\begin{proof}
    The same argument as in \cite[Proposition 3.5.2]{JNW} works. Applying $\Hom_S(-, S)$ to resolution of $Q^* \simeq (b_0, b_1)(1) \simeq \mathfrak{W}_{\on{tf}}^{1}$ yields
    \begin{center}
    \begin{tikzcd}
     \dotsb \ar[r, "M"] & S(-3)^{\oplus 2} \ar[r, "N"] & S(-3) \oplus S(-1) \ar[r, "M"] & S(-1)^{\oplus 2} 
    \end{tikzcd}
\end{center}
which, by the same reasoning as the original resolution, is exact in nonzero degrees, so in the category of ungraded $S$-modules we have $\Ext_S^i(\mathfrak{W}_{\on{tf}}^{1}, S) = 0$ for all $i > 0$, and hence 
\begin{align*}
    \Ext_{\mathfrak{X}}^i(\mathfrak{W}_{\on{tf}}^{1}, S(1)) = \Ext_{\mathfrak{X}}^i(\mathfrak{W}_{\on{tf}}^{1}, S(-1)) = 0,
\end{align*}
which gives the desired result by Proposition \ref{prop:wtf-identifications}. 
\end{proof}

\begin{corollary}\label{cor:wtf-is-mcm}
    $\mathfrak{W}_{\on{tf}}$ is a maximal Cohen-Macaulay sheaf on $\mathfrak{X}$. 
\end{corollary}
\begin{proof}
    In the previous proof, we showed that $\Ext^i_S(\mathfrak{W}_{\on{tf}}, S) = 0$ for all $i > 0$. The ring $S$ is a local complete intersection by Proposition \ref{prop:alg-properties}(b), hence Gorenstein, so vanishing of $\Ext^i_S(\mathfrak{W}_{\on{tf}}, S)$ for all $i > 0$ is equivalent to $\mathfrak{W}_{\on{tf}}$ being maximal Cohen-Macaulay. 
\end{proof} 

\begin{proposition}
    $\Ext^i_{\mathfrak{X}}(\mathfrak{W}_{\on{tf}}^{1}, \mathfrak{W}_{\on{tf}}^{1}) = 0$. 
\end{proposition}
\begin{proof}
    Equivalently, we need to show that the higher self-Exts of $Q^*$ vanish. Let $\conj{Q}^*$ be defined by the exact sequence
    \begin{align*}
        0 \to (b_0, b_1)(1) \to S(1) \to \conj{Q}^*
    \end{align*}
    where the first map is the natural inclusion. Since $Q^* = (b_0, b_1)(1) \hookrightarrow S(1)$ is an isomorphism in degree $i > 0$, we conclude that $\conj{Q}^*_{i} = 0$ for all $i > 0$ and hence that $\Hom_{\mathfrak{X}}(S(i), \conj{Q}^*) = 0$ for all $i < 0$. Moreover, since $(b_0, b_1)(1)$ is also generated in degree $1$, we also have $\Hom_{\mathfrak{X}}(Q^*, \conj{Q}^*) = 0$.  Therefore, applying $\Hom_{\mathfrak{X}}(Q^*, -)$ to the exact sequence defining $\conj{Q}^*$ yields exact sequences
    \begin{align*}
        \Ext_{\mathfrak{X}}^{i}(Q^*,  S(1)) \to \Ext_{\mathfrak{X}}^i(Q^*,  \conj{Q}^*) \to \Ext^{i + 1}_{\mathfrak{X}}(Q^*, Q^*) \to \Ext_{\mathfrak{X}}^{i + 1}( Q^*, S(1))
    \end{align*}
    for $i \geq 0$. 
    
    We have $\Hom_{\mathfrak{X}}(Q^*, \conj{Q}^*) = 0$ and $\Ext_{\mathfrak{X}}^{1}( Q^*, S(1)) = 0$ using Proposition \ref{prop:ext-vanishing-1-spw-sp}, so we conclude $\Ext^1_{\mathfrak{X}}(Q^*, Q^*) = 0$ too. For $i > 0$, vanishing of $\Ext_{\mathfrak{X}}^{i}(Q^*, S(1))$ shows that we have an isomorphism $\Ext_{\mathfrak{X}}^i(Q^*,  \conj{Q}^*) \to \Ext^{i + 1}_{\mathfrak{X}}(Q^*, Q^*)$, so it suffices to prove that $\Ext_{\mathfrak{X}}^i(Q^*,  \conj{Q}^*) = 0$ for all $i > 0$. Since $\Hom_{\mathfrak{X}}(S(i), \conj{Q}^*) = 0$ for all $i < 0$, applying $\Hom_{\mathfrak{X}}(-, \conj{Q}^*)$ to the periodic resolution of $Q^*$ yields the zero complex, so we conclude that $\Ext^i_{\mathfrak{X}}(Q^*, \conj{Q}^*) = 0$ for all $i > 0$ (and also $i = 0$), as desired. 
\end{proof}

\begin{proposition}
    $\Ext^i_{\mathfrak{X}}(\mathfrak{W}_{\on{tf}}^{1}, \mathfrak{W}_{\on{tf}}^{\omega}) = 0$.
\end{proposition}
\begin{proof}
    Analogously to our identifications $\mathfrak{W}_{\on{tf}}^{1} \simeq (b_0, b_1)(1) \simeq Q^*$ in Proposition \ref{prop:q*-is-b0b1}, we may identify $\mathfrak{W}^{\omega}_{\on{tf}} \simeq (c_0, c_1)(-1)$ with the image of
    \begin{center}
        \begin{tikzcd}[ampersand replacement=\&, row sep=large, column sep=large]
            S(1)^{\oplus 2} \ar[r, "{\begin{bmatrix} y_1 & y_0 \\ -c_1 & c_0\end{bmatrix}}"] \& S(1) \oplus S(-1).
        \end{tikzcd}
    \end{center}
    We denote this image by $Q^{*}_{\omega}$. Likewise, we define $\conj{Q}^*_\omega$ via the exact sequence
    \begin{align*}
        0 \to (c_0, c_1)(-1) \to S(-1) \to \conj{Q}^*_\omega \to 0.
    \end{align*}
    We have $\conj{Q}^*_{\omega, i} = 0$ for $i < 0$, and $\conj{Q}^*_{\omega, i} = S_{i - 1}/IS_{i - 1}$ for $i \geq 0$, where $I = (b_0c_0, b_0c_1, b_1c_0, b_1c_1)$ is the reducibility ideal of $R$. We apply $\Hom(-, \conj{Q}_\omega^*)$ to the resolution of $Q^*$: 
    \begin{center}
    \begin{tikzcd}
        \Hom_{\mathfrak{X}}(S(-1)^{\oplus 2}, \conj{Q}^*_\omega) \ar[r]  & \Hom_{\mathfrak{X}}(S(-3) \oplus S(-1), \conj{Q}^*_\omega) \ar[r] & \Hom_{\mathfrak{X}}(S(-3)^{\oplus 2}, \conj{Q}^*_\omega) \ar[r] & \dotsb
    \end{tikzcd}
    \end{center}
    Since $\conj{Q}_{\omega, i}^* = S_{i - 1}/IS_{i - 1}$ in degrees $i \geq 0$ and $S(-i)$ is generated by $S(-i)_i \simeq R$,  we have \begin{align*}\Hom_{\mathfrak{X}}(S(-i), \conj{Q}_{\omega}^*) &\simeq \Hom_{\mathfrak{X}}(S(-i), S(-1)/IS(-1))\\
    &\simeq \Hom_{\mathfrak{X}}(S(-i)/IS(-i), S(-1)/IS(-1)) \\
    &\simeq S_{i - 1}/IS_{i - 1},
    \end{align*}
    for $i \geq 0$, so the previous complex of $R$-modules is isomorphic to
    \begin{align*}
        (S_{0}/IS_{0})^{\oplus 2} \to S_{0}/IS_{0} \oplus S_{2}/IS_{2} \to (S_{2}/IS_{2})^{\oplus 2} \to \dotsb
    \end{align*}
    where the maps alternate between $\conj{M} = \begin{bmatrix}\conj{b}_0 & \conj{b}_1 \\ -\conj{x}_0 & \conj{x}_1\end{bmatrix}$ and $\conj{N} = \begin{bmatrix}\conj{x}_1 & -\conj{b}_1 \\ -\conj{x}_0 & \conj{b}_0\end{bmatrix}$, with the bars denoting residues mod $I$. But this is again a $2$-periodic resolution of the type discussed in \cite[\S3.5]{JNW}, so it is exact in positive cohomological degree, so we conclude that $\Ext_{\mathfrak{X}}^i(Q^*, \conj{Q}^*_\omega) = 0$ for $i > 0$.

    Applying $\Hom_{\mathfrak{X}}(Q^*, -)$ to the short exact sequence defining $\conj{Q}_\omega^*$ yields exact sequences
    \begin{align*}
        \Ext_{\mathfrak{X}}^{i}(Q^*,  S(-1)) \to \Ext_{\mathfrak{X}}^i(Q^*,  \conj{Q}_\omega^*) \to \Ext^{i + 1}_{\mathfrak{X}}(Q^*, Q_\omega^*) \to \Ext_{\mathfrak{X}}^{i + 1}( Q^*, S(-1))
    \end{align*}
    for $i \geq 0$. If $i > 0$, then the outer terms vanish, so vanishing of $\Ext_{\mathfrak{X}}^i(Q^*, \conj{Q}^*_\omega) = 0$ for $i > 0$ implies vanishing of $\Ext_{\mathfrak{X}}^i(Q^*, {Q}^*_\omega) = 0$ for $i > 0$ for $i > 1$. 
    
    It remains to prove $\Ext_{\mathfrak{X}}^1(Q^*, {Q}^*_\omega) = 0$. This will follow from the exact sequence above if we can show that $\Hom_{\mathfrak{X}}(Q^*, S(-1)) \to \Hom_{\mathfrak{X}}(Q^*, \conj{Q}^*_\omega)$ is surjective. A map $Q^* \to \conj{Q}_\omega^*$ is uniquely determined by the restriction $Q^*_{1} \to \conj{Q}^*_{\omega, 1}$, which we identify with a map $S_{2} \to R/I$, so the question is whether every $R$-module homomorphism $S_2 \to R/I$---equivalently, a homomorphism $S_2/IS_2 \to R/I$---lifts to an $R$-module homomorphism $S_2 \to R$. The $R$-module $S_2/IS_2$ may be presented with generators $\conj{b}_0, \conj{b}_1$ and relation $\conj{x}_0\conj{b}_1 + \conj{x}_1 \conj{b}_0 = 0$. From this presentation, we may compute that $\Hom_R(S_2/IS_2, R/I)$ is cyclic and generated by the map sending $\conj{b}_0 \mapsto \conj{x}_0, \conj{b}_1 \mapsto -\conj{x}_1$. This map lifts to the map $S_2 \to R$ sending $b_0 \mapsto x_0, b_1 \mapsto -x_1$; indeed, the lifted map is simply multiplication by $\frac{x_0}{b_0} = -\frac{x_1}{b_1}$. We conclude that $\Hom_{\mathfrak{X}}(Q^*, S(-1)) \to \Hom_{\mathfrak{X}}(Q^*, \conj{Q}^*_\omega)$ is surjective and that $\Ext_{\mathfrak{X}}^1(Q^*, {Q}^*_\omega) = 0,$ as desired.
\end{proof}

Since we have shown vanishing of each component in the direct sum decomposition, we conclude:
\begin{theorem}\label{thm:ext-vanishing}
    $\Ext_{\mathfrak{X}}^i(\mathfrak{W}_{\on{tf}}, \mathfrak{W}_{\on{tf}}) = 0$ for $i > 0$. 
\end{theorem}

\section{Cohomological considerations}
\subsection{Hecke modules and Poincar\'e duality}

We first record some results relating Hecke modules to group cohomology that strengthen \cite[Lemma 5.22]{paskunas13} since they do not seem readily accessible in the literature, although we shall only require Pa\v{s}k\={u}nas' result in the arguments that follow in the next subsection.  

Let $G$ be a group with open profinite subgroup $H$ such that $gHg^{-1}$ has finite index in $H$ for all $g \in G$. We write $\mathcal{H} := C_c^\infty(H\backslash G/H)$ for the Hecke algebra  with coefficients in a field $k$. Given a smooth left $G$-module $\pi$ over $k$, we endow $\mathcal{I}(\pi) := \pi^H$ with the structure of a right $\mathcal{H}$-module by the formula
\begin{align*}
    v T_g := \sum_{u \in H/(H \cap g^{-1}Hg)} ug^{-1}v,
\end{align*}
where $v \in \pi^H$ and $T_g \in \mathcal{H}$ is the function sending the double coset $HgH$ to $1$ and all other double cosets to $0$. Then $\mathcal{I}$ defines a left-exact functor from the category of smooth $G$-modules (which has enough injectives) to the category of Hecke modules, so we may take its right derived functors $\mathbb{R}^i \mathcal{I}$. 

The Hecke action on $\mathbb{R}^i\mathcal{I}(\pi)$ may be described concretely using group cohomology. 
\begin{proposition}\label{prop:group-cohom-Hecke}
   Suppose that restriction to $H$ preserves injectives. For all $i \geq 0$, there is a natural isomorphism of Hecke modules $H^i(H, \pi) \simeq \mathbb{R}^i\mathcal{I}(\pi)$, where the Hecke action of $T_g \in \mathcal{H}$ on $H^i(H, \pi)$ is given by the composition
    \begin{center}
    \begin{tikzcd}
        H^i(H, \pi) \ar[r, "\on{res}_g"] &H^i(H \cap gH g^{-1}, \pi) \ar[r, "c_g"] & H^i(g^{-1}H g \cap H, \pi) \ar[r, "\on{cor}_g"] &H^i(H, \pi)
    \end{tikzcd}
    \end{center}
    where the first map is restriction, the second is induced by conjugation,\footnote{Meaning that this map is induced by functoriality by sending $v \in \pi$ to $g^{-1} \cdot v$, so that $g^{-1}hg$ acts by sending $g^{-1}v$ to $g^{-1}(hv)$.} and the third is corestriction.
\end{proposition}
\begin{proof}
   Take an injective resolution 
    \begin{align*}
        0 \to \pi \to J_0 \to J_1 \to \dots
    \end{align*}
    in the category of smooth $G$-modules. Since restriction to $H$ preserves injectives, this is also an injective resolution in the category of $H$-modules, so by taking $H$-invariants we obtain an isomorphism of $k$-modules
    \begin{align*}
        \mathbb{R}^i\mathcal{I}(\pi) \simeq H^i(H, \pi).
    \end{align*}
    Moreover, restriction from $H$ to a finite-index subgroup also preserves injectives by Frobenius reciprocity, so the same injective resolution holds when considering $\pi$ as a $(gHg^{-1} \cap H)$-module. Therefore, the corestriction map $H^i(gHg^{-1} \cap H, \pi) \to H^i(H, \pi)$ is induced at the level of $J_i^{gHg^{-1} \cap H}$ by
    \begin{align*}
        \on{cor}(v) = \sum_{u \in H/(gHg^{-1} \cap H)} uv, 
    \end{align*}
    and the desired composition at the level of the injective complex is given by
    \begin{align*}
        v \mapsto \sum_{u \in H/(gHg^{-1} \cap H)} u g^{-1}v
    \end{align*}
    which matches the Hecke action on $\mathcal{I}(J^i)$. Hence the Hecke action on cohomology agrees with previously defined Hecke action on $\mathbb{R}^i\mathcal{I}(\pi)$. 
\end{proof}

Framing the Hecke action in this way allows us to show a certain compatibility of the Hecke action with Poincar\'e duality.
\begin{proposition}\label{prop:poincare-duality}
    Assume that $H$ is a Poincar\'e pro-$p$ group of dimension $d$ and that $k$ is a finite field of characteristic $p$. Let $\pi$ be a finite $G$-module. Then cup product induces a perfect pairing of Hecke modules
    \begin{align*}
        H^i(H, \pi) \times H^{d - i}(H, \pi^\vee) \to H^d(H, k) \simeq k
    \end{align*}
    inducing an isomorphism $H^i(H, \pi) \simeq H^{d - i}(H, \pi)^{\vee, \iota}$, where $(-)^\vee$ denotes the $k$-linear dual and $\iota$ denotes the action of the Hecke algebra after composing with the standard anti-involution $\iota: T_g \mapsto T_{g^{-1}}$. 
\end{proposition}
\begin{proof}
    The fact that cup product induces a perfect pairing of $k$-vector spaces follows from \cite[Proposition 30]{serre}. Restriction and corestriction are adjoint operators under the cup product pairing; in the notation of Proposition \ref{prop:group-cohom-Hecke}, this means that $\on{cor}_g(x) \cup y = x \cup \on{res}_{g^{-1}}(x)$ and $\on{res}_g(x) \cup y = x \cup \on{cor}_{g^{-1}}(x)$. Additionally, $c_{g}$ is adjoint to $c_{g^{-1}}$. Hence:
    \begin{align*}
        xT_g \cup y &= (\on{cor}_g \circ c_{g} \circ \on{res}_g x) \cup y \\
        &= x \cup (\on{cor}_{g^{-1}} \circ c_{g^{-1}} \circ \on{res}_{g^{-1}} y)\\
        &= x \cup yT_{g^{-1}}
    \end{align*}
    so that the operators $T_g$ and $T_{g^{-1}}$ are adjoints with respect to the cup product pairing. 
\end{proof}

These results can be used to give a more direct proof of, for example, \cite[Proposition 10.1]{paskunas13} in the $p \geq 5$ case. Restriction to $I_1/Z$ preserves injectives by the proof of \cite[Lemma 5.22]{paskunas13}, and for $p\geq 5$, the pro-$p$ Iwahori Hecke algebra $I_1/Z_1$ mod center is a Poincar\'e pro-$p$ group of dimension $3$. Knowing already that $\mathbb{R}^1\mathcal{I}(1) \simeq \mathcal{I}(\Ind_{P}^G \alpha)$ by \cite[Proposition 11.2]{paskunas_supersingular}, we conclude $$\mathbb{R}^2\mathcal{I}(1) \simeq \mathcal{I}(\Ind_{P}^G \alpha)^{\vee, \iota} \simeq \mathcal{I}(\Ind_{P}^G \alpha),$$ and $$\mathbb{R}^3\mathcal{I}(1) \simeq \mathcal{I}(1)^{\vee, \iota} \simeq \mathcal{I}(1),$$ using the explicit presentations of these Hecke modules as in \cite[\S9]{paskunas_supersingular} to show that they are self-dual. 

\subsection{Infinite projective dimension in the $p = 3$ Steinberg block}

Unfortunately, the pro-$3$ Iwahori Hecke algebra 
\begin{align*}
    I_1 := \left\{\begin{pmatrix}
        1 + 3 \BZ_3 & \BZ_3 \\
        3\BZ_3 & 1 + 3\BZ_3
    \end{pmatrix}\right\} \subset \GL_2(\BQ_3)
\end{align*}
is not a Poincar\'e pro-$3$ group, nor is $I_1/Z_1$, where $Z_1 = Z\cap I_1$, since these groups have nontrivial $3$-torsion. For example, 
\begin{align*}
    \begin{bmatrix}
        -2 & -1 \\
        3 & 1
    \end{bmatrix}^3 = \begin{bmatrix}
        1 & 0 \\ 0 & 1
    \end{bmatrix}.
\end{align*}
Thus Proposition \ref{prop:poincare-duality} is not applicable to help compute the Hecke modules $\mathbb{R}^i\mathcal{I}(1)$. 

Moreover, we claim that these Hecke modules are nonzero for all $i \geq 0$. By Proposition \ref{prop:group-cohom-Hecke}, it suffices to show that $H^i(I_1/Z_1, k) \neq 0$ for all $i \geq 0$, with the hypotheses of this proposition again following from the proof of \cite[Lemma 5.22]{paskunas13}. We have $H^i(\BZ/3, k) = k$ for all $i \geq 0$, since the standard $2$-periodic complex computing these cohomology groups is
\begin{center}
    \begin{tikzcd}
        0 \ar[r] & k \ar[r, "1 - \sigma"] & k \ar[r, "1 + \sigma + \sigma^2"] & k \ar[r, "1 - \sigma"] & k \ar[r, "1 + \sigma + \sigma^2"] & \dotsb
    \end{tikzcd}
\end{center}
where $\sigma$ is a generator of $\Z/3$, but these maps are all $0$. Since $I_1/Z_1$ has $3$-torsion, this group contains $\Z/3$ as a closed subgroup, so by \cite[Proposition 3.3.5]{cohomology-of-number-fields}, the $p$-cohomological dimension of $I_1/Z_1$ is also infinity, so, by \cite[Proposition 3.3.2]{cohomology-of-number-fields}, we conclude that $H^i(I_1/Z_1, k) \neq 0$ for all $i \geq 0$.

From \cite[Proposition 5.26]{paskunas13}, for $\pi, \tau \in \on{Mod}_{G/Z}^{\on{lfin}}(k)$, we have an exact sequence of $k$-vector spaces
\begin{align*}
    0 \to \Ext^1_{\mathcal{H}}(\mathcal{I}(\pi), \mathbb{R}^{i - 1}\mathcal{I}(\tau)) \to \Ext^i_{G/Z}(\pi, \tau) \to \Hom_{\mathcal{H}}(\mathcal{I}(\pi), \mathbb{R}^i\mathcal{I}(\tau)) \to 0.
\end{align*}
Hence non-vanishing of $\mathbb{R}^i\mathcal{I}(1)$ implies non-vanishing of $\Ext^i_{G/Z}(\pi, 1)$ for some irreducible $\pi$, since we may choose $\tau$ such that $\mathcal{I}(\pi)$ is a submodule of $\mathbb{R}^i\mathcal{I}(1)$. We conclude:
\begin{proposition}\label{prop:infinite-cohom-dim}
    The simple right module over $\tilde{E}_{\mathfrak{B}}$ corresponding to $1^\vee \in \mathfrak{C}(\O)$ under Morita theory has infinite projective dimension. In particular, $\tilde{E}_{\mathfrak{B}}$ has infinite global dimension. 
\end{proposition}
\begin{proof}
    Taking the Pontryagin dual of the previous discussion, we have shown that $\Ext_{\mathfrak{C}(k)}^i(1^\vee, -) \neq 0$ for infinitely many $i$. By \cite[Corollary 5.17]{paskunas13}, this implies that the same Ext groups computed in $\mathfrak{C}(\O)$ are also nonzero. Hence, no finite projective resolution of $1^\vee$ exists in $\mathfrak{C}(\O)$. Finally, under Morita theory, $\mathfrak{C}(\O)^{\mathfrak{B}}$ is equivalent to the category of compact right modules over $\tilde{E}_{\mathfrak{B}}$. 
\end{proof}
\subsection{Non-Steinberg blocks}\label{sec:non-steinberg}
Moving towards formulating a categorical $p$-adic Langlands correspondence, we briefly explain how previously existing results cover the non-Steinberg $3$-adic cases.

The results of \cite[\S8-9]{paskunas13} remain valid for $p = 3$, as explicitly stated at the beginning of these sections. All computations in \cite[\S3.2-3.3, \S5.3-5.4]{JNW}, including the tangent and obstruction computations and the presentations of $S$ and $R$, are unaffected by taking $p = 3$. We conclude that their proofs for the generic and non-generic of type I (i.e. scalar) cases continue to hold for $p = 3$. Alternatively, for the generic case we could utilize the new arguments developed in \S4.

In the supersingular case, the discussion of \cite[\S2.3.1]{JNW} remains valid, but, as mentioned in Remark 3.1.2 of loc.\ cit., for $p = 3$ there exists a $2$-dimensional irreducible residual representation $\conj{\rho}$ such that $R_{\conj{\rho}}^{\square}$ is singular. In fact, there is exactly one such residual representation up to isomorphism: $\conj{\rho} = \Ind_{\BQ_9}^{\BQ_3} \omega_2^2$, where $\omega_2$ is the niveau $2$ fundamental character. This is the unique $2$-dimensional irreducible mod $3$ representation of $\Gamma$ satisfying $\conj{\rho} \otimes \omega \simeq \conj{\rho}$, since all such representations are of the form $\Ind_{\BQ_9}^{\BQ_3} \omega_2^i \simeq \Ind_{\BQ_9}^{\BQ_3} \omega_2^{3i}$ for some $i$, and the condition for this being isomorphic to its own $\omega$-twist is that $i + 4 \equiv 3i \bmod 8$ since $\omega = \omega_2^4$. Consequently, for this $\conj{\rho}$ we have $$H^2(\Gamma, \on{Ad}^0 \conj{\rho}) \simeq H^0(\Gamma, (\on{Ad}^0 \conj{\rho}) \otimes \omega) \simeq \Hom_{\Gamma}(\conj{\rho}, \conj{\rho} \otimes \omega) \neq 0,$$
so the universal deformation ring (equivalently, the universal pseudodeformation ring) is singular. However, by \cite[Theorem 1.2]{paskunas16}, we still have $\tilde{E}_{\mathfrak{B}} \simeq R \simeq R_{\conj{\rho}}$, so we may still define a fully faithful functor $\on{Mod}_{G/Z}^{\on{lfin}}(\O)^{\mathfrak{B}} \to \on{QCoh}(\mathfrak{X}_{\mathfrak{B}})$ at the level of abelian categories by sending
\begin{align*}
    V \rightsquigarrow R(1) \otimes_{R} V
\end{align*}
as in \cite[\S 5.2]{JNW}. 

\subsection{Local cohomology}

Due to the infinite projective dimension discussed in \S5.2, \cite[Assumption 4.2.1(c)]{JNW} does not hold. This causes the proof of Lemma 4.2.8 in loc.\ cit.\ to fail. Despite this, the lemma still holds; in fact it is still true that the functor $M \rightsquigarrow M[\mathfrak{m}^\infty]$ has bounded cohomological dimension even in our setting. 

We work in some generality for this subsection. Let $R$ be a commutative noetherian ring and let $E$ be an (associative, unital) left noetherian $R$-algebra with $R$ acting centrally. For an ideal $I$ of $R$, let $\Gamma_I: \on{LMod}(E) \to \on{LMod}(E)$ denote the \emph{local cohomology functor} sending a left $E$-module $M$ to the submodule $M[I^\infty]$ consisting of elements annihilated by some power of $I$. 

Suppose $I = (x_1, \dots, x_d)$ with $x_i \in R$. For an $E$-module $M$, let $\check{C}^\bullet(\mathbf{x}, M)$ denote the \v{C}ech complex associated to this set of generators, treating $M$ as an $R$-module; see \cite[Tag 0G6F]{stacks-project} for a precise definition. Since the elements $x_i$ act centrally, the localizations $M_{x_{i_0} \dotsb x_{i_j}}$ involved in the complex are naturally well-defined left $E$-modules, the differentials in $\check{C}^\bullet(\mathbf{x}, M)$ are $E$-module homomorphisms, and $\check{C}^\bullet(\mathbf{x}, M)$ is indeed a complex by the standard argument. 

\begin{lemma}
    The cohomology groups $\mathbb{R}^i\Gamma_I(M)$ are computed by the cohomology of $\check{C}^\bullet(\mathbf{x}, M)$. 
\end{lemma}
\begin{proof}
    This is a standard result when the base ring is commutative; see \cite[Theorem 5.1.20]{local-cohomology}. But as long as the $x_i$ act centrally, the proof in loc.\ cit.\ also works over a noncommutative base $E$. This proof demonstrates effaceability the functors $H^i(\check{C}^\bullet(\mathbf{x}, M))$ for $i > 0$, so that these agree with the universal $\delta$-functor $\mathbb{R}^i\Gamma_I$. 
    
    The main fact meriting discussion in this argument is that if $J$ is an injective left $E$-module and $x$ is central in $E$, then the localization $J_{x}$ is also an injective left $E$-module. Centrality of $x$ ensures that this localization produces a well-defined left $E$-module. Let $\mathfrak{a}$ be a left ideal of $E$. Since $E$ is left noetherian, $\mathfrak{a}$ is finitely generated, so that $\Hom_E(\mathfrak{a}, J_{x}) \simeq \Hom_E(\mathfrak{a}, J)_{x}$. Since $J$ is injective, we have a short exact sequence
    \begin{align*}
        0 \to \Hom_E(E/E\mathfrak{a}, J) \to \Hom_E(E, J) \to \Hom_E(\mathfrak{a}, J) \to 0,
    \end{align*}
    and since localization (by a central element) is exact we conclude that we also have an exact sequence
    \begin{align*}
        0 \to \Hom_E(E/E\mathfrak{a}, J)_{x} \to \Hom_E(E, J)_{x} \to \Hom_E(\mathfrak{a}, J)_{x} \to 0,
    \end{align*}
    and hence we have a short exact sequence
    \begin{align*}
        0 \to \Hom_E(E/E\mathfrak{a}, J_x) \to \Hom_E(E, J_x) \to \Hom_E(\mathfrak{a}, J_x) \to 0
    \end{align*}
    by the isomorphism discussed previously. Then injectivity of $J_x$ follows by Baer's criterion, since we have shown that any $E$-module homomorphism $\mathfrak{a} \to J_x$ extends to $E$. 
\end{proof}

\begin{corollary}\label{cor:bounded-local-cohom}
    If $I$ is generated by $d$ elements, then $\mathbb{R}^i\Gamma_I = 0$ for $i > d$. 
\end{corollary}
\begin{proof}
    This follows immediately from the previous lemma and the fact that $\check{C}^i(\mathbf{x}, M) = 0$ for $i > d$. 
\end{proof}
\subsection{Fully faithful embeddings of derived categories}
Let $p = 3$ and let $\mathfrak{B}$ be the Steinberg block. We reuse the notation of \S1.2 and introduce additional notation. For an abelian category $\mathcal{C}$, we write $\mathcal{D}^b(\mathcal{C})$ for the bounded derived $\infty$-category of $\mathcal{C}$. For an compact associative algebra $E$, we write $\on{LMod}(E)$ for the category of left $E$-modules, and we write $\on{LMod}_{\on{disc}}(E)$ for the category of discrete left $E$-modules, which are both abelian categories. We let $\mideal$ denote the maximal ideal of the pseudodeformation ring $R$, and we let $\on{IndCoh}_{\mathfrak{m}}(\mathfrak{X}_{\mathfrak{B}})$ denote the full subcategory of $\on{IndCoh}(\mathfrak{X})$ consisting of objects with $\mideal^\infty$-torsion cohomology groups, meaning that each element $f$ in each cohomology group is annihilated by $\mideal^n$ for some $n$ depending on $f$. We write $\on{Perf}(E)$ for the dg nerve of the category of finitely generated bounded complexes of projective left $E$-modules. The $\infty$-category $\on{Perf}(E)$ is characterized as the full subcategory of $\mathcal{D}(\on{LMod}(E))$ consisting of compact objects, and moreover $\on{Perf}(E)$ generates $\mathcal{D}(E)$, so that $\Ind(\on{Perf}(E)) = \mathcal{D}(\on{LMod}(E))$. These statements remain true independent of the finitude of the projective dimension of left $E$-modules. 

\begin{theorem}\label{thm:fully-faithful-perf}
    The sheaf $\mathfrak{W}_{\on{tf}}$ defined in \S\ref{sec:candidate} induces a fully faithful functor 
    \begin{align*}
        \on{Perf}(\tilde{E}_{\mathfrak{B}}) &\to \mathcal{D}^b(\on{Coh}(\mathfrak{X}_{\mathfrak{B}})) \\
        P_\bullet &\rightsquigarrow \mathfrak{W}_{\on{tf}} \otimes^L_{\tilde{E}_{\mathfrak{B}}} P_\bullet.
    \end{align*}
    Taking ind-completions induces fully faithful functors
    \begin{align*}
        \mathcal{D}(\on{LMod}(\tilde{E}_{\mathfrak{B}})) &\to \on{IndCoh}(\mathfrak{X}_{\mathfrak{B}}) \\
        \mathcal{D}(\on{LMod}_{\on{disc}}(\tilde{E}_{\mathfrak{B}})) &\to \on{IndCoh}_{\mideal}(\mathfrak{X}_{\mathfrak{B}}).
    \end{align*}
\end{theorem}
\begin{proof}
    The arguments of \cite{JNW}, 4.2.2-4.2.8, 4.2.10-4.2.11 go through despite the failure of Assumption 4.2.1(c), since we can use the bounded cohomological dimension from Corollary \ref{cor:bounded-local-cohom} to complete the proof of Lemma 4.2.8 in loc.\ cit. 
    
    We sketch the overall argument. The right action of $\tilde{E}_{\mathfrak{B}}$ on $\mathfrak{W}_{\on{tf}}$ defining the tensor product comes from Theorem \ref{thm:endomorphism-algebra-isomorphism}, so the claimed derived tensor product is well-defined. To prove full faithfulness of $\on{Perf}(\tilde{E}_{\mathfrak{B}}) \to \mathcal{D}^b(\on{Coh}(\mathfrak{X}_{\mathfrak{B}}))$, it suffices to prove this claim at the level of triangulated (homotopy) categories (see \cite[\S2.4]{JNW} for discussion), so we need to show that we have a natural isomorphism $\Hom(V_\bullet, V'_\bullet) \simeq \Hom(\mathfrak{W}_{\on{tf}} \otimes_{\tilde{E}_{\mathfrak{B}}}^L V_\bullet, \mathfrak{W}_{\on{tf}} \otimes_{\tilde{E}_{\mathfrak{B}}}^L V'_\bullet )$ for any bounded complexes of finitely generated projective left $\tilde{E}_{\mathfrak{B}}$-modules $V_\bullet, V'_\bullet$. By d\'evissage, it suffices to prove this when $V_\bullet$ and $V_{\bullet}'$ are isomorphic to $\tilde{E}_{\mathfrak{B}}$ supported in a single homological degree. This is equivalent to showing that 
    $$\Ext_{\mathfrak{X}}^i(\mathfrak{W}_{\on{tf}} \otimes_{\tilde{E}_{\mathfrak{B}}} \tilde{E}_{\mathfrak{B}}, \mathfrak{W}_{\on{tf}} \otimes_{\tilde{E}_{\mathfrak{B}}} \tilde{E}_{\mathfrak{B}}) = \Ext_{\tilde{E}_{\mathfrak{B}}}^i(\tilde{E}_{\mathfrak{B}}, \tilde{E}_{\mathfrak{B}}) = \begin{cases}
        \tilde{E}_{\mathfrak{B}}^{\on{op}} : & i=0 \\
        0 : &i\neq 0
    \end{cases}$$
    for all $i \in \BZ$. But the group on the left hand side is just $\Ext_{\mathfrak{X}}^i(\mathfrak{W}_{\on{tf}}, \mathfrak{W}_{\on{tf}})$. Hence the claim follows for $i = 0$ by Theorem \ref{thm:endomorphism-algebra-isomorphism} and for $i \neq 0$ by Theorem \ref{thm:ext-vanishing}. 

    Passing to ind-completions automatically yields a fully faithful embedding $$\mathcal{D}(\on{LMod}(\tilde{E}_{\mathfrak{B}})) \to \on{IndCoh}(\mathfrak{X}_{\mathfrak{B}}).$$ The natural functor $\mathcal{D}(\on{LMod}_{\on{disc}}(\tilde{E}_{\mathfrak{B}})) \to \mathcal{D}(\on{LMod}(\tilde{E}_{\mathfrak{B}}))$ is fully faithful by the argument in \cite[Lemma 4.2.8]{JNW}, where we use Corollary \ref{cor:bounded-local-cohom} to replace the requirement of finite global dimension of $\tilde{E}_{\mathfrak{B}}$ in the original argument. More precisely, the setup in \cite[Lemma 4.2.8]{JNW} requires showing that the derived functors of $\Gamma_{\mideal}$ have bounded cohomological dimension. In the original argument, this is proven using the finite global dimension assumption, but we have proven this fact directly in Corollary \ref{cor:bounded-local-cohom}. Hence we may naturally view $\mathcal{D}(\on{LMod}_{\on{disc}}(\tilde{E}_{\mathfrak{B}}))$ as a subcategory of $\mathcal{D}(\on{LMod}(\tilde{E}_{\mathfrak{B}}))$.
    
    Finally, the argument in \cite[Proposition 4.2.11]{JNW}, using the hypertor spectral sequence, shows that the essential image of $\mathcal{D}(\on{LMod}_{\on{disc}}(\tilde{E}_{\mathfrak{B}}))$ is contained in $\on{IndCoh}_{\mideal}(\mathfrak{X}_{\mathfrak{B}})$, so we conclude that the embedding $\mathcal{D}(\on{LMod}(\tilde{E}_{\mathfrak{B}})) \to \on{IndCoh}(\mathfrak{X}_{\mathfrak{B}})$ restricts to a fully faithful embedding $\mathcal{D}(\on{LMod}_{\on{disc}}(\tilde{E}_{\mathfrak{B}})) \to \on{IndCoh}_{\mathfrak{m}}(\mathfrak{X}_{\mathfrak{B}})$. 
\end{proof}
\begin{remark} In the $p \geq 5$ setting, \cite[Proposition 5.5.3]{JNW} implies that these fully faithful embeddings of derived categories have cohomological amplitude $[0, 1]$, with the only truly derived behavior arising from the trivial representation. Lacking a similar description of the image of the irreducibles under these functors, it is unclear what the cohomological amplitude is in our $p = 3$ case. Due to the infinite global dimension of $\tilde{E}_{\mathfrak{B}}$, it is possible that the cohomological amplitude is unbounded, but it is also possible that the cohomological amplitude may again be $[0, 1]$ or at least something small. Strictly speaking, we have not yet ruled out that our functors are $t$-exact for the $p = 3$ Steinberg block, although this would be very surprising. 
\end{remark}

\begin{remark}
    We note that $\on{IndCoh}(\mathfrak{X}_{\mathfrak{B}})$ is not the same as $\mathcal{D}(\on{QCoh}(\mathfrak{X}_{\mathfrak{B}}))$. Presenting $\mathfrak{X}_{\mathfrak{B}} = [\Spec S/\mathbb{G}_m]$, the ring $S$ is not regular. Take a graded $S$-module of infinite projective dimension. We may consider $M$ as a complex, concentrated in degree $0$, and we compare this complex to a projective resolution of $M$. Working in $\on{IndCoh}(\mathfrak{X}_{\mathfrak{B}})$, the images of these two complexes are not isomorphic. However, these two complexes are identified in $\mathcal{D}(\on{QCoh}(\mathfrak{X}_{\mathfrak{B}}))$, so the natural functor $\on{IndCoh}(\mathfrak{X}_{\mathfrak{B}}) \to \mathcal{D}(\on{QCoh}(\mathfrak{X}_{\mathfrak{B}}))$ is not fully faithful. 

    This distinction was already relevant in \cite{JNW}, since the moduli stack of Galois representations with residual representation $\omega \oplus 1$ is singular even when $p \geq 5$. 
\end{remark}

\begin{corollary}\label{cor:categorical-p-LLC}
    The functor 
    $$C_\bullet \rightsquigarrow \mathfrak{W}_{\on{tf}} \otimes_{\tilde{E}_{\mathfrak{B}}}^L C_\bullet$$
    defines a fully faithful embedding 
    $$\mathcal{D}(\on{Mod}_{G/Z}^{\on{lfin}}(\O))^{\mathfrak{B}} \hookrightarrow \on{IndCoh}_\mideal(\mathfrak{X}_\mathfrak{B}) \hookrightarrow \on{IndCoh}(\mathfrak{X}_\mathfrak{B}).$$
\end{corollary}
\begin{proof}
     Morita theory yields an equivalence $\on{LMod}_{\on{disc}}(\tilde{E}_{\mathfrak{B}}) \simeq \on{Mod}_{G/Z}^{\on{lfin}}(\O)^{\mathfrak{B}}$, so the claim follows by Theorem \ref{thm:fully-faithful-perf}. 
\end{proof}

\printbibliography

@article {JNW,
    AUTHOR = {Johansson, Christian and Newton, James and Wang-Erickson,
              Carl},
     TITLE = {Moduli stacks of {G}alois representations and the {$p$}-adic
              local {L}anglands correspondence for {${\rm
              GL}_2(\Bbb{Q}_{p})$}},
   JOURNAL = {Compos. Math.},
  FJOURNAL = {Compositio Mathematica},
    VOLUME = {162},
      YEAR = {2026},
    NUMBER = {2},
     PAGES = {368--450},
      %ISSN = {0010-437X,1570-5846},
   MRCLASS = {22E50 (11F80 11F85 11S37)},
  MRNUMBER = {5088238},
       %DOI = {10.1017/S0010437X2610298X},
       %URL = {https://doi-org.libproxy.berkeley.edu/10.1017/S0010437X2610298X},
}

@article {paskunas16,
    AUTHOR = {Pa\v sk\=unas, Vytautas},
     TITLE = {On 2-dimensional 2-adic {G}alois representations of local and
              global fields},
   JOURNAL = {Algebra Number Theory},
  FJOURNAL = {Algebra \& Number Theory},
    VOLUME = {10},
      YEAR = {2016},
    NUMBER = {6},
     PAGES = {1301--1358},
      %ISSN = {1937-0652,1944-7833},
   MRCLASS = {11F80},
  MRNUMBER = {3544298},
MRREVIEWER = {Bao\ Viet\ Le Hung},
       %DOI = {10.2140/ant.2016.10.1301},
       %URL = {https://doi-org.libproxy.berkeley.edu/10.2140/ant.2016.10.1301},
}

@article {paskunas13,
    AUTHOR = {Pa\v sk\=unas, Vytautas},
     TITLE = {The image of {C}olmez's {M}ontreal functor},
   JOURNAL = {Publ. Math. Inst. Hautes \'Etudes Sci.},
  FJOURNAL = {Publications Math\'ematiques. Institut de Hautes \'Etudes
              Scientifiques},
    VOLUME = {118},
      YEAR = {2013},
     PAGES = {1--191},
      %ISSN = {0073-8301,1618-1913},
   MRCLASS = {22E35 (11S37 22E50)},
  MRNUMBER = {3150248},
MRREVIEWER = {Ivan\ Mati\'c},
       %DOI = {10.1007/s10240-013-0049-y},
       %URL = {https://doi-org.libproxy.berkeley.edu/10.1007/s10240-013-0049-y},
}

@article {paskunas_supersingular,
    AUTHOR = {Pa\v sk\=unas, Vytautas},
     TITLE = {Extensions for supersingular representations of {${\rm
              GL}_2(\Bbb Q_p)$}},
   JOURNAL = {Ast\'erisque},
  FJOURNAL = {Ast\'erisque},
    NUMBER = {331},
      YEAR = {2010},
     PAGES = {317--353},
      %ISSN = {0303-1179,2492-5926},
      ISBN = {978-2-85629-282-2},
   MRCLASS = {22E50 (11F70)},
  MRNUMBER = {2667891},
MRREVIEWER = {Gabor\ Wiese},
}

@misc{DEG26,
      title={A categorical p-adic Langlands correspondence for $\GL_2(\BQ_p)$}, 
      author={Andrea Dotto and Matthew Emerton and Toby Gee},
      year={2026},
      eprint={2603.26887},
      archivePrefix={arXiv},
      primaryClass={math.NT},
      %url={https://arxiv.org/abs/2603.26887}, 
}

@article {paskunas-tung,
    AUTHOR = {Pa\v sk\=unas, Vytautas and Tung, Shen-Ning},
     TITLE = {Finiteness properties of the category of {${\rm mod}\,p$}
              representations of {${\rm GL}_2 (\Bbb Q_p)$}},
   JOURNAL = {Forum Math. Sigma},
  FJOURNAL = {Forum of Mathematics. Sigma},
    VOLUME = {9},
      YEAR = {2021},
     PAGES = {Paper No. e80, 39},
      %ISSN = {2050-5094},
   MRCLASS = {22E50 (11F80)},
  MRNUMBER = {4350140},
MRREVIEWER = {Ivan\ Mati\'c},
       %DOI = {10.1017/fms.2021.72},
       %URL = {https://doi-org.libproxy.berkeley.edu/10.1017/fms.2021.72},
}

@article {bockle-3adic,
    AUTHOR = {B\"ockle, Gebhard},
     TITLE = {Deformation rings for some mod 3 {G}alois representations of
              the absolute {G}alois group of {$\bold Q_3$}},
   JOURNAL = {Ast\'erisque},
  FJOURNAL = {Ast\'erisque},
    NUMBER = {330},
      YEAR = {2010},
     PAGES = {529--542},
      %ISSN = {0303-1179,2492-5926},
      ISBN = {978-2-85629-281-5},
   MRCLASS = {11F80 (11F85 11S37 20F05)},
  MRNUMBER = {2642411},
MRREVIEWER = {Kentaro\ Nakamura},
}

@incollection {chenevier-pseudoreps,
    AUTHOR = {Chenevier, Ga\"etan},
     TITLE = {The {$p$}-adic analytic space of pseudocharacters of a
              profinite group and pseudorepresentations over arbitrary
              rings},
 BOOKTITLE = {Automorphic forms and {G}alois representations. {V}ol. 1},
    SERIES = {London Math. Soc. Lecture Note Ser.},
    VOLUME = {414},
     PAGES = {221--285},
 PUBLISHER = {Cambridge Univ. Press, Cambridge},
      YEAR = {2014},
      %ISBN = {978-1-107-69192-6},
   MRCLASS = {11F70 (11F80 14G22 20E18)},
  MRNUMBER = {3444227},
MRREVIEWER = {Cameron\ Franc},
}

@article{CWE18,
   title={Algebraic families of Galois representations and potentially semi-stable pseudodeformation rings},
   volume={371},
   %ISSN={1432-1807},
   %url={http://dx.doi.org/10.1007/s00208-017-1557-8},
   %DOI={10.1007/s00208-017-1557-8},
   number={3-4},
   journal={Mathematische Annalen},
   publisher={Springer Science and Business Media LLC},
   author={Wang-Erickson, Carl},
   year={2017},
   month=July, pages={1615–1681},
   shorthand = {WE17}
   }

@article {BC09,
    AUTHOR = {Bella\"iche, Jo\"el and Chenevier, Ga\"etan},
     TITLE = {Families of {G}alois representations and {S}elmer groups},
   JOURNAL = {Ast\'erisque},
  FJOURNAL = {Ast\'erisque},
    NUMBER = {324},
      YEAR = {2009},
     PAGES = {xii+314},
      %ISSN = {0303-1179,2492-5926},
      ISBN = {978-2-85629-264-8},
   MRCLASS = {11F80 (11F70 11F85 11G40 11S25 14F30 14G22)},
  MRNUMBER = {2656025},
MRREVIEWER = {Adolfo\ Quir\'os},
}

@article{bip,
   title={On local Galois deformation rings},
   volume={11},
   %ISSN={2050-5086},
   %url={http://dx.doi.org/10.1017/fmp.2023.25},
   %DOI={10.1017/fmp.2023.25},
   journal={Forum of Mathematics, Pi},
   publisher={Cambridge University Press (CUP)},
   author={Böckle, Gebhard and Iyengar, Ashwin and Paškūnas, Vytautas},
   year={2023} }

@book {serre,
    AUTHOR = {Serre, Jean-Pierre},
     TITLE = {Cohomologie galoisienne},
    SERIES = {Lecture Notes in Mathematics},
    VOLUME = {5},
   EDITION = {Fifth},
 PUBLISHER = {Springer-Verlag, Berlin},
      YEAR = {1994},
     PAGES = {x+181},
      %ISBN = {3-540-58002-6},
   MRCLASS = {12G05 (11R34)},
  MRNUMBER = {1324577},
       %DOI = {10.1007/BFb0108758},
       %URL = {https://doi-org.libproxy.berkeley.edu/10.1007/BFb0108758},
}

@article {breuil,
    AUTHOR = {Breuil, Christophe},
     TITLE = {Sur quelques repr\'esentations modulaires et {$p$}-adiques de
              {${\rm GL}_2(\bold Q_p)$}. {I}},
   JOURNAL = {Compositio Math.},
  FJOURNAL = {Compositio Mathematica},
    VOLUME = {138},
      YEAR = {2003},
    NUMBER = {2},
     PAGES = {165--188},
      %ISSN = {0010-437X,1570-5846},
   MRCLASS = {11F33 (11F80 11F85)},
  MRNUMBER = {2018825},
MRREVIEWER = {David\ L.\ Savitt},
       %DOI = {10.1023/A:1026191928449},
       %URL = {https://doi-org.libproxy.berkeley.edu/10.1023/A:1026191928449},
}

@misc{stacks-project,
  author       = {The {Stacks project authors}},
  title        = {The Stacks project},
  howpublished = {\url{https://stacks.math.columbia.edu}},
  year         = {2026},
  shorthand = {Sta26}
}

@article {barthel-livne,
    AUTHOR = {Barthel, Laure and Livn\'e, Ron},
     TITLE = {Irreducible modular representations of {${\rm GL}_2$} of a
              local field},
   JOURNAL = {Duke Math. J.},
  FJOURNAL = {Duke Mathematical Journal},
    VOLUME = {75},
      YEAR = {1994},
    NUMBER = {2},
     PAGES = {261--292},
      %ISSN = {0012-7094,1547-7398},
   MRCLASS = {22E50 (11F70)},
  MRNUMBER = {1290194},
MRREVIEWER = {David\ Goldberg},
       %DOI = {10.1215/S0012-7094-94-07508-X},
       %URL = {https://doi.org/10.1215/S0012-7094-94-07508-X},
}

@incollection {EGH,
    AUTHOR = {Emerton, Matthew and Gee, Toby and Hellmann, Eugen},
     TITLE = {An introduction to the categorical {$p$}-adic {L}anglands
              program},
 BOOKTITLE = {The {L}anglands program. {II}. {G}eometrization of the
              {L}anglands correspondence},
    SERIES = {Proc. Sympos. Pure Math.},
    VOLUME = {112.2},
     PAGES = {167--419},
 PUBLISHER = {Amer. Math. Soc., Providence, RI},
      YEAR = {2025},
      %ISBN = {978-1-4704-7438-6},
   MRCLASS = {11R39},
  MRNUMBER = {5007752},
       %DOI = {10.1090/pspum/112.2/02066},
       %URL = {https://doi-org.libproxy.berkeley.edu/10.1090/pspum/112.2/02066},
}

@book {cohomology-of-number-fields,
    AUTHOR = {Neukirch, J\"urgen and Schmidt, Alexander and Wingberg, Kay},
     TITLE = {Cohomology of number fields},
    SERIES = {Grundlehren der mathematischen Wissenschaften [Fundamental
              Principles of Mathematical Sciences]},
    VOLUME = {323},
   EDITION = {Second},
 PUBLISHER = {Springer-Verlag, Berlin},
      YEAR = {2008},
     PAGES = {xvi+825},
      %ISBN = {978-3-540-37888-4},
   MRCLASS = {11R34 (11-02 11G45 11R23 11S20 11S25 11S31 12G05)},
  MRNUMBER = {2392026},
       %DOI = {10.1007/978-3-540-37889-1},
       %URL = {https://doi-org.libproxy.berkeley.edu/10.1007/978-3-540-37889-1},
}

@book {local-cohomology,
    AUTHOR = {Brodmann, M. P. and Sharp, R. Y.},
     TITLE = {Local cohomology},
    SERIES = {Cambridge Studies in Advanced Mathematics},
    VOLUME = {136},
   EDITION = {Second},
      NOTE = {An algebraic introduction with geometric applications},
 PUBLISHER = {Cambridge University Press, Cambridge},
      YEAR = {2013},
     PAGES = {xxii+491},
      ISBN = {978-0-521-51363-0},
   MRCLASS = {13D45 (13-01)},
  MRNUMBER = {3014449},
MRREVIEWER = {Alberto\ F.\ Boix},
}

@article {emerton-ord-ii,
    AUTHOR = {Emerton, Matthew},
     TITLE = {Ordinary parts of admissible representations of {$p$}-adic
              reductive groups {II}. {D}erived functors},
   JOURNAL = {Ast\'erisque},
  FJOURNAL = {Ast\'erisque},
    NUMBER = {331},
      YEAR = {2010},
     PAGES = {403--459},
      %ISSN = {0303-1179,2492-5926},
      %ISBN = {978-2-85629-282-2},
   MRCLASS = {22E50 (20G25)},
  MRNUMBER = {2667883},
MRREVIEWER = {Maarten\ Sander\ Solleveld},
}
\end{document}